\documentclass{amsart}[12pt]
\AtBeginDocument{\addtocontents{toc}{\protect\setlength{\parskip}{0pt}}}

\usepackage{amssymb, enumitem}
\usepackage{comment}
\usepackage{xcolor}
\usepackage{tikz}
\usepackage{hyperref}
\hypersetup{
   colorlinks,
    linkcolor={red!60!black},
    citecolor={blue!60!black}, 
    urlcolor={blue!90!black}
}

\DeclareMathOperator{\dist}{dist}
\DeclareMathOperator{\spt}{spt}
\DeclareMathOperator{\hdim}{\dim_H}
\DeclareMathOperator{\pdim}{\dim_P}
\DeclareMathOperator{\diam}{diam}
\newcommand{\G}{\mathbb{G}}
\newcommand{\X}{\mathbb{X}}
\newcommand{\A}{\mathbb{A}}

\newcommand{\Z}{\mathbb{Z}}
\newcommand{\R}{\mathbb{R}}
\newcommand{\N}{\mathbb{N}}

\newtheorem{theorem}{Theorem}[section]
\newtheorem*{theorem*}{Theorem}
\newtheorem{corollary}[theorem]{Corollary}
\newtheorem{proposition}[theorem]{Proposition}
\newtheorem{lemma}[theorem]{Lemma}
\newtheorem{terminology}[theorem]{Terminology}

\theoremstyle{definition}
\newtheorem{definition}[theorem]{Definition}
\newtheorem{remark}[theorem]{Remark}

\newcommand{\e}{\varepsilon}
\numberwithin{equation}{section}
\numberwithin{figure}{section}

\title[Spread-out Furstenberg sets II: general configurations]{Spread-out Furstenberg sets II: general configurations and fractal restricted projections}
\author{Alex Cohen}
\author{Pablo Shmerkin}
\author{Hong Wang}
\author{Dmitrii Zakharov}

\begin{document}

\begin{abstract}
    We prove optimal dual Furstenberg estimates for Frostman and Katz--Tao configurations for which the shadings are not locally concentrated near pencils. The starting point is the sticky estimate of Orponen and Shmerkin in a companion paper. We deduce restricted projection theorems in the plane and in higher dimensions. In particular, Ahlfors-regular measures on the circle satisfy a Marstrand-type projection theorem. Applications include restricted projections to surfaces in the $\mathbb{G}(4,2)$ Grassmannian, pinned distances, discretized sum-product, and a new proof of the dimension formula for planar self-affine sets.
\end{abstract}

\maketitle

\tableofcontents

\section{Introduction}

This paper and the companion paper of Orponen and Shmerkin \cite{SpreadOutFurstSticky} study dual Furstenberg sets and restricted projections when the families of planes satisfy a local non-concentration condition with respect to pencils at most point and scales, which we call the \emph{spread-out} condition. The companion paper introduces spread-outness and proves a sharp Furstenberg estimate in the ``sticky case'', where the point set is ``almost Ahlfors-David regular''. Our main task is to remove this regularity assumption, which is achieved in full generality in Theorems \ref{thm:furst-general} and \ref{thm:proj-general}.

The two papers form a single project. We present them separately to clarify the exposition and to distinguish the authorship of the sticky estimate from that of the passage to general configurations. We repeat some introductory material from \cite{SpreadOutFurstSticky} to keep this paper self-contained.

\subsection{Restricted projections}

For $V\in\G(d,k)$, let $\Pi_V\colon\R^d\to V$ denote orthogonal projection onto $V$. The Marstrand-Mattila projection theorem \cite{Mar,MR409774} says that almost every projection has the largest possible dimension. More precisely, if $E \subset \R^{d}$ is Borel and $\gamma_{d,k}$ denotes Lebesgue measure on the Grassmannian $\G(d,k)$ of $k$-planes in $\R^d$, then 
\begin{equation} \label{eq:marstrand-mattila}
    \hdim \Pi_{V}(E) = \min\{\hdim E,k\} \quad\text{ for $\gamma_{d,k}$-a.e. } V \in \G(d,k).
\end{equation}
What happens if the planes are chosen from a much smaller family? In projection theory, one studies $\hdim\Pi_{V}(E)$ for $\nu$-typical $V$, where $\nu$ is a highly singular measure on $\G(d,k)$, sometimes with additional assumptions on $E$. See \cite{Wang26RestrictedProjections} for a survey, including connections with the Kakeya problem.

We ask which measures $\nu$ can replace $\gamma_{d,k}$ in \eqref{eq:marstrand-mattila}. Much is known when the measure is length measure on a smooth curve in $\G(d,k)$; see, for example, \cite{FasslerOrponen14,KaenmakiOrponenVenieri25,PYZ22,GanGuthMaldague24} for $\G(3,1)$, \cite{FasslerOrponen14,GGGHMW22} for $\G(3,2)$, and \cite{GGW22, Du25RestrictedMarstrand, Wang26RestrictedProjections} for higher-dimensional Grassmannians and further references. For smooth curves, the picture is fairly complete in dimension and co-dimension $1$ ($k=1$ or $d-1$). Much less is known for $1<k<d-1$. These proofs rely on smoothness, and it is not clear how to extend them to other measures. Chen \cite{Chen18RestrictedProjections} took a different approach, proving a restricted projection theorem for a class of random measures on $\G(3,1)$. His argument does not extend to $\G(2,1)$.

In particular, no genuine restricted Marstrand-type theorem was known for measures on $\G(2,1)$. Our first result applies to uniformly perfect measures of directions. A locally finite Borel measure $\mu$ on a metric space $X$ is \emph{uniformly perfect} if $\diam(\spt \mu) > 0$ and there are $\beta \in [0,1)$ and $D > 1$ such that $\mu(B(x,r)) \leq \beta \cdot \mu(B(x,Dr))$ for every closed ball $B(x,r) \subset X$ with $\spt \mu \not\subset B(x,Dr)$. Every non-trivial Ahlfors $s$-regular measure with $s > 0$ is uniformly perfect. (A measure $\mu$ on a metric space $X$ is \emph{Ahlfors $s$-regular} if there is $C\ge 1$ such that $C^{-1}r^s \le \mu(B(x,r)) \le Cr^s$ for all $x\in\spt\mu,\, r\in (0,\diam(\spt\mu)].$) 
\begin{theorem} \label{thm:projection-ADR}
    Let $\nu$ be a uniformly perfect measure on $\mathbb{G}(2,1)$. For any Borel set $E\subset\R^2$,
    \begin{displaymath}
        \hdim \Pi_V(E) = \min\{\hdim E,1\} \quad \text{ for $\nu$-a.e. } V \in \G(2,1).
    \end{displaymath}
\end{theorem}

Kaufman and Mattila \cite{MR407248} constructed singular measures $\nu$ on $\G(2,1)$, of dimension arbitrarily close to $1$, for which the Marstrand-Mattila estimate fails for some compact set $E$ and all directions  in $\spt(\nu)$. Both $E$ and $\nu$ have a special structure. In particular, $\nu$ is supported on a ``well-spaced'' set $\Theta$: for many pairs of scales $0<\delta\ll r$, the set $\Theta\cap B_r$ is contained in a $\delta$-ball for every $r$-ball $B_r$. Uniform perfectness rules out these gaps between scales. 

In higher dimensions, there are simple linear obstructions to a Marstrand-Mattila estimate. Namely, for $W\in\G(d,d-k)$, consider the Schubert cycle
\[
    \mathcal{S}_W = \{V\in\G(d,k):V\cap W\ne\{0\}\}.
\]
Then
\[
    \hdim(\Pi_V(W^\perp)) \le k-1 < k = \min\{\dim W^\perp,k\} \quad \text{for all } V \in \mathcal{S}_W.
\]
In particular, Ahlfors-regularity is not enough to guarantee a Marstrand-Mattila estimate for any $d\ge 3$ and arbitrary $k$. Schubert cycles are essentially the same as the pencils used in projection theory, see \cite[Definition~1.4]{Wang26RestrictedProjections}. 

The Kaufman-Mattila example also shows why a global non-concentration assumption is not enough. In $\G(2,1)$, the Schubert cycles are points. A well-spaced set avoids concentration at any one point globally, but can be concentrated near a point at many scales. Theorem \ref{thm:projection-non-conc-Schubert-cycle} shows that a mild local non-concentration near Schubert cycles at every scale suffices for a Marstrand-Mattila estimate.
\begin{theorem}
\label{thm:projection-non-conc-Schubert-cycle}
    Let $1 \leq k < d$. Let $\nu$ be a nonzero finite Borel measure on $\G(d,k)$. Assume that, for some $\e\in(0,1/2)$,
    \begin{equation}\label{eq:schubert-contraction}
        \nu\bigl(B_V(r)\cap\mathcal{S}_W(\e r)\bigr)
        \le (1-\e) \nu(B_V(r))
    \end{equation}
    for all $V\in\spt\nu$, $W\in\G(d,d-k)$, and $0<r\le1$.
    Then, for any Borel set $E\subset\R^d$,
    \[
        \hdim \Pi_V(E) = \min\{\hdim E,k\} \quad \text{ for $\nu$-a.e. } V \in \G(d,k).
    \]
    In particular, for any $1\le k<d$, there are measures supported on sets of arbitrarily small (positive) dimension in $\G(d,k)$ for which a Marstrand-type projection theorem holds.
\end{theorem}

He \cite{He20} proved a higher-dimensional projection theorem under a global (rather than local) power-law non-concentration near Schubert cycles, with a dimension bound strictly above $(k/d)\hdim E$ for $0<\hdim E<d$; this is a higher dimensional version of Bourgain's projection theorem \cite{Bo2}. Under the stronger \emph{local} non-concentration assumption \eqref{eq:schubert-contraction}, we obtain the optimal Marstrand-Mattila estimate. 

We prove these results through discretized estimates, which are also useful in applications. Theorem \ref{thm:proj-general} is the discretized version of Theorem \ref{thm:projection-non-conc-Schubert-cycle}.

\subsection{Optimal Furstenberg estimates}

Projection theorems are heuristically (and, at the discretized level, formally) special cases of Furstenberg-type estimates. In the classical Furstenberg problem, one is given a set $E\subset\R^2$ and a family of lines $\mathcal{L}$, and one assumes $\hdim(E\cap\ell)\ge s$ for all $\ell\in\mathcal{L}$. The goal is to estimate $\hdim E$ in terms of $s$ and $\hdim\mathcal{L}$. The sharp version of this problem was established in \cite{23arXiv230808819R}: with $\hdim\mathcal{L}=t$, one has
\[
 \hdim(E)\ge \min\left\{s+t, \frac{3s+t}{2},s+1\right\},
\]
and this is sharp. The $\frac{3s+t}{2}$ term comes from a lattice example, essentially the same construction of Kaufman and Mattila \cite{MR407248}; see also Wolff's survey  \cite{Wolff99}. We ask whether additional conditions on $E\cap\ell$ can rule out this example and remove the term.

Point-line duality leads to an equivalent formulation. Let $E\subset\R^2$ have $\hdim(E)=t$ and, for each $x\in E$, let $\mathcal{L}_x$ be a family of lines through $x$ of dimension $\ge s$. Then
\[
\hdim\left(\bigcup_{x\in E}\mathcal{L}_x\right)\ge \min\left\{s+t, \frac{3s+t}{2},s+1\right\}.
\]
The set on the left is a subset of the affine Grassmannian $\A(2,1)$ of lines in $\R^2$. Roughly speaking, the projection case corresponds to the situation where the set of directions of $\mathcal{L}_x$ is independent of $x$. Very recently, Cs\"ornyei and Stull \cite[Theorem~1.1]{CsornyeiStull26Hyperplanes} partially generalized this estimate to higher dimensions, replacing bundles of lines by bundles of hyperplanes.

We study this dual problem in general dimension, allowing bundles of  $k$-planes in $\R^d$ for any $1\le k<d$. Theorem \ref{thm:Furstenberg-continuous} is the Furstenberg counterpart of Theorems \ref{thm:projection-ADR} and \ref{thm:projection-non-conc-Schubert-cycle}.
Recall that the (lower) Hausdorff dimension of a nonzero finite Borel measure $\mu$ is
\begin{equation} \label{eq:hdim-measure}
    \hdim\mu=\inf\{\hdim A:A\text{ is Borel and }\mu(A)>0\}.
\end{equation}

\begin{theorem} \label{thm:Furstenberg-continuous}
    Let $1 \leq k < d$. Let $E\subset\R^d$ be a Borel set with $\hdim(E)=t>0$. For each $x\in E$, let $\Theta_x\subset\G(d,k)$ be the support of a measure $\nu_x$ satisfying the assumptions of Theorem \ref{thm:projection-non-conc-Schubert-cycle}, with $\hdim\nu_x\ge s$. Then,
    \[
        \hdim\left(\bigcup_{x\in E}\{x+V:V\in \Theta_x\}\right)\ge \min\{t,d-k\}+s.
    \]
\end{theorem}

For example (after dualizing again), if $E$ contains a bi-Lipschitz image of the ternary Cantor set in every direction, then $\hdim(E)\ge 1+\log 2/\log 3$. This says that copies of the Cantor set (even after bi-Lipschitz distortion) in different directions have very little overlap. This can be seen as a fractal analog of the classical Kakeya problem.

Theorems \ref{thm:projection-non-conc-Schubert-cycle} and \ref{thm:Furstenberg-continuous} will follow from the more general results in \S\ref{s:continuous-versions}, where spread-outness replaces local Schubert non-concentration.

\subsection{Discretized formulations}
\label{s:discretized-formulations}

A \emph{dyadic scale} is an element of $2^{-\Z}$. For $r\in2^{-\N}$ and $m\in\N$, let $\mathcal D_r(\R^m)=\{r(j+[0,1)^m):j\in\Z^m\}$. For $E\subset\R^m$, let $\mathcal D_r(E)$ be the cubes in $\mathcal D_r(\R^m)$ that meet $E$, write $|E|_r=|\mathcal D_r(E)|$, and denote the lower-left endpoint of $Q\in\mathcal D_r(\R^m)$ by $z(Q)$. For a point $x\in\R^m$, $\mathcal D_r(x)$ denotes the unique cube in $\mathcal D_r(\R^m)$ containing $x$. For a family $\mathcal J$ of dyadic cubes, write $z\mathcal J=\{z(Q):Q\in\mathcal J\}$.

We use the standard notions of \emph{Katz--Tao} and \emph{Frostman} sets. For $\delta\in2^{-\N}$, $t\ge0$, and $C\ge1$, a nonempty bounded set $P\subset\R^d$ is \emph{$(\delta,t,C)$-Frostman} if
\[
|P\cap Q|_\delta\le Cr^t|P|_\delta,
\qquad Q\in\mathcal D_r(\R^d),\quad r\in2^{-\N}\cap[\delta,1],
\]
and \emph{$(\delta,t,C)$-Katz--Tao} if
\[
|P\cap Q|_\delta\le C (r/\delta)^t, \qquad Q\in\mathcal D_r(\R^d),\quad r\in2^{-\N}\cap[\delta,1].
\]

Let  $\theta^*$ denote the transpose of $\theta$. Identify $\R^d=\R^k\times\R^{d-k}$. For $\theta\in\mathbf G_{d,k}=\R^{k\times(d-k)}$, let
\[
V_\theta=\operatorname{Graph}(\theta^*)
=\{(x,\theta^*x):x\in\R^k\}.
\]
The map $\theta\mapsto V_\theta$ identifies $\mathbf G_{d,k}$ with a Zariski-dense open subset of $\G(d,k)$ (this is the higher dimensional analog of the parametrization of lines by $(a,b)\leftrightarrow y=ax+b$). We similarly identify the corresponding big cell $\mathbf A_{d,k}$ in the affine Grassmannian with $\R^{k\times(d-k)}\times\R^{d-k}$ via
\[
(\theta,b)\longmapsto V_{\theta,b}=\{(x,\theta^*x+b):x\in\R^k\},
\]
and use dyadic covering numbers in these coordinates.

\begin{definition}[Plane configurations]\label{def:planeConfigurations}
Let $M\in\N$. A \emph{$(\delta,M)$-plane configuration} is a pair $(P,\mathbf\Theta)$, where $P\subset\delta\Z^d\cap[0,1]^d$ is nonempty and, for each $p\in P$,
\[
\mathbf\Theta(p)\subset\mathcal D_\delta([-1,1]^{k\times(d-k)}),
\qquad M\le|\mathbf\Theta(p)|<2M.
\]
We call $\mathbf\Theta(p)$ the \emph{shading} of $p$ and omit $M$ when it is not needed. We identify $\mathbf{\Theta}(p)$ with its union, so that $\theta\in\mathbf{\Theta}(p)$ means $\theta$ is an element of some cube in $\mathbf{\Theta}(p)$.

The \emph{Furstenberg set} associated to $(P,\mathbf\Theta)$ is defined using the lower-left corners $z\mathbf\Theta(p)$:
\[
\mathcal F(P,\mathbf\Theta)
=\{V_{\theta}+p:p\in P, \theta\in z\mathbf\Theta(p)\}\subset\mathbf A_{d,k}.
\]
\end{definition}
Note that allowing all $\theta\in\mathbf\Theta(p)$ instead of just the lower-left corners $z\mathbf\Theta(p)$ in the definition of $\mathcal F(P,\mathbf\Theta)$ changes its $\delta$-covering number by at most a factor depending on $d,k$, and so in practice we may ignore the difference. 

The \emph{model hyperplanes} in $\mathbf G_{d,k}$ are
\[
H_{[x,y,u]}=\{\theta\in\mathbf G_{d,k}:x\cdot\theta\cdot y=u\},
\qquad x\in\R^k\setminus\{0\},\quad y\in\R^{d-k}\setminus\{0\},\quad u\in\R.
\]
The family of model hyperplanes in $\mathbf G_{d,k}$ is denoted by $\mathfrak{H}_{d,k}$. Each $H\in \mathfrak{H}_{d,k}$ is the intersection with $\mathbf G_{d,k}$ of a Schubert divisor in $\G(d,k)$, and every affine tangent hyperplane to a Schubert divisor at a smooth point  is a model hyperplane in big cell coordinates. See \S\ref{s:bigcell}.

\begin{definition}[Singular scales and spread-out families]\label{def:intro-spreadout}\label{def:spread-out-cube-family}
Fix $\beta\in(0,1]$ and $C>0$, and let $\mathcal J\subset\mathcal D_\delta(\mathbf G_{d,k})$ be a nonempty finite family. All balls and distances in $\mathbf G_{d,k}$ are Euclidean. For $\mathcal J'\subset\mathcal J$ and $\theta\in\mathcal J'$, a scale $r\in2^{-\N}\cap[\delta,1]$ is \emph{$\beta$-singular for $(\mathcal J',\theta)$} if there exists $H\in\mathfrak{H}_{d,k}$ such that
\[
\mathcal J'\cap B(\theta,r)
\subset\{\omega\in\mathbf G_{d,k}:\dist(\omega,H)<\beta r\}.
\]

The family $\mathcal J$ is \emph{$(\delta,\beta,C)_{d,k}$-spread-out} if $|\mathcal J'|\le C2^{-\beta\mathfrak s  }|\mathcal J|$ whenever $\mathcal J'\subset\mathcal J$, $\mathfrak s\in\N\cup\{0\}$, and every $\theta\in\bigcup\mathcal J'$ has at least $\mathfrak s$ scales that are $\beta$-singular for $(\mathcal J',\theta)$. A plane configuration $(P,\mathbf\Theta)$ is \emph{$(\delta,\beta,C,M)$-spread-out} if $\mathbf\Theta(p)$ is $(\delta,\beta,C)_{d,k}$-spread-out for every $p\in P$. 
\end{definition}

The definitions agree with those of \cite{SpreadOutFurstSticky}. We recall spread-out measures and record their basic properties in Section \ref{sec:spread-out}. Examples are given in \cite[\S 6.1]{SpreadOutFurstSticky} and Section \ref{s:non-planar-measures-are-spread-out}.

Our main theorem is the following sharp lower bound for the number of affine $k$-planes in a Furstenberg configuration.
\begin{theorem}\label{thm:furst-general}
Let $1\le k<d$, $s\in(0,1]$, and $t\in(0,d-k]$. For every $\eta>0$ there is $\e=\e(s,t,\eta,d,k)>0$ such that, whenever $\delta\in2^{-\N}$ is sufficiently small in terms of $d,k,s,t,\eta$ and $M\in\N$, every $(\delta,s,\delta^{-\e},M)$-spread-out configuration $(P,\mathbf\Theta)$ such that $P$ is $(\delta,t,\delta^{-\e})$-Frostman satisfies 
\[
    |\mathcal F(P,\mathbf\Theta)|_\delta\ge \delta^{\eta-t}\cdot M.
\]
If, instead, $P$ is $(\delta,t,\delta^{-\e})$-Katz--Tao, then $|\mathcal F(P,\mathbf\Theta)|_\delta\ge \delta^{\eta}|P|_{\delta}\cdot M$.
\end{theorem}
The Frostman and Katz--Tao formulations are in fact equivalent; see \S\ref{s:frostman-to-katz-tao} for the more challenging implication.

The projection theorem corresponds to the case where every shading $\mathbf\Theta(p)$ is dense in a common family $\Theta$. For $\theta\in\mathbf G_{d,k}$, write $\pi_\theta(x,y)=x+\theta\cdot y$, with $(x,y)\in\R^k\times\R^{d-k}$, so $\pi_\theta$ takes values in $\R^k$ and $\ker\pi_\theta=V_\theta^\perp$. In Theorem \ref{thm:proj-general}, we interchange $k$ and $d-k$: thus $\theta\in\R^{(d-k)\times k}$ and $\pi_\theta:\R^{d-k}\times\R^k\to\R^{d-k}$. 

\begin{theorem} \label{thm:proj-general}
    
Let $s\in(0,1]$, $t\in(0,d-k]$, and $\eta>0$. There exists $\e>0$ such that the following holds whenever $\delta\in 2^{-\N}$ and $\delta\ll_{s,t,\eta}1$.

Let $\mathbf\Theta\subset\mathcal D_\delta([-1,1]^{(d-k)\times k})$ be $(\delta,s,\delta^{-\e})_{d,d-k}$-spread-out. For every $(\delta,t,\delta^{-\e})$-Frostman set $P\subset [0,1]^d$, there is $\mathbf\Theta'\subset\mathbf\Theta$ such that $|\mathbf\Theta'|\ge(1-\delta^\e)|\mathbf\Theta|$, and
\[
 |\pi_{\theta}(P')|_\delta \ge\delta^{\eta-t}
 \quad \text{for all }\theta\in\mathbf\Theta',\ P'\subset P,\
 |P'|_\delta\ge\delta^\e|P|_\delta.
\]
If  $P$ is instead $(\delta,t,\delta^{-\e})$-Katz--Tao, then the same conclusion holds with $|\pi_{\theta}(P)|_\delta\ge\delta^{\eta}|P|_\delta$ as the projection bound.
\end{theorem}

\subsection{Applications}
\label{s:applications}

Theorem \ref{thm:furst-general} has applications to smooth restricted projections, distance sets, additive combinatorics, and self-affine sets. Section \ref{sec:examples-applications} contains the proofs and further discussion.

\subsubsection{Marstrand for surfaces in \texorpdfstring{$\G(4,2)$}{G(4,2)}}

Although our motivation comes from fractal restricted projections and Furstenberg sets, Theorem \ref{thm:furst-general} also applies to smooth submanifolds of $\G(d,k)$. For $\G(4,2)$, combining the theorem with the results of \cite{GGW22}, we obtain a restricted projection theorem for generic surfaces.
\begin{theorem} \label{thm:intro-projection-G42}
    Let $M$ be the closed unit ball in $\R^2$. There is an open and dense set $\mathcal{U}$ of embeddings $F\colon M\to\G(4,2)$ in the compact-open $C^\infty$ topology such that, for all $F\in\mathcal{U}$ and for any Borel set $E\subset\R^4$,
    \[
        \hdim(\Pi_{F(x)}(E)) = \min\{\hdim(E),2\} \quad \text{for $\mathcal{H}^2$-a.e. } x\in M.
    \]
\end{theorem}
See Section \ref{s:restricted-projections-surfaces} for the proof and context.

\subsubsection{Pinned distance projections}

Given a set $E\subset\R^d$ and $y\in\R^d$, let $\Delta^y(E)=\{|x-y|:x\in E\}$ denote the pinned distance set of $E$ with respect to $y$. A simple reduction from pinned distance projections to orthogonal projections allows us to apply Theorem \ref{thm:proj-general}: if $E\subset\R^d$ is Borel and $\nu$ is a spread-out measure on $\R^d$, then
\[
    \hdim(\Delta^y(E)) = \min\{\hdim(E),1\} \quad \text{for $\nu$-a.e. } y\in\R^d.
\]
See \S\ref{S:distance-sets} for the simple proof and discussion.

\subsubsection{Additive combinatorics and fractals}

It is well-known that projection and Furstenberg estimates lead to additive combinatorial estimates for fractal sets \cite{Bo2}. Extending \cite[Theorem~6.5]{SpreadOutFurstSticky}, we obtain a discretized sum-product estimate for spread-out sets satisfying a Katz--Tao condition. See \S\ref{s:discretized-sum-product} for the statement, proof, and further discussion.

\subsubsection{A new proof of the Bárány-Hochman-Rapaport theorem}

The Ledrappier-Young formula and more general versions connect projection theorems to the dimension theory of dynamical systems. As a direct corollary, we recover the main result of \cite{BHR}: the Hausdorff dimension of a planar, strongly separated self-affine set equals its affinity dimension, provided that the linear parts of the defining affine maps generate a non-compact and totally irreducible subgroup of $\mathrm{GL}(2,\R)$. See \S\ref{s:self-affine} for further discussion.

\subsection{Proof strategy}

The starting point in the proof of Theorem \ref{thm:furst-general} is the sticky Furstenberg estimate in the companion paper, \cite[Theorem~1.9]{SpreadOutFurstSticky}, where $P$ is ``almost Ahlfors-regular''. We recall the definitions from \cite{SpreadOutFurstSticky} in Section \ref{sec:spread-out}. Ren and Wang \cite[\S 6]{23arXiv230808819R} similarly passed from regular to general planar Furstenberg configurations by decomposing branching functions and combining estimates at different scales, although the details differ substantially.

Our argument to reduce the general case to the sticky case has several parallels to the proof of the 3D Kakeya theorem \cite{WangZahl25Kakeya,GuthWangZahl26} and it is significantly simpler, but there are some important differences. We establish the estimate by induction on the loss $\delta^{\eta}$: we show that an estimate with loss $\delta^{\eta}$ implies an estimate with loss $\delta^{\eta'}$ for some $\eta'<\eta$. 

The sticky theorem handles configurations whose branching function is approximately linear. Otherwise, there is an intermediate scale $\rho$ at which the point set has excess branching, $|P|_{\rho} \ge \rho^{-(d-k)-\chi}$ for some positive $\chi$.

We use this excess branching to improve the exponent and close the induction. After some rescaling arguments inspired by \cite{DemeterWang25,WangWu26TwoEnds}, we are led to studying configurations of slablets. A priori we know very little about their branching function, but we show that we can reduce to one of two cases: the slablets satisfy a Frostman condition with error $\delta^{-\e}$ for a small but fixed $\e\ll \rho$, or they satisfy a Katz--Tao condition at a suitable scale. In the Frostman case, the gain comes from the excess branching at scale $\rho$; in the Katz--Tao case, it comes from the subcritical number of slablets that arises from the failure of the Frostman condition with error $\delta^{-\e}$. In both cases, we can close the induction and obtain a better exponent $\eta'<\eta$.


\subsection{Notation}

Given a finite set $P$, we write $m_P$ for the uniform probability measure on $P$.

We write $|P|$ for the cardinality of a finite set $P$. We use the cube and covering notation of Section \ref{s:discretized-formulations}; we also allow non-dyadic scales $\delta>0$. For $E\subset\R^d$, its \emph{$\delta$-discretization} is $E^{(\delta)}=z\mathcal D_\delta(E)\subset\delta\Z^d$, so $|E|_\delta=|E^{(\delta)}|$. Up to a factor depending only on the ambient dimension, $|E|_\delta$ is the minimal number of $\delta$-balls needed to cover $E$.

The open $r$-neighborhood of a set $E\subset\R^d$ is denoted by $E(r)$. We write $B(0,r)$ for the ball of radius $r$ centered at the origin. For $\lambda>0$, $\lambda Q$ and $\lambda B$ denote the concentric dilations of a cube $Q$ and a ball $B$ by the factor $\lambda$.

If $(X,\mathcal{F},\mu)$ is a measure space, and $f \colon X \to Y$ is a map, we write $f\mu$ for the push-forward measure $(f\mu)(B) = \mu(f^{-1}(B))$ on $Y$.

We write $A \lessapprox_{\delta} B$ to mean $A \leq B(\log(1/\delta))^{C}$, where the implicit exponent $C$ is independent of $\delta$, but may depend on the other ambient parameters, which will be specified as needed. We define $\gtrapprox_{\delta}$ similarly.

By $A\gg_{\text{par}} 1$ we mean that $A$ is sufficiently large, depending on the parameters $\text{par}$. We define $A\ll_{\text{par}} 1$ similarly.

We usually work with dyadic scales. When a scale $\delta$ is not dyadic, we may replace $\delta$ by the closest dyadic scale. The resulting constant factors are harmless in every such application.

The dimensions $1\le k\le d$ are fixed in each section. All implicit constants may depend on $d,k$.

Logarithms are always base $2$ unless otherwise specified.

\subsection{AI declaration}

All the main ideas of the paper were conceived by the authors without AI involvement. At the writing stage, PS used several LLMs to draft proofs of technical results in some parts of the paper, and to search for errors and simplifications in the arguments. Some applications involved more substantial AI assistance, particularly in \S\ref{s:restricted-projections-surfaces}. The authors carefully checked all AI-generated text and revised it substantially to ensure correctness and clarity. Any remaining errors or inaccuracies in the paper are the sole responsibility of the authors.

\subsection{Acknowledgments}

AC is supported by a Clay Research Fellowship. 

Part of this project was carried out during visits of PS to MIT and NYU-Courant in 2025. PS is grateful to these institutions and to Larry Guth for their hospitality. PS also acknowledges support from the Simons Foundation via a Simons Fellowship and from a Discovery Grant from the Natural Sciences and Engineering Research Council (NSERC). This material is partially based upon work supported by the Swedish Research Council under grant no. 2021-06594 while PS was in residence at Institut Mittag-Leffler in Djursholm, Sweden, during the semester Interactions between fractal geometry, harmonic analysis, and dynamics in 2026.

HW is supported by NSF CAREER DMS-2238818 and NSF DMS-2055544. HW thanks Larry Guth for hospitality during her visit to MIT. 

DZ was supported by the Simons Dissertation Fellowship. DZ thanks NYU-Courant for the hospitality during his 2025 visit when this project was started.

\section{Spread-out measures and branching functions}
\label{sec:spread-out}

We recall the definitions from \cite{SpreadOutFurstSticky} and record the properties used in this paper. We also review uniformization and the branching function. This section is largely expository, and we refer to \cite{SpreadOutFurstSticky} for proofs.

\subsection{Big cells in the Grassmannian}

\label{s:bigcell}

We include the discussion of big cells from \cite{SpreadOutFurstSticky} to keep the paper self-contained. For $0\le k\le d$, let $\G(d,k)$ and $\A(d,k)$ denote the Grassmannian of $k$-dimensional linear subspaces of $\R^d$ and the affine Grassmannian of $k$-dimensional affine subspaces of $\R^d$, respectively.

The matrix space $\mathbf{G}_{d,k}=\R^{k\times(d-k)}$ can be identified with a Zariski-dense and open subset of $\G(d,k)$ via the embedding
\begin{equation}\label{def:graph}
    \theta \to V_{\theta} = \operatorname{Graph}(\theta^*) =  \{ (x, \theta^* x) : x \in \R^{k}\},
\end{equation}
where $\theta^*$ is the transpose of $\theta$. The orthogonal complement $V_{\theta}^\perp$ equals $\ker(\pi_\theta)$, where $\pi_{\theta} \colon \R^{k} \times \R^{d - k} \to \R^{k}$ is the map $\pi_\theta(x,y)=x+\theta\cdot y$.  We think of $\mathbf{G}_{d,k}$ as a ``big cell'' in the Grassmannian. 

For the remainder of this section, assume that $1\le k<d$. Let $\X_{k,d-k}=\bigl((\R^{k}\setminus \{0\})\times (\R^{d-k}\setminus \{0\})\times \R\bigr)/\sim$, where $\sim$ is the equivalence relation given by $(x,y,t)\sim (x',y',t')$ if there exist $\lambda,\mu\in\R\setminus\{0\}$ such that $x'=\lambda x$, $y'=\mu y$, and $t'=\lambda\mu t$. We write $[x,y,t]$ for the equivalence class of $(x,y,t)$. Given $[x,y,t]\in \X_{k,d-k}$, let 
\begin{displaymath}
    H_{[x,y,t]} = \bigl\{ \theta\in \mathbf{G}_{d,k}: x\cdot \theta\cdot y = t \bigr\}.    
\end{displaymath}
The definition is independent of the representative of the equivalence class. Identifying $\mathbf{G}_{d,k}$ with $\R^{k(d-k)}$, we see that $H_{[x,y,t]}\in \A(k(d-k),k(d-k)-1)$ is an affine hyperplane in $\R^{k (d-k)}$. Viewed as a subset of $\G(d,k)$, the same set $H_{[x,y,t]}$ corresponds to the Schubert divisor
\begin{equation} \label{eq:schubert-variety}
    \{\theta\in \mathbf{G}_{d,k}: \dim(V_{\theta}\cap W_{[x,y,t]})\ge 1\},
\end{equation}
where 
\begin{displaymath}
    W_{[x,y,t]}=\bigl\{(sx,v)\in \R^k\times\R^{d-k}: y\cdot v=ts\bigr\}\in \G(d,d-k).
\end{displaymath}

We compare this terminology with the pencils used in projection theory. For a nonzero proper subspace $L\subset\R^d$, let
\[
    \mathcal P_L=\{V\in\G(d,k):
    \dim(L\cap V^\perp)>\max\{\dim L-k,0\}\}.
\]
See e.g. \cite[Definition~1.4]{Wang26RestrictedProjections}. Thus $\mathcal P_L$ is the Schubert variety on which the orthogonal projection of $L$ onto $V$ has dimension less than $\min\{\dim L,k\}$.y

If $\dim L=k$, then
\[
    \mathcal P_L=\{V\in\G(d,k):V\cap L^\perp\ne\{0\}\},
\]
so $\mathcal P_L$ is the Schubert divisor determined by $L^\perp$. Under the identification \eqref{def:graph}, the divisor \eqref{eq:schubert-variety} is therefore $\mathcal P_{W_{[x,y,t]}^\perp}\cap\mathbf G_{d,k}=H_{[x,y,t]}$.

Pencils with $\dim L\ne k$ are Schubert varieties of codimension $|\dim L-k|+1$, rather than divisors. Moreover, a general codimension one pencil need not be an affine hyperplane in a fixed big-cell chart. At every smooth point its tangent hyperplane has rank-one normal and is in $\mathfrak{H}_{d,k}$. Thus the model hyperplanes are affine local models for arbitrary pencil divisors. 

The set $\mathbf{G}_{d,k}$ has two natural metrics: the one induced from the natural metric on $\G(d,k)$ invariant under the action of the orthogonal group, and the one induced from the Euclidean metric on $\R^{k\times (d-k)}$. The two metrics induce the same topology and are bi-Lipschitz equivalent on compact subsets of $\mathbf{G}_{d,k}$. Unless otherwise specified, we will always use the Euclidean metric on $\mathbf{G}_{d,k}$.

These considerations extend to the affine Grassmannian $\A(d,k)$ as well. The big cell $\mathbf{A}_{d,k}$ of $\A(d,k)$ is the set of affine $k$-planes $V_{(\theta,b)} = \{(x,\theta^* x+b) : x \in \R^{k}\}$, with $\theta\in \mathbf{G}_{d,k}$ and $b\in \R^{d-k}$. The big cell is once again Zariski-dense and open. We identify $\mathbf{A}_{d,k}\cong \R^{(d-k)(k+1)}$ by identifying the plane $V_{\theta,b}$ with the pair $(\theta,b) \in \R^{k(d - k) + (d - k)}$. We will use a dyadic system $\{\mathcal{D}_{\delta}(\mathbf{A}_{d,k}) : \delta \in 2^{-\Z}\}$ on $\mathbf{A}_{d,k}$ which is the standard dyadic system on $\R^{(d - k)(k + 1)}$ under this coordinate identification.  The metrics on $\mathbf{A}_{d,k}$ inherited from $\A(d,k)$ and $\R^{(k+1)(d-k)}$ are bi-Lipschitz equivalent on compact subsets of $\mathbf{A}_{d,k}$.

For $\theta\in\mathbf G_{d,k}$ and $(x,y)\in\R^k\times\R^{d-k}$, write $\rho_\theta(x,y)=y-\theta^*x=\pi_{-\theta^*}(y,x)$. The affine coordinates satisfy
\begin{equation}\label{eq:affine-incidence-coordinates}
    V_\theta+(x,y)\longleftrightarrow(\theta,\rho_\theta(x,y)).
\end{equation}
On fixed bounded sets, this coordinate map is Lipschitz, so changing point and plane representatives by $O(r)$ changes the affine coordinates by $O(r)$. Since the dyadic system on $\mathbf A_{d,k}$ is the product system, for $r>0$, any finite $\Theta\subset r\Z^{k\times(d-k)}$, and sets $P_\theta\subset\R^d$,
\begin{equation}\label{eq:affine-incidence-counting}
    \left|\bigcup_{\theta\in\Theta}\{V_\theta+p:p\in P_\theta\}\right|_r
    =\sum_{\theta\in\Theta}|\rho_\theta(P_\theta)|_r.
\end{equation}

The Grassmannian $\G(d,k)$ and bounded subsets of $\A(d,k)$ admit finite covers by rotations of compact subsets of their big cells. Many questions about $\G(d,k)$ and $\A(d,k)$ therefore reduce to questions in Euclidean coordinates. The advantage for spread-outness is that the model hyperplanes $H_{[x,y,t]}$ are affine hyperplanes in $\mathbf{G}_{d,k}$.

\subsection{Spread-out sets and measures}

The geometric idea of spread-outness is that a set in $\mathbf{G}_{d,k}$ cannot concentrate near model hyperplanes $H_{[x,y,t]}$ at many scales around many points. The definitions quantify this idea by counting singular scales.

\begin{definition}[$\beta$-singular scale]
    Let $\beta\in (0,1]$, let $r \in (0,1]$, let $\Theta \subset \mathbf{G}_{d,k}$, and let $\theta \in \Theta$. We say that $r$ is a \emph{$\beta$-singular scale} for $(\Theta,\theta)$ if there exists $[x,y,t]\in \X_{k,d-k}$ such that $\Theta\cap B(\theta,r)\subset H_{[x,y,t]}(\beta r)$. The balls are with respect to the Euclidean metric on $\R^{k(d-k)}$. (The particular metric plays little role.) Geometrically, $\Theta\cap B(\theta,r)$ lies in a strip of width comparable to $\beta r$ around one of the affine hyperplanes $H_{[x,y,t]}$.

    For $\Theta \subset \R^{k(d-k)}$ and $\delta,\Delta \in 2^{-\N}$ with $\delta \le \Delta$, we denote
    \begin{displaymath}
        \mathfrak{s}_{d,k}(\theta \mid \Theta,\beta,\delta \to \Delta)
        =
        \bigl|\{r \in 2^{-\N} \cap [\delta,\Delta] : r \text{ is a $\beta$-singular scale for } (\Theta,\theta)\}\bigr|.
    \end{displaymath}
    We abbreviate $\mathfrak{s}_{d,k}(\theta \mid \Theta,\beta,\delta) = \mathfrak{s}_{d,k}(\theta \mid \Theta,\beta,\delta \to 1)$ and, when $\delta$ is clear, $\mathfrak{s}_{d,k}(\theta \mid \Theta,\beta) = \mathfrak{s}_{d,k}(\theta \mid \Theta,\beta,\delta)$. We also omit the subscripts $d,k$ when they are clear from context.

    If $\mathcal{J} \subset \mathcal{D}_{\rho}(\mathbf{G}_{d,k})$ is a family of dyadic cubes for some dyadic scale $\rho$, and $\theta \in \bigcup \mathcal{J}$, we write $\mathfrak{s}_{d,k}(\theta \mid \mathcal{J},\beta,\delta \to \Delta)=\mathfrak{s}_{d,k}(\theta \mid \cup \mathcal{J},\beta,\delta \to \Delta)$.
    If $\mathbf{J} \in \mathcal{J}$ is itself a cube, we write
    \begin{displaymath}
        \mathfrak{s}_{d,k}(\mathbf{J} \mid \mathcal{J},\beta,\delta \to \Delta)
        =
        \inf_{\theta \in \mathbf{J}}\mathfrak{s}_{d,k}(\theta \mid \mathcal{J},\beta,\delta \to \Delta).
    \end{displaymath}
    Thus, when no confusion is likely, we silently identify a family of dyadic cubes with its union in singular scale counts.
\end{definition}

\begin{definition}[$(\delta,\beta,C)_k$-spread-out measure]\label{def:spread-out-measure-k}
    Let $\delta \in 2^{-\N}$ and $C>0$, $\beta \in (0,1]$. A finite nonzero Borel measure $\nu$ on $\R^k$, $k \geq 1$, is called \emph{$(\delta,\beta,C)_k$-spread-out} if the following holds. For a Borel set $E \subset \R^k$, a point $x \in E$, and dyadic scales $\delta \leq \Delta \leq 1$, write
    \begin{displaymath}
        \mathfrak{s}_k(x \mid E,\beta,\delta \to \Delta)
        =
        \bigl|\{r \in 2^{-\N} \cap [\delta,\Delta] : \exists H \in \A(k,k-1) \text{ such that } E \cap B(x,r) \subset H(\beta r)\}\bigr|.
    \end{displaymath}
    Let $\mathfrak{s} \in \N\cup\{0\}$, and let $\Theta \subset \spt \nu$ be a Borel set such that
    \begin{displaymath}
        \mathfrak{s}_k(x \mid \Theta,\beta,\delta \to 1) \geq \mathfrak{s} \text{ for all } x \in \Theta.
    \end{displaymath}
    Then $\nu(\Theta) \leq C \cdot 2^{-\beta \mathfrak{s}}$.
    We say that a nonempty bounded set $P \subset \R^k$ is \emph{$(\delta,\beta,C)_k$-spread-out} if $\mathcal D_\delta(P)$ satisfies Definition \ref{def:spread-out-cube-family}, with all affine hyperplanes in $\R^k$ in place of the model hyperplanes.
\end{definition}

Next, we extend this definition to measures on $\mathbf{G}_{d,k}$.
\begin{definition}[$(\delta,\beta,C)_{d,k}$-spread-out measure]\label{def:perfectSet}
    Let $\delta \in 2^{-\N}$, $C>0$, and $\beta\in(0,1]$. A finite nonzero Borel measure $\mu$ on $\mathbf{G}_{d,k}$ is called \emph{$(\delta,\beta,C)_{d,k}$-spread-out} if the following holds. Let $\mathfrak{s} \in \N\cup\{0\}$, and let $\Theta \subset \spt \mu$ be a Borel set such that
    \begin{displaymath}
        \mathfrak{s}_{d,k}(\theta \mid \Theta,\beta,\delta \to 1) \geq \mathfrak{s} \text{ for all } \theta \in \Theta.
    \end{displaymath}
    Then $\mu(\Theta) \leq C \cdot 2^{-\beta \mathfrak{s}}$.
    We say that a nonempty bounded set $\Theta \subset \mathbf{G}_{d,k}$ is \emph{$(\delta,\beta,C)_{d,k}$-spread-out} if $\mathcal D_\delta(\Theta)$ is $(\delta,\beta,C)_{d,k}$-spread-out in the sense of Definition \ref{def:spread-out-cube-family}.

    We omit the subscripts $d,k$ when they are clear from context.
\end{definition}

For measures on the full Grassmannian, we use the following globalization of Definition \ref{def:perfectSet}.
\begin{definition}[Spread-out measures on $\G(d,k)$]\label{def:spreadOutContinuous}
Fix a finite Borel partition $\G(d,k)=\bigsqcup_{j=1}^{L}\Omega_j$ subordinate to compact big-cell charts: for each $j$ there are $U_j\in O(d)$ and a map $\phi_j\colon\Omega_j\to[-1,1]^{k\times(d-k)}$ such that $U_jV=V_{\phi_j(V)}$ for $V\in\Omega_j$. A finite nonzero Borel measure $\mu$ on $\G(d,k)$ is called \emph{$\beta$-spread out} if there exists $C>0$ such that $\phi_j\,(\mu|_{\Omega_j})$ is $(\delta,\beta,C)_{d,k}$-spread out for all $\delta\in2^{-\N}$ and every $j$ with $\mu(\Omega_j)>0$. Finally, $\mu$ is called spread out if it is $\beta$-spread out for some $\beta\in(0,1]$.
\end{definition}

Recall that spread-out families of dyadic cubes were defined in Definition \ref{def:spread-out-cube-family}. In the singular-scale notation, the hypothesis is $\mathfrak{s}_{d,k}(\mathbf J\mid\mathcal J',\beta,\Delta\to1)\ge\mathfrak s$ for every $\mathbf J\in\mathcal J'$.

It is shown in \cite[\S 6.1]{SpreadOutFurstSticky} that uniformly perfect measures on the line are spread out; we extend this to higher dimensions in Section \ref{s:non-planar-measures-are-spread-out}.

\subsection{Regularity}

For $\delta>0$, $\alpha\ge0$, and $C>0$, a finite nonzero Borel measure $\mu$ on $\R^d$ is \emph{$(\delta,\alpha,C)_d$-Frostman} if
    \[
        \mu(B(x,r))\le Cr^\alpha
        \qquad(x\in\R^d,\ r\ge\delta).
    \]
    We call $\mu$ \emph{$(\alpha,C)_d$-Frostman} if this holds for every $r>0$. Note that we do not assume that $\mu$ is a probability measure and we do not normalize by the total mass. The subscript $d$ is omitted when clear from context.

    \begin{definition} \label{def:Ahlfors-regular}
        We say that a non-trivial finite measure $\mu$ on $\R^d$ is \emph{$(\delta,\alpha,C)_d$-regular} if it is $(\delta,\alpha,C)_d$-Frostman and
        \[
            |\spt\mu\cap B(x,R)|_r\le C(R/r)^\alpha
            \qquad(x\in\R^d,\ \delta\le r\le R).
        \]
        A non-empty bounded set $P\subset\R^d$ is \emph{$(\delta,\alpha,C)_d$-regular} if $m_{P^{(\delta)}}$ is $(\delta,\alpha,C)_d$-regular. The subscript $d$ is omitted when clear from context.
    \end{definition}

    \begin{remark}
        In \cite[Definition~6.1]{OS23}, regularity consists of a Frostman condition and a global covering bound at scale $\delta^{1/2}$ only. In Definition \ref{def:Ahlfors-regular}, we require the covering bound for all $\delta\le r\le R$.
    \end{remark}

\subsection{Basic properties}
\label{s:spread-out-properties}

Under the natural identifications $\mathbf G_{k+1,k}\cong\R^k\cong\mathbf G_{k+1,1}$, the model hyperplanes are precisely the affine hyperplanes in $\R^k$, so the spread-out notions coincide. We recall the basic properties of spread-out measures and sets from \cite{SpreadOutFurstSticky}.

\begin{lemma}\label{lem:spread-out-subset}
If $\mu$ is $(\delta,\beta,C)_{d,k}$-spread-out and $A$ is a Borel set with $\mu(A)>0$, then $\mu|_A$ is $(\delta,\beta,C)_{d,k}$-spread-out, and $\mu(A)^{-1}\mu|_A$ is $(\delta,\beta,C/\mu(A))_{d,k}$-spread-out.

If $\Theta$ is a bounded $(\delta,\beta,C)_{d,k}$-spread-out set and $\varnothing\ne\Theta'\subset\Theta$, then $\Theta'$ is $(\delta,\beta,C|\Theta|_\delta/|\Theta'|_\delta)_{d,k}$-spread-out. The same holds for subfamilies of dyadic $\delta$-cubes, with cardinalities in place of covering numbers.
\end{lemma}
\begin{proof}
See \cite[Lemma~2.9]{SpreadOutFurstSticky}.
\end{proof}

\begin{lemma}\label{lem:spread-out-scaling}
For $\lambda\ne0$ and $a\in\mathbf G_{d,k}$, let $T_{\lambda,a}(\theta)=\lambda^{-1}(\theta-a)$. If $\mu$ is $(\delta,\beta,C)_{d,k}$-spread-out, then:
\begin{enumerate}[label=\textup{(\roman*)}]
    \item if $|\lambda|\in2^{-\N}\cap[\delta,1]$, then $T_{\lambda,a}\mu$ is $(\delta/|\lambda|,\beta,C)_{d,k}$-spread-out;
    \item if $\delta\le|\lambda|\le\delta^{-1}$, then $T_{\lambda,a}\mu$ is $(\delta,\beta/2,C')_{d,k}$-spread-out, where $C'\lesssim_\beta C\max\{|\lambda|,|\lambda|^{-1}\}^{\beta}$.
\end{enumerate}
\end{lemma}
\begin{proof}
For $\lambda>0$, see \cite[Lemma~2.12]{SpreadOutFurstSticky}. Reflection $\theta\mapsto-\theta$ preserves singular scale counts.
\end{proof}

\begin{lemma}\label{lem:spread-out-sub-cube}
Let $\Theta\subset\mathbf G_{d,k}$ be a bounded $(\delta,\beta,C)_{d,k}$-spread-out set, and let $\Delta\in2^{-\N}\cap[\delta,1]$. There is a family $\mathcal Q\subset\mathcal D_\Delta(\Theta)$ with $\sum_{Q\in\mathcal Q}|\Theta\cap Q|_\delta\ge|\Theta|_\delta/2$ such that $\Delta^{-1}((\Theta\cap Q)-z(Q))$ is $(\delta/\Delta,\beta,C')_{d,k}$-spread-out for every $Q\in\mathcal Q$, where $C'\lesssim_{d,k}C(1+\log(\Delta/\delta))$.
\end{lemma}
\begin{proof}
See \cite[Lemma~2.8]{SpreadOutFurstSticky}.
\end{proof}

\begin{lemma}\label{lem:spread-out-coarsening}
Let $\mathbf\Theta\subset\mathcal D_\delta(\mathbf G_{d,k})$ be $(\delta,\beta,C)_{d,k}$-spread-out, and let $r\in2^{-\N}\cap[\delta,1]$. If the nonzero cardinalities $|\{\mathbf I\in\mathbf\Theta:\mathbf I\subset\mathbf J\}|$, $\mathbf J\in\mathcal D_r(\mathbf\Theta)$, differ by a factor of at most $K$, then $\mathcal D_r(\mathbf\Theta)$ is $(r,\beta,2KC)_{d,k}$-spread-out.
\end{lemma}
\begin{proof}
See \cite[Corollary~5.4]{SpreadOutFurstSticky}.
\end{proof}

\begin{lemma}\label{lem:uniform-spread-out-discretization}
Let $R\ge1$, $\delta\in2^{-\N}$ with $\delta\ll_R1$, and let $\mu$ be a $(\delta,\beta,C)_{d,k}$-spread-out probability measure on $[-R,R]^{k\times(d-k)}$. For every Borel set $E$ with $\mu(E)>0$, there is a nonempty family $\mathcal J\subset\mathcal D_\delta(E)$ which is $(\delta,\beta,C')_{d,k}$-spread-out, where $C'\lessapprox_\delta C/\mu(E)$. If $\mu$ is also $(\delta,t,C)$-Frostman with $t>0$, then $|\mathcal J|\gtrapprox_\delta\mu(E)C^{-1}\delta^{-t}$. The implicit constants may depend on $R,\beta,t$.
\end{lemma}
\begin{proof}
See \cite[Lemma~2.11]{SpreadOutFurstSticky}.
\end{proof}

\begin{lemma}\label{lem:line-decay-spread-out}
Let $\mu$ be a locally finite Borel measure on $\R$ with $\mu(B(0,1))>0$. Suppose that, for some $\e\in(0,1/8)$,
\[
    \mu(B(x,8\e R))\le\frac14\mu(B(x,R))
    \qquad\text{whenever }B(x,R)\subset B(0,5).
\]
Then $\mu(B(0,1))^{-1}\mu|_{B(0,1)}$ is $(\delta,\beta,C)_1$-spread-out for every $\delta\in2^{-\N}$, with $\beta>0$ and $C\ge1$ depending only on $\e$ and $\mu(B(0,5))/\mu(B(0,1))$.
\end{lemma}
\begin{proof}
See the covering argument in \cite[Proposition~6.2]{SpreadOutFurstSticky}, which starts from \cite[Equation (6.1)]{SpreadOutFurstSticky}.
\end{proof}

\subsection{Selecting a spread-out shading from an average}

\begin{lemma}\label{lem:select-shading-from-average}
    Let $\delta\in2^{-\N}$, $n\in\N$, $N\ge1$, $s\in(0,1]$, and $C>0$. Let $\mathbf{\Xi}_{1},\ldots,\mathbf{\Xi}_{n}\subset\mathcal{D}_{\delta}(\R^{k\times(d-k)})$ be $(\delta,s,C)_{d,k}$-spread-out families supported in a common bounded set independent of $\delta$, and suppose that $N\le |\mathbf{\Xi}_{i}|<2N$ for $i=1,\ldots,n$. Then there is a family $\mathbf{\Xi}\subset\mathcal{D}_{\delta}(\R^{k\times(d-k)})$ such that:
    \begin{enumerate}[label=\textup{(\roman*)}]
        \item $\mathbf{\Xi}\subset\bigcup_{i=1}^{n}\mathbf{\Xi}_{i}$;
        \item $\mathbf{\Xi}$ is $(\delta,s,C')_{d,k}$-spread-out, with $C'\lessapprox_{\delta} C$;
        \item there is $H\in 2^{\N_0}$ such that $H\le|\{i:\theta\in\mathbf\Xi_i\}|<8H$ for every $\theta\in\mathbf\Xi$, and $H|\mathbf\Xi|\approx_\delta nN$.
    \end{enumerate}
\end{lemma}
\begin{proof}
    Apply \cite[Proposition~5.5]{SpreadOutFurstSticky} with equal scales, using a fixed dyadic dilation to place the common support in $[-1,1]^{k\times(d-k)}$, and pull the almost cover back. We obtain $\mathbf\Xi\subset\bigcup_i\mathbf\Xi_i$, an $\alpha$-almost cover with $\alpha\gtrapprox_\delta1$ and dyadic multiplicity $H$ satisfying
    \[
        H\le|\{i:\theta\in\mathbf\Xi_i\}|<2H
        \quad(\theta\in\mathbf\Xi),\qquad
        H|\mathbf\Xi|\approx_\delta\sum_{i=1}^n|\mathbf\Xi_i|\sim nN.
    \]
    By \cite[Lemma~5.3]{SpreadOutFurstSticky}, $\mathbf\Xi$ is $(\delta,s,\alpha^{-1}C)_{d,k}$-spread-out. All three conclusions follow.
\end{proof}

\subsection{Uniformization and the branching function}
\label{s:uniformization}

We recall the definitions and some basic facts about uniform sets and the branching function.

\begin{definition}[Uniform sets and branching numbers]\label{def:uniform-set}
Let $\Delta\in2^{-\N}$ and $m\in\N$, and let $P\subset[0,1)^d$. We say that $P$ is \emph{$(\Delta,m)$-uniform with branching numbers $(N_j)_{j=0}^{m-1}$} if every cube $Q\in\mathcal D_{\Delta^j}(P)$ contains exactly $N_j$ cubes from $\mathcal D_{\Delta^{j+1}}(P)$, for each $0\le j\le m-1$.
\end{definition}

\begin{definition}[Branching function]\label{def:branching-function}
If $P$ is nonempty and $(\Delta,m)$-uniform with branching numbers $(N_j)_{j=0}^{m-1}$, its \emph{branching function} is the piecewise-linear function $f=f_P:[0,m]\to\R$ determined by
\[
    f(i)=\log_{1/\Delta}|P|_{\Delta^i}
    =\sum_{j=0}^{i-1}\log_{1/\Delta}N_j,
    \qquad i=0,\ldots,m.
\]
It satisfies $f(0)=0$ and is non-decreasing and $d$-Lipschitz.
\end{definition}

\begin{lemma}\label{lem:uniformization}
Fix $\Delta\in2^{-\N}$ and $m\in\N$, and write $\delta=\Delta^m$. For any set $P\subset[0,1)^d$ there is a $(\Delta,m)$-uniform subset $P'\subset P$ such that
\[
    |P'|_\delta\ge(4(d+1)\log(1/\Delta))^{-m}|P|_\delta.
\]
\end{lemma}
\begin{proof}
See the proof of \cite[Lemma~7.3]{OS23}, with $d$-dimensional cubes in place of squares.
\end{proof}

For $f:[0,m]\to\R$ and $0\le a<b\le m$, write $s_f(a,b)=(f(b)-f(a))/(b-a)$.
\begin{definition}
Given $f:[a,b]\to\R$ and numbers $\e>0,\sigma$, we say that $(f,a,b)$ is:
\begin{itemize}
    \item \emph{$(\sigma,\e)$-superlinear} if $f(x)-f(a)\ge \sigma(x-a)-\e(b-a)$ for all $x\in[a,b]$.
    \item \emph{$(\sigma,\e)$-sublinear} if $f(b)-f(x)\le \sigma(b-x)+\e(b-a)$ for all $x\in[a,b]$.
    \item \emph{$(\sigma,\e)$-linear} if $|f(x)-f(a)-\sigma(x-a)|\le \e(b-a)$ for all $x\in[a,b]$.
\end{itemize}
\end{definition}

The next lemma translates geometric properties of $P$ into inequalities for its branching function. See, for example, \cite[Lemma 8.3]{OS23} for some of the short deductions; the others are similar and omitted.
\begin{lemma} \label{lem:branching-translation}
    Write $\delta=\Delta^m$. Let $P\subset \delta \Z^d$ be a $(\Delta,m)$-uniform set with branching function $f$. Then,  writing $C_\Delta(1)$ for a constant depending only on $\Delta$ and $d$ that may differ between the implications,
    \begin{enumerate}[label=\textup{(\roman*)},ref=\roman*]
        \item $P$ is $(\delta,t,C_\Delta(1)\delta^{-\e})$-Frostman if and only if $(f,0,m)$ is $(t,\e)$-superlinear;
        \item $P$ is $(\delta,t,C_\Delta(1)\delta^{-\e})$-Katz--Tao if and only if $(f,0,m)$ is $(t,\e)$-sublinear;
        \item\label{it:branching-linear-regular} If $(f,0,m)$ is $(t,\e)$-linear, then $P$ is $(\delta,t,C_\Delta(1)\delta^{-2\e})$-regular.
    \end{enumerate}
\end{lemma}

\section{Proof of Theorem \ref{thm:furst-general}}

\subsection{A self-improvement formulation}

Fix $1\le k<d$. We prove Theorem \ref{thm:furst-general} by improving any estimate with a positive loss in the exponent. We record the estimate to be improved.
\begin{terminology}[$ \mathsf{Furst}(s;\eta) $]
    Let $\mathsf{Furst}(s;\eta)$ denote the following statement: for every $\rho>0$, there is $\e=\e(s,\eta,\rho)>0$ such that, for all dyadic $\delta\ll_{s,\eta,\rho,\e}1$, if the configuration $(P,\mathbf\Theta)$ is $(\delta,s,\delta^{-\e},M)$-spread-out, and $P$ is $(\delta,d-k,\delta^{-\e})$-Frostman, then $|\mathcal F(P,\mathbf\Theta)|_\delta\ge \delta^{\eta+\rho-(d-k)}M$.
\end{terminology}


The assertion $\mathsf{Furst}(s;0)$ is the case $t=d-k$ of Theorem~\ref{thm:furst-general}. By Lemma \ref{lem:reduction-to-self-improvement}, it will be enough to improve $\mathsf{Furst}(s;\eta)$: we must show that $\mathsf{Furst}(s;\eta)$ implies $\mathsf{Furst}(s;\eta')$ for some $0\le\eta'<\eta$ depending on $s$ and $\eta$.

\subsection{Katz--Tao sets}
\label{s:frostman-to-katz-tao}

We first pass from $\mathsf{Furst}(s;\eta)$ to an estimate for $(\delta,d-k)$-Katz--Tao sets, as defined in \S\ref{s:discretized-formulations}. The idea is to take enough random translates of the point set to obtain a Frostman set. The argument is similar to \cite[Lemma 3.1]{DemeterWang25}.
\begin{lemma} \label{lem:katz-tao-induction}
    Assume $\mathsf{Furst}(s;\eta)$. For every $\rho>0$, there is $\e=\e(s,\eta,\rho)>0$ such that the following holds whenever $\delta\in2^{-\N}$ satisfies $\delta\ll_{s,\eta,\rho,\e}1$, $(P,\mathbf\Theta)$ is a $(\delta,s,\delta^{-\e},M)$-spread-out configuration, and $P$ is a $(\delta,d-k,\delta^{-\e})$-Katz--Tao set:
    \begin{equation}\label{eq:katz-tao-induction-closed}
        |\mathcal F(P,\mathbf\Theta)|_\delta
        \gtrsim \delta^{\eta+\rho}|P|M.
    \end{equation}
\end{lemma}
\begin{proof}    
    Choose $0<\e<\rho/2$ so that $\mathsf{Furst}(s;\eta)$ applies with $\rho/2$ in place of $\rho$ and Frostman and spread-out constants $\delta^{-3\e}$.

    Let $R=\lceil |P|^{-1}\delta^{-(d-k)}\rceil$, choose $v_1,\ldots,v_R$ independently and uniformly from $\delta\Z^d\cap[0,1]^d$, and set $X=\bigcup_{\ell=1}^R(P+v_\ell)$. If $R=1$, then $|X|\ge\delta^{-(d-k)}/2$. If $R>1$, the expected number of pairs of distinct representations of the same point is at most $O(R^2|P|^2\delta^d)=O(\delta^kR|P|)$, so Markov's inequality yields $|X|\gtrsim R|P|\gtrsim\delta^{-(d-k)}$ with probability at least $9/10$. For a dyadic cube $Q$ meeting $[0,2]^d$ and having side length $r\in[\delta,1]$, the sum $\sum_{\ell=1}^R|(P+v_\ell)\cap Q|$ has expectation and summands bounded by $O(\delta^{-\e}(r/\delta)^{d-k})$. By Chernoff's inequality \cite[Theorem 1.8]{MR2289012} and a union bound, there is a realization such that
    \begin{equation}\label{eq:katz-tao-random-translation-bounds}
        |X|\gtrsim R|P|,
        \qquad
        |X\cap Q|\le\delta^{-2\e}(r/\delta)^{d-k}
    \end{equation}
    simultaneously for all such $Q$.

    Choose a unit dyadic cube $Q_0$ containing at least $3^{-d}|X|$ points of $X$, and let $X_0=(X\cap Q_0)-z(Q_0)$. For each $x\in X_0$, choose a representation $x+z(Q_0)=p+v_\ell$ and set $\mathbf\Theta_0(x)=\mathbf\Theta(p)$. The choice of $Q_0$ and \eqref{eq:katz-tao-random-translation-bounds} show that $X_0$ is $(\delta,d-k,\delta^{-3\e})$-Frostman. The plane families are unchanged, so $(X_0,\mathbf\Theta_0)$ is $(\delta,s,\delta^{-\e},M)$-spread-out. Moreover, $\mathcal F(X_0,\mathbf\Theta_0)$ is covered by the translates $V\mapsto V+v_\ell-z(Q_0)$ of $\mathcal F(P,\mathbf\Theta)$, so $|\mathcal F(X_0,\mathbf\Theta_0)|_\delta\lesssim R|\mathcal F(P,\mathbf\Theta)|_\delta$. We have $|\mathcal F(X_0,\mathbf\Theta_0)|_\delta\ge \delta^{\eta+\rho/2-(d-k)}M$ by $\mathsf{Furst}(s;\eta)$. Since $R\le 1+|P|^{-1}\delta^{-(d-k)}$ and the Katz--Tao condition yields $|P|\lesssim\delta^{-\e-(d-k)}$, we have $\delta^{-(d-k)}/R\gtrsim\delta^\e|P|$. Hence $|\mathcal F(P,\mathbf\Theta)|_\delta\gtrsim \delta^{\eta+\rho/2+\e}|P|M$, establishing \eqref{eq:katz-tao-induction-closed} since $\e<\rho/2$.
\end{proof}

\subsection{Self-improvement implies the general Furstenberg estimate}

The self-improvement property $\mathsf{Furst}(s;\eta)\implies\mathsf{Furst}(s;\eta')$ for some $0\le\eta'<\eta$ leads to Theorem \ref{thm:furst-general} by a closedness argument. 

\begin{lemma} \label{lem:reduction-to-self-improvement}
    Assume that, for every $s\in(0,1]$ and $\eta\in(0,d-k]$, there exists $\eta'=\eta'(\eta,s)\in[0,\eta)$ such that $\mathsf{Furst}(s;\eta)\implies\mathsf{Furst}(s;\eta')$. Then Theorem \ref{thm:furst-general} holds.
\end{lemma}
\begin{proof}

Fix $s\in(0,1]$. The trivial bound $|\mathcal F|_\delta\gtrsim M$ establishes $\mathsf{Furst}(s;d-k)$. By the definition of $\mathsf{Furst}(s;\eta)$, the exponents $\eta\ge0$ for which $\mathsf{Furst}(s;\eta)$ holds form a nonempty closed upper interval. Self-improvement forces its minimum to be zero. Thus $\mathsf{Furst}(s;0)$ holds.
 
Fix $t\in(0,d-k]$ and $\eta>0$, and choose $0<\e\ll_{s,t,\eta}1$. Let $(P,\mathbf\Theta)$ satisfy the hypotheses of Theorem \ref{thm:furst-general}. By \cite[Lemma~3.9]{SpreadOutFurstSticky} with $K=2$, there is a piece $P'$ which is $(\delta,t,\delta^{-O(\e)})$-Frostman and satisfies $\delta^{-t+O(\e)}\le |P'|_\delta\le\delta^{-t}$. Since $t\le d-k$, the set $P'$ is $(\delta,d-k,\delta^{-O(\e)})$-Katz--Tao. Restricting the shadings to $P'$ and applying Lemma \ref{lem:katz-tao-induction} with $\eta=0$ and $\eta/2$ in place of $\rho$, we obtain
\[
 |\mathcal F(P,\mathbf\Theta)|_\delta
 \ge |\mathcal F(P',\mathbf\Theta|_{P'})|_\delta
 \gtrsim\delta^{\eta/2}|P'|_\delta M
 \ge\delta^{\eta/2-t+O(\e)}M.
\]
For $\e\ll_\eta1$ and then $\delta\ll_{s,t,\eta,\e}1$, we conclude that $|\mathcal F(P,\mathbf\Theta)|_\delta\ge\delta^{\eta-t}M$.
\end{proof}

\subsection{Reduction to the uniform non-sticky case}

\label{s:furst-nonsticky-reduction}

The goal of this section is to reduce Theorem~\ref{thm:furst-general} to the case where $P$ is uniform and non-sticky in a quantitative sense, as described in the following proposition.
\begin{proposition}
\label{prop:non-sticky-furst-gain}

Assume $\mathsf{Furst}(s;\eta)$ holds and fix $\xi>0$. Then, there exists $\eta'=\eta'(s,\eta,\xi)\in[0,\eta)$ such that, for all $\rho>0$, there is $\e=\e(s,\eta,\rho,\xi)>0$ such that the following holds.  Fix $\Delta\in 2^{-\N}$ such that 
\[
\log_{1/\Delta}(4(d+1)\log(1/\Delta))<\e/2.
\]
Let $m\gg_{\eta,\rho,\xi,\e,\Delta} 1$ and set $\delta=\Delta^m$. Let $(P,\mathbf\Theta)$ be a $(\delta,s,\delta^{-\e},M)$-spread-out configuration, where $P$ is $(\Delta,m)$-uniform and $(\delta,d-k,\delta^{-\e})$-Frostman. Let $f$ be the branching function of $P$, and suppose that
\begin{equation} \label{eq:Frostman-and-KT}
f(m)  \le (d-k)m + O_{\Delta}(1),
\end{equation}
and
\begin{equation} \label{eq:non-sticky}
f(n)-(d-k)n \ge \xi m \quad\text{for some } n\in[0,m].
\end{equation}
Then,
\[
|\mathcal F(P,\mathbf\Theta)|_\delta \ge \delta^{\eta'+\rho-(d-k)}M.
\] 
\end{proposition}

In the rest of this subsection we show that Proposition~\ref{prop:non-sticky-furst-gain} implies Theorem~\ref{thm:furst-general}. Proposition~\ref{prop:non-sticky-furst-gain} then be proved in the rest of this section. Figure \ref{fig:sticky-vs-non-sticky} illustrates the main idea: either $|f(n)-(d-k)n|\le \xi m$ for all $n\in[0,m]$, in which case we can apply the sticky theorem \cite[Theorem~1.9]{SpreadOutFurstSticky}, or $f(n)-(d-k)n>\xi m$ for some $n\in[0,m]$, in which case we can apply Proposition~\ref{prop:non-sticky-furst-gain}.

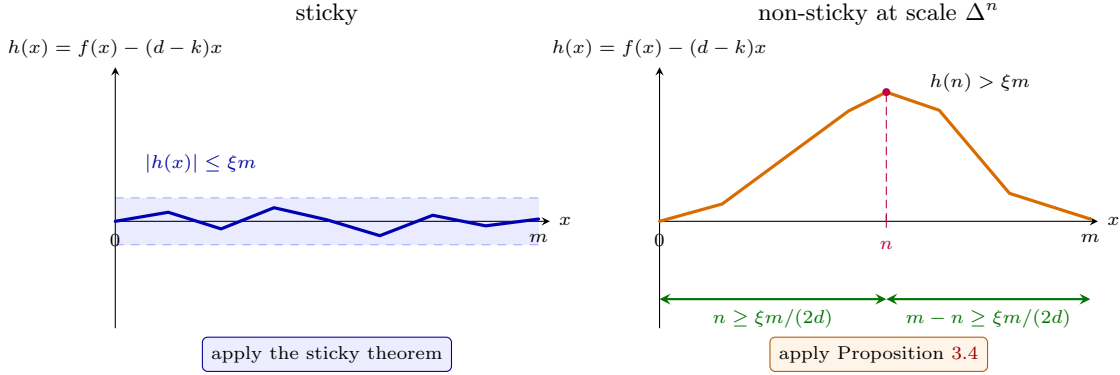
\begin{figure}[htbp]
\centering
\begin{tikzpicture}[x=1cm,y=1cm,>=stealth,line cap=round,line join=round]
    \node[font=\small] at (3.45,5.05) {sticky};
    \node[font=\small] at (10.75,5.05) {non-sticky at scale $\Delta^n$};

    \fill[blue!8] (0.65,2.00) rectangle (6.25,2.62);
    \draw[blue!35,dashed] (0.65,2.00) -- (6.25,2.00);
    \draw[blue!35,dashed] (0.65,2.62) -- (6.25,2.62);
    \draw[->] (0.65,2.31) -- (6.40,2.31) node[right,font=\scriptsize] {$x$};
    \draw[->] (0.65,0.90) -- (0.65,4.35) node[above,font=\scriptsize] {$h(x)=f(x)-(d-k)x$};
    \draw[blue!70!black,very thick] plot coordinates {
        (0.65,2.31) (1.35,2.43) (2.05,2.21) (2.75,2.49)
        (3.45,2.33) (4.15,2.12) (4.85,2.39) (5.55,2.25) (6.25,2.34)};
    \node[font=\scriptsize,blue!65!black,anchor=west,fill=white,inner sep=1pt] at (1.00,3.08)
        {$|h(x)|\le\xi m$};
    \node[font=\scriptsize,below] at (0.65,2.27) {$0$};
    \node[font=\scriptsize,below] at (6.25,2.27) {$m$};
    \node[draw=blue!55!black,fill=blue!7,rounded corners=2pt,
        font=\scriptsize] at (3.45,0.55) {apply the sticky theorem};

    \draw[->] (7.85,2.31) -- (13.65,2.31) node[right,font=\scriptsize] {$x$};
    \draw[->] (7.85,0.90) -- (7.85,4.35) node[above,font=\scriptsize] {$h(x)=f(x)-(d-k)x$};
    \draw[orange!85!black,very thick] plot coordinates {
        (7.85,2.31) (8.68,2.54) (9.48,3.13) (10.35,3.77)
        (10.85,4.02) (11.55,3.78) (12.48,2.68) (13.55,2.34)};
    \draw[purple,densely dashed] (10.85,2.31) -- (10.85,4.02);
    \node[font=\scriptsize,anchor=west,fill=white,inner sep=2pt] at (11.35,4.15)
        {$h(n)>\xi m$};
    \fill[purple] (10.85,4.02) circle (1.5pt);
    \draw[purple] (10.85,2.23) -- (10.85,2.39);
    \node[font=\scriptsize,below] at (7.85,2.27) {$0$};
    \node[font=\scriptsize,below,purple] at (10.85,2.17) {$n$};
    \node[font=\scriptsize,below] at (13.55,2.27) {$m$};
    \draw[<->,green!45!black,thick] (7.85,1.28) -- (10.85,1.28)
        node[midway,below,font=\scriptsize] {$n\ge\xi m/(2d)$};
    \draw[<->,green!45!black,thick] (10.85,1.28) -- (13.55,1.28)
        node[midway,below,font=\scriptsize] {$m-n\ge\xi m/(2d)$};
    \node[draw=orange!70!black,fill=orange!9,rounded corners=2pt,
        font=\scriptsize] at (10.75,0.55) {apply Proposition \ref{prop:non-sticky-furst-gain}};
\end{tikzpicture}
\caption{If $f(x)$ stays uniformly close to $(d-k)x$ we apply the sticky Theorem; otherwise, there is a scale $\Delta^n$ at which $f(n)$ is much larger than $(d-k)n$, and Proposition~\ref{prop:non-sticky-furst-gain} applies.}
\label{fig:sticky-vs-non-sticky}
\end{figure}

\begin{proof} [Proof of Theorem~\ref{thm:furst-general} assuming Proposition~\ref{prop:non-sticky-furst-gain}]

Recall the definitions of $(\Delta,m)$-uniform sets and branching functions in \S\ref{s:uniformization}.

Assume $\mathsf{Furst}(s;\eta)$ holds. By Lemma~\ref{lem:reduction-to-self-improvement}, it suffices to show that $\mathsf{Furst}(s;\eta')$ holds for some $0\le\eta'<\eta$ depending on $s$ and $\eta$.

By the sticky theorem \cite[Theorem~1.9]{SpreadOutFurstSticky}, there exists $\xi=\xi(s,\eta/2,d-k)>0$ such that, for all $\delta\ll_{s,\eta}1$, if $(P,\mathbf\Theta)$ is a $(\delta,s,\delta^{-3\xi},M)$-spread-out configuration, and $P$ is $(\delta,d-k,\delta^{-3\xi})$-regular, then
\begin{equation} \label{eq:sticky-gain}
    |\mathcal F(P,\mathbf\Theta)|_\delta\ge \delta^{\eta/2-(d-k)}M.
\end{equation}

Let $\eta'$ be the maximum of $\eta/2$ and the exponent provided by Proposition~\ref{prop:non-sticky-furst-gain} for this $\xi$; then $\eta'\in[\eta/2,\eta)$ depends on $d,s,\eta$ only. Fix $\rho>0$, and let $\e$ be the minimum of $\xi$ and the value provided by Proposition~\ref{prop:non-sticky-furst-gain}. Choose $\Delta\in 2^{-\N}$ such that
\[
\log_{1/\Delta}(4(d+1)\log(1/\Delta))<\e/2. 
\]
Let $(P,\mathbf\Theta)$ be a $(\delta,s,\delta^{-\e},M)$-spread-out configuration, where $P$ is $(\delta,d-k,\delta^{-\e/2})$-Frostman. We may assume without loss of generality that $\delta=\Delta^m$ for some $m\gg_{\eta,\rho,\xi,\e,\Delta}1$ (interpolating between powers of $\Delta$ incurs a harmless $O_\Delta(1)$ loss). 

By Lemma \ref{lem:uniformization} and our choice of $\Delta$, the set $P$ contains a $(\Delta,m)$-uniform subset $P_0$ with $|P_0|\ge\delta^{\e/2}|P|$. Therefore $P_0$ is $(\delta,d-k,\delta^{-\e})$-Frostman. If $|P_0|_{\delta}\le \delta^{-(d-k)}$, let $P'=P_0$. Otherwise, let $\Delta^j$ be the largest scale such that $|P_0|_{\Delta^j}\ge\delta^{-(d-k)}$, and choose one point of $P_0$ from each cube in $\mathcal D_{\Delta^j}(P_0)$ to form $P'$. Then $P'$ is $(\Delta,m)$-uniform and $(\delta,d-k,\delta^{-\e})$-Frostman, with $\delta^{-(d-k)}\le|P'|_\delta\le\Delta^{-d}\delta^{-(d-k)}$.

Let $f$ be the branching function of $P'$. Then \eqref{eq:Frostman-and-KT} holds for $f$ because $|P'|_{\delta}\lesssim_{\Delta} \delta^{-(d-k)}$. 

If $f(x)-(d-k)x\le \frac{3}{2}\xi m$ for every $x\in[0,m]$, then the Frostman condition on $P'$ and the definition of the branching function show that $P'$ is $(\delta,d-k,\delta^{-3\xi})$-regular if $m\gg 1$. Hence \eqref{eq:sticky-gain} applies to $(P',\mathbf\Theta|_{P'})$, and
\[
 |\mathcal F(P,\mathbf\Theta)|_\delta \ge |\mathcal F(P',\mathbf\Theta|_{P'})|_\delta \ge \delta^{\eta/2-(d-k)}M.
\]
Since $\eta'\ge\eta/2$, this establishes the bound in $\mathsf{Furst}(s;\eta')$ under the sticky assumption on $P'$. 

Otherwise, the non-sticky condition \eqref{eq:non-sticky} holds for $P'$, and Proposition~\ref{prop:non-sticky-furst-gain} applies to $(P',\mathbf\Theta|_{P'})$ to establish the bound in $\mathsf{Furst}(s;\eta')$.

\end{proof}

\subsection{Gains away from the critical exponent}

The critical size for the point set $P$ is $\delta^{-(d-k)}$. Assuming $\mathsf{Furst}(s;\eta)$, we get a gain in the Furstenberg bound for Frostman sets with $|P|_{\delta}\ge\delta^{-t}$ for some $t>d-k$, and for Katz--Tao sets with $|P|_{\delta}\le\delta^{-t}$ for some $t<d-k$. This will be a key ingredient in the proof of Proposition \ref{prop:non-sticky-furst-gain}.
\begin{lemma}\label{lem:away-from-critical}
Let $s\in(0,1]$ and $\eta>0$, and assume $\mathsf{Furst}(s;\eta)$.
\begin{enumerate}[label=\textup{(\roman*)},ref=\roman*]
    \item\label{it:frostman-away-from-critical}
    Let $t\in(d-k,d]$. There exists $\e>0$ such that the following holds whenever $\Delta\in2^{-\N}$ satisfies $\Delta\ll_{s,\eta,t}1$ and $m\gg_{s,\eta,t,\e,\Delta}1$. Write $\delta=\Delta^m$. If $(P,\mathbf\Theta)$ is a $(\delta,s,\delta^{-\e},M)$-spread-out configuration, where $P$ is a $(\delta,d-k,\delta^{-\e})$-Frostman set satisfying $|P|_\delta\ge\delta^{-t}$, then
    \begin{equation}\label{eq:frostman-away-conclusion}
        |\mathcal F(P,\mathbf\Theta)|_\delta
        \ge \delta^{\eta-\frac{3\eta(t-d+k)}{8k}-(d-k)}M.
    \end{equation}
    \item\label{it:katz-tao-away-from-critical}
    Let $t\in(0,d-k)$. There exists $\e>0$ such that the following holds whenever $\Delta\in2^{-\N}$ satisfies $\Delta\ll_{s,\eta,t}1$ and $m\gg_{s,\eta,t,\e,\Delta}1$. Write $\delta=\Delta^m$. If $(P,\mathbf\Theta)$ is a $(\delta,s,\delta^{-\e},M)$-spread-out configuration, where $P$ is a $(\delta,d-k,\delta^{-\e})$-Katz--Tao set satisfying $|P|_\delta\le\delta^{-t}$, then
    \begin{equation}\label{eq:katz-tao-away-conclusion}
        |\mathcal F(P,\mathbf\Theta)|_\delta
        \ge \delta^{\eta-\frac{3\eta(d-k-t)}{8(d-k)}}|P|_\delta M.
    \end{equation}
\end{enumerate}
\end{lemma}

Before proving the lemma, we recall some preliminaries. 

We use the following form of the Lipschitz-function decomposition in \cite[Lemma 8.5]{OS23}.
\begin{lemma} \label{lem:combinatorial-mixed-super-linear}
    Fix $d\in\N$ and $0<s<t<d$. For every $\e>0$ there is $\tau=\tau(\e,s,t)>0$ such that the following holds: for any non-decreasing $d$-Lipschitz function $f:[0,m]\to\R$ such that $f(0)=0$ and $f(x)\ge tx-\e m$ for every $x\in[0,m]$, there exist a sequence of intervals $\{[a_j,b_j]\}_{j=1}^{n}\subset [0,m]$ with $b_{j}\le a_{j+1}$ for all $1\le j<n$, such that:
\begin{enumerate}[label=\textup{(\roman*)},ref=(\roman*)]
    \item\label{it:combinatorial-mixed-block-types} For each $j$, one of the following alternatives holds:
        \begin{enumerate}[label=\textup{(\alph*)},ref=(\alph*)]
            \item $s_f(a_{j},b_{j}) =: t_j\ge s$ and $(f,a_{j},b_{j})$ is $(t_j,\e)$-linear.
            \item $s_f(a_{j},b_{j}) = s$ and $(f,a_{j},b_{j})$ is $(s,\e)$-superlinear.
        \end{enumerate}
    \item\label{it:combinatorial-mixed-length} $b_j - a_j \ge \tau m$ for all $j$.
    \item\label{it:combinatorial-mixed-coverage}
    \[
        \bigl|[0,m]\setminus \cup_j[a_j,b_j]\bigr|
        \lesssim \frac{\e}{t-s}m.
    \]
    \end{enumerate}
\end{lemma}
We comment on the bound in \ref{it:combinatorial-mixed-coverage}, which is not explicitly stated in \cite[Lemma 8.5]{OS23}. In the construction in \cite{OS23}, the initial segment $[0,a]$ is discarded only when  $a\le \e m/(t-s)$. The remaining discarded pieces have total length $O(\e m)$. This yields \ref{it:combinatorial-mixed-coverage}.

For sticky point sets, \cite[Theorem~1.9]{SpreadOutFurstSticky} and compactness yield a Furstenberg bound uniform over compact intervals of exponents, which we record in the following lemma.
\begin{lemma}\label{lem:ahlfors-uniform-exponent}
    Let $s\in(0,1]$, let $I\subset(0,d]$ be a non-empty compact interval, and let $\eta>0$. There is $\e=\e(s,I,\eta)>0$ such that the following holds for every $t\in I$ and $\delta\in2^{-\N}$ with $0<\delta\ll_{s,I,\eta,\e}1$. If $(P,\mathbf\Theta)$ is a $(\delta,s,\delta^{-\e},M)$-spread-out configuration and $P$ is $(\delta,t,\delta^{-\e})$-regular, then $|\mathcal F(P,\mathbf\Theta)|_\delta\ge\delta^{\eta-\min\{t,d-k\}}M$.
\end{lemma}

\begin{proof}[Proof of Lemma \ref{lem:away-from-critical}]
    In both cases we assume $P\subset [0,1)^d$ without loss of generality.

\textup{(\ref{it:frostman-away-from-critical})}
    The gain from $t>d-k$ is analogous to the gain for Frostman tube families with more than $\delta^{-2}$ tubes in \cite[\S8]{GuthWangZahl26}. Let $\rho=\eta(t-d+k)/(8k)$ and note that $\frac{d+k-t}{2k}\eta+\rho=\eta-3\eta(t-d+k)/(8k)<\eta$.

    Choose $0<\zeta\ll_{s,t,\eta,\rho}1$. Let $\tau>0$ be the parameter in Lemma \ref{lem:combinatorial-mixed-super-linear}\textup{\ref{it:combinatorial-mixed-length}} for slopes $d-k-\zeta$ and $d-k$ and error $\zeta^2$. Set $\e=\tau\zeta^2/100$. Take $\Delta\ll_{s,\eta,t}1$ so that \cite[Corollary~5.9]{SpreadOutFurstSticky} applies with $J_0=\lceil20/\tau+20\rceil$ and parameter $\e$, and $\log_{1/\Delta}(4(d+1)\log(1/\Delta))\le\e$.

    Apply Lemma \ref{lem:uniformization} to obtain a $(\Delta,m)$-uniform subset $P'\subset P$. For $m\gg_{s,\eta,t,\e,\Delta}1$, the choices of $\Delta$ and $\e$ ensure that $|P'|_\delta\ge\delta^{\zeta^2-t}$ and that $P'$ is $(\delta,d-k,\delta^{-\zeta^2})$-Frostman. By Lemma \ref{lem:branching-translation}, its branching function $f$ satisfies
    \begin{equation}\label{eq:frostman-away-uniformized-branching}
        f(x)\ge(d-k)x-\zeta^2m\quad(0\le x\le m),
        \qquad f(m)\ge(t-\zeta^2)m.
    \end{equation}
    Lemma \ref{lem:combinatorial-mixed-super-linear}\ref{it:combinatorial-mixed-coverage} provides a decomposition of all but $O(\zeta m)$ of $[0,m]$ into disjoint intervals $[a_j,b_j]$. By \ref{it:combinatorial-mixed-length}, we have $b_j-a_j\ge\tau m$, and \ref{it:combinatorial-mixed-block-types} distinguishes two types of interval:
    \begin{enumerate}[label=\textup{(\alph*)},ref=(\alph*)]
        \item\label{it:frostman-away-superlinear-block} $|s_f(a_j,b_j)-(d-k-\zeta)|\le\zeta^2$ and $(f,a_j,b_j)$ is $(d-k,O(\zeta))$-superlinear;
        \item\label{it:frostman-away-linear-block} $(f,a_j,b_j)$ is $(t_j,O(\zeta))$-linear for some $t_j\ge d-k-2\zeta$.
    \end{enumerate}
    Let $\mathcal S$ denote the indices satisfying \ref{it:frostman-away-superlinear-block}. The lower bound on $f(m)$ in \eqref{eq:frostman-away-uniformized-branching} bounds the total length of the intervals with $j\in\mathcal S$. Using their slopes and the $d$-Lipschitz bound on their complement, we obtain
    \begin{equation}\label{eq:frostman-away-critical-length}
        \sum_{j\in\mathcal S}(b_j-a_j)
        \le\frac{d-t}{k}m+O(\zeta m).
    \end{equation}
    Rounding the endpoints, we may assume that $a_j,b_j\in\Z$.

    Apply \cite[Corollary~5.9]{SpreadOutFurstSticky}, with $J_0=\lceil20/\tau+20\rceil$ and parameter $\e$, to the increasing list formed by $0,m$ and the distinct scales $a_j,b_j$. Denote the intermediate configuration associated with $[a_j,b_j]$ by $(P_j,\mathbf\Theta_j)$, its shading parameter by $M_j$, and its Furstenberg set by $\mathcal F_j$, and write $r_j=\Delta^{b_j-a_j}$. The ratios in \cite[Eq. (5.10)]{SpreadOutFurstSticky} corresponding to the complementary scale intervals are $\gtrsim1$ and may be omitted. Hence
    \begin{equation}\label{eq:frostman-away-product}
        \frac{|\mathcal F(P',\mathbf\Theta|_{P'})|_\delta}{M}
        \gtrapprox_\delta
        \prod_j\frac{|\mathcal F_j|_{r_j}}{M_j}.
    \end{equation}

    The two types of interval play different roles. On intervals of type \ref{it:frostman-away-superlinear-block}, we use the Katz--Tao bound in Lemma \ref{lem:katz-tao-induction}. On intervals of type \ref{it:frostman-away-linear-block}, we use the sticky bound in Lemma \ref{lem:ahlfors-uniform-exponent}. The gain on the linear intervals, together with \eqref{eq:frostman-away-critical-length}, will improve the product in \eqref{eq:frostman-away-product}.

    Let $g_j$ be the branching function of $P_j$. The choice of $\e$ and the bound $b_j-a_j\ge\tau m$ imply
    \begin{equation}\label{eq:frostman-away-intermediate-spread-out}
        \mathbf\Theta_j(p)\text{ is }(r_j,s,r_j^{-\zeta^2})_{d,k}\text{-spread-out}
        \qquad(p\in P_j).
    \end{equation}
    We also have, by \cite[(5.11)]{SpreadOutFurstSticky},
    \begin{equation}\label{eq:frostman-away-intermediate-branching}
        |g_j(i)-f(a_j+i)+f(a_j)|
        \le\zeta (b_j-a_j),
        \qquad 0\le i\le b_j-a_j.
    \end{equation}

    If $j\in\mathcal S$, we deduce from condition \ref{it:frostman-away-superlinear-block} that
    \[
    \begin{aligned}
        g_j(b_j-a_j)-g_j(i)&\le(d-k)(b_j-a_j-i)+O(\zeta)(b_j-a_j)
            &&(0\le i\le b_j-a_j),\\
        g_j(b_j-a_j)&\ge(d-k-O(\zeta))(b_j-a_j).
    \end{aligned}
    \]
    Thus
    \begin{equation}\label{eq:frostman-away-superlinear-configuration}
        P_j\text{ is }(r_j,d-k,r_j^{-O(\zeta)})\text{-Katz--Tao},
        \qquad |P_j|\ge r_j^{-(d-k)+O(\zeta)}.
    \end{equation}
    By \eqref{eq:frostman-away-intermediate-spread-out} and \eqref{eq:frostman-away-superlinear-configuration}, Lemma \ref{lem:katz-tao-induction} applies with $\rho/10$ in place of $\rho$. Hence
    \[
        \frac{|\mathcal F_j|_{r_j}}{M_j}
        \gtrsim r_j^{\eta+\rho/10-(d-k)+O(\zeta)}.
    \]
    If $j\notin\mathcal S$, condition \ref{it:frostman-away-linear-block} yields
    \[
        |g_j(i)-t_j i|\le O(\zeta)(b_j-a_j),
        \qquad 0\le i\le b_j-a_j,
    \]
    so $P_j$ is $(r_j,t_j,r_j^{-O(\zeta)})$-regular by Lemma \ref{lem:branching-translation}\textup{(\ref{it:branching-linear-regular})}. For $\zeta\ll_{s,t,\eta,\rho}1$, Lemma \ref{lem:ahlfors-uniform-exponent} applies with $I=[(d-k)/2,d]$, exponent $t_j$, and $\rho$ in place of $\eta$, using \eqref{eq:frostman-away-intermediate-spread-out}. Since $t_j\ge d-k-2\zeta$, we obtain
    \[
        \frac{|\mathcal F_j|_{r_j}}{M_j}
        \ge r_j^{\rho-(d-k)+O(\zeta)}.
    \]

    The $O(\zeta (b_j-a_j))$ terms in the lower bounds for $|\mathcal F_j|_{r_j}/M_j$, the contribution of $[0,m]\setminus\bigcup_j[a_j,b_j]$, and the polylogarithmic loss in \eqref{eq:frostman-away-product} contribute $O_{t,\eta,\rho}(\zeta)$ to the exponent of $\delta=\Delta^m$. Using \eqref{eq:frostman-away-critical-length} and taking $\zeta\ll_{t,\eta,\rho}1$, we conclude \eqref{eq:frostman-away-conclusion}.

\medskip
\noindent\textup{(\ref{it:katz-tao-away-from-critical})}
If $\eta\ge d-k$, then
\[
    \eta-\frac{3\eta(d-k-t)}{8(d-k)}-t
    \ge\frac{5(d-k-t)}8>0,
\]
so $|P|_\delta\le\delta^{-t}$ and $|\mathcal F(P,\mathbf\Theta)|_\delta\ge M$ already establish \eqref{eq:katz-tao-away-conclusion}. Assume henceforth that $0<\eta<d-k$.

We again separate linear and superlinear intervals, this time after reversing the branching function. Let $\rho=\eta(d-k-t)/(8(d-k))$, so $0<\rho<d$. Choose $\zeta,\tau,\e$, and $\Delta$ as in part~\textup{(\ref{it:frostman-away-from-critical})}, with $I=[\rho,d]$ in Lemma \ref{lem:ahlfors-uniform-exponent} and slopes $k-\zeta$ and $k$ in Lemma \ref{lem:combinatorial-mixed-super-linear}.

Apply Lemma \ref{lem:uniformization} to obtain a $(\Delta,m)$-uniform subset $P'\subset P$. For $m\gg_{s,\eta,t,\e,\Delta}1$, we have $|P'|_\delta\ge\delta^{\zeta^2}|P|_\delta$, and $P'$ is $(\delta,d-k,\delta^{-\zeta^2})$-Katz--Tao. Its branching function $f$ satisfies
    \begin{equation}\label{eq:katz-tao-away-uniformized-branching}
        f(m)-f(x)\le(d-k)(m-x)+\zeta^2m
        \quad(0\le x\le m),
        \qquad f(m)\le tm.
    \end{equation}

    Apply Lemma \ref{lem:combinatorial-mixed-super-linear} to $dy-f(m)+f(m-y)$ and reflect the resulting intervals about $m/2$. If $\mathcal S$ denotes the reflected superlinear blocks, then
    \begin{equation}\label{eq:katz-tao-away-block-properties}
    \begin{aligned}
        s_f(a_j,b_j)&=d-k+\zeta,\quad
        (f,a_j,b_j)\text{ is }(d-k+\zeta,\zeta^2)\text{-sublinear}
        &&(j\in\mathcal S),\\
        (f,a_j,b_j)&\text{ is }(t_j,\zeta^2)\text{-linear for some }t_j\le d-k+\zeta
        &&(j\notin\mathcal S).
    \end{aligned}
    \end{equation}
    Since $f$ is non-decreasing, we deduce from \eqref{eq:katz-tao-away-uniformized-branching} that
    \begin{equation}\label{eq:katz-tao-away-critical-length}
        \sum_{j\in\mathcal S}(b_j-a_j)
        \le\frac{t}{d-k+\zeta}m
        \le\frac{t}{d-k}m.
    \end{equation}

    Apply \cite[Corollary~5.9]{SpreadOutFurstSticky} to $0,m$ and the distinct scales $a_j,b_j$. Use the same notation as in part \textup{(\ref{it:frostman-away-from-critical})}, so that \eqref{eq:frostman-away-product}, \eqref{eq:frostman-away-intermediate-spread-out}, and \eqref{eq:frostman-away-intermediate-branching} remain valid.

    For $j\in\mathcal S$, \eqref{eq:katz-tao-away-block-properties} and \eqref{eq:frostman-away-intermediate-branching} show that
    \begin{equation}\label{eq:katz-tao-away-superlinear-configuration}
        P_j\text{ is }(r_j,d-k,r_j^{-O(\zeta)})\text{-Katz--Tao}.
    \end{equation}
    By \eqref{eq:frostman-away-intermediate-spread-out} and \eqref{eq:katz-tao-away-superlinear-configuration}, we may apply Lemma \ref{lem:katz-tao-induction} to get
    \begin{equation}\label{eq:katz-tao-away-superlinear-bound}
        \frac{|\mathcal F_j|_{r_j}}{M_j}
        \gtrsim\Delta^{(\eta+\rho/10)(b_j-a_j)-g_j(b_j-a_j)}.
    \end{equation}
    For $j\notin\mathcal S$ with $t_j\ge\rho$, equations \eqref{eq:frostman-away-intermediate-branching} and \eqref{eq:katz-tao-away-block-properties}, together with Lemma \ref{lem:branching-translation}\textup{(\ref{it:branching-linear-regular})}, show that $P_j$ is $(r_j,t_j,r_j^{-O(\zeta)})$-regular. Taking $\zeta\ll_{s,t,\eta,\rho}1$ and using \eqref{eq:frostman-away-intermediate-spread-out}, we may apply Lemma \ref{lem:ahlfors-uniform-exponent} with $\rho$ in place of $\eta$. When $t_j<\rho$, use $|\mathcal F_j|_{r_j}\ge M_j$. In either case,
    \begin{equation}\label{eq:katz-tao-away-linear-bound}
        \frac{|\mathcal F_j|_{r_j}}{M_j}
        \ge\Delta^{-g_j(b_j-a_j)+(\rho+O(\zeta))(b_j-a_j)}.
    \end{equation}
    Multiplying \eqref{eq:katz-tao-away-superlinear-bound} and \eqref{eq:katz-tao-away-linear-bound} over their respective index sets in \eqref{eq:frostman-away-product}, and using \eqref{eq:katz-tao-away-critical-length}, we find
    \[
        \frac{|\mathcal F(P',\mathbf\Theta|_{P'})|_\delta}{M}
        \gtrapprox_\delta
        \Delta^{-f(m)+[\rho+(\eta+\rho/10)t/(d-k)+O_{t,\eta,\rho}(\zeta)]m}.
    \]
    Since
    \[
        \rho+\left(\eta+\frac{\rho}{10}\right)\frac{t}{d-k}
        <\eta-3\rho,
        \qquad
        \Delta^{-f(m)}=|P'|_\delta\ge\delta^{\zeta^2}|P|_\delta,
    \]
    for $\zeta\ll_{t,\eta,\rho}1$, \eqref{eq:katz-tao-away-conclusion} follows.
\end{proof}

\subsection{Refinement combinatorics}
\label{s:entropy-setup}

Let $(P,\mathbf\Theta)$ be as in Proposition \ref{prop:non-sticky-furst-gain}. We will iteratively refine the incidence set $\mathcal I_0=\{(p,\theta):p\in P,\ \theta\in z\mathbf\Theta(p)\}$.
A subset $\mathcal{J}'\subset\mathcal{J}\subset\mathcal{I}_0$ is a \emph{refinement} of $\mathcal J$. If $|\mathcal J'|\ge\kappa|\mathcal J|$, it is a \emph{$\kappa$-refinement}, denoted $\mathcal J'\prec_\kappa\mathcal J$.

For $\mathcal J\subset\mathcal I_0$, let $P_{\mathcal J}$ be its projection to the $P$-coordinate, and write
\[
    \Theta_{\mathcal J}(p)=\{\theta:(p,\theta)\in\mathcal J\},
    \qquad
    \mathcal F_{\mathcal J}=\{V_\theta+p:(p,\theta)\in\mathcal J\}.
\]
We regard expressions in $p,\theta$ as maps on $\mathcal J$. If $X,Y$ are such maps, define
\[
    X(\mathcal J\mid Y=y)
    =\{X(p,\theta):(p,\theta)\in\mathcal J,\ Y(p,\theta)=y\}.
\]
For $0\le i\le m$, denote the smallest shading size at scale $\Delta^i$ by $M_{\mathcal J}(i)=\min_{p\in P_{\mathcal J}}|\mathcal D_{\Delta^i}(\Theta_{\mathcal J}(p))|$.

We also define
\[
    \mathcal E_{\mathcal J}(i)=
    \log_{1/\Delta}
    \frac{|\mathcal D_{\Delta^i}(\mathcal F_{\mathcal J})|}
    {M_{\mathcal J}(i)},
    \qquad 0\le i\le m.
\]
Our goal is to show that $\mathcal E_{\mathcal I_0}(m)\ge(d-k-\eta'-\rho)m$. We will prove the bound for a suitable refinement $\mathcal J$. To transfer it to $\mathcal I_0$, we need to keep $M_{\mathcal J}(m)$ roughly constant through the refinements.

\begin{definition}[Protected maps]
Let $Y$ be a map on $\mathcal{J}\subset\mathcal{I}_0$. For $\mathcal J'\subset\mathcal J$, we say that $Y$ is \emph{$\kappa$-protected from $\mathcal J$ to $\mathcal J'$} if
\[
    |Y^{-1}(y)\cap \mathcal{J}'|\ge\kappa\cdot |Y^{-1}(y)\cap\mathcal{J}| \quad\text{for every }y\in Y(\mathcal J').
\]
\end{definition}

Protection is multiplicative: if $Y$ is $\kappa_1$-protected from $\mathcal J_0$ to $\mathcal J_1$ and $\kappa_2$-protected from $\mathcal J_1$ to $\mathcal J_2$, then $Y$ is $\kappa_1\kappa_2$-protected from $\mathcal J_0$ to $\mathcal J_2$.

\begin{definition}\label{def:uniform-pair}
For $K\ge1$ and maps $X,Y$ on $\mathcal J\subset\mathcal{I}_0$, we say that $X$ is \emph{$K$-uniform over $Y$} if, for every $y\in Y(\mathcal J)$ and $x_1,x_2\in X(\mathcal J\mid Y=y)$,
\[
    \bigl| (X,Y)^{-1}(x_2,y)\cap \mathcal{J}\bigr| \le K
    |(X,Y)^{-1}(x_1,y)\cap \mathcal{J}|.
\]

\end{definition}

Refining an incidence set can greatly reduce individual fibers, even when the total loss is small. Lemma \ref{lem:protected-fiber-refinement} discards the fibers that lose too many incidences.
\begin{lemma}\label{lem:protected-fiber-refinement}
    Let $\emptyset\ne\widetilde{\mathcal J}\subset\mathcal J$, and let $Z_1,\ldots,Z_R$ be maps on $\mathcal J$. There is $\mathcal J'\prec_{1/2}\widetilde{\mathcal J}$ such that each $Z_i$ is $\frac{|\widetilde{\mathcal J}|}{2R|\mathcal J|}$-protected from $\mathcal J$ to $\mathcal J'$.
\end{lemma}

\begin{proof}
    Start with $\mathcal J_0=\widetilde{\mathcal J}$. Whenever a fiber $A=Z_i^{-1}(z)$ satisfies $0<|\mathcal J_j\cap A|<|\widetilde{\mathcal J}||A|/(2R|\mathcal J|)$, delete $\mathcal J_j\cap A$. Stop when no such fiber remains. Each fiber is deleted at most once, so the deletions for each $i$ remove at most $|\widetilde{\mathcal J}|/(2R)$ incidences. Summing over $i$ completes the proof.
\end{proof}

\subsection{Sparse scales and tree pigeonholing}
\label{s:hyperdyadic-scales}

For $0<\tau<1/10$ and $n\in\N$, define
\[
    \mathcal H_\tau(n)=
    \left\{\left\lfloor\frac{jn}{\lceil1/\tau\rceil}\right\rfloor:
    j=0,\ldots,\lceil1/\tau\rceil\right\}.
\]
Then $0,n\in\mathcal H_\tau(n)$, $|\mathcal H_\tau(n)|=O_\tau(1)$, and consecutive elements differ by at most $\tau n+1$. We will use the sparse set of scales $\mathcal H_\tau(n)$ in order to simultaneously uniformize the branching of many trees. Taking $\tau$ small enough, we will be able to interpolate to all integer scales with a small loss.


 This choice of scales is inspired by \cite[Lemma~2.13]{WangWu26TwoEnds} and \cite[\S2.3]{GuthWangZahl26}. There are few enough indices to regularize all the relevant fibers with $\Delta^{o(n)}$ loss. At the same time, the gaps are short enough to interpolate to every integer scale with a loss of at most $\Delta^{-O(\tau n)}$.

\begin{definition}
    Let $Y$ and $(X_h)_{h\in\mathcal H_\tau(n)}$ be maps on $\mathcal J\subset\mathcal I_0$. We say that $(X_h)_{h\in\mathcal H_\tau(n)}$ is an \emph{$\mathcal H_\tau(n)$-indexed $\mathcal J$-tree over $Y$} if $X_0$ is determined by $Y$, and $X_h$ is determined by $X_{h'}$ whenever $h<h'$ are consecutive in $\mathcal H_\tau(n)$.
\end{definition}

For fixed $y\in Y(\mathcal J)$, the sets $X_h(\mathcal J\mid Y=y)$ are levels of a rooted tree. The level-$0$ set is a singleton, and every other vertex has a unique parent. It is useful to think that the tree is defined also at values $h\notin \mathcal H_\tau(n)$, but we only record the sparse set of levels $\mathcal H_\tau(n)$. In the applications, the vertices are nested cubes or slablets.

The number of offspring may vary from vertex to vertex and from one value of $y$ to another. Lemma \ref{lem:hyperdyadic-fiber-regularization} refines the incidence set so that a single branching function describes the descendant counts in each family of trees, up to a factor $\Delta^{-o(n)}$.

\begin{lemma}\label{lem:hyperdyadic-fiber-regularization}
Fix $0<\tau<1/10$ and integers $n,R\ge1$. Let $\emptyset\ne\mathcal J\subset\mathcal I_0$. For each $1\le r\le R$, let $Y_r$ and $(X_h^r)_{h\in\mathcal H_\tau(n)}$ be maps on $\mathcal J$, with $(X_h^r)_h$ a tree over $Y_r$. Let $\mathcal Z$ be a finite family of maps on $\mathcal J$, and set
\[
    A=\max\left\{R,\frac{|\mathcal Z|}{|\mathcal H_\tau(n)|},
                     \frac{\log_{1/\Delta}|\mathcal J|}{n}\right\}.
\]
If $n\gg_{A,\tau,\Delta}1$, there is $\mathcal J'\prec_{\Delta^{\tau n}}\mathcal J$ such that:
\begin{enumerate}[label=\textup{(\alph*)},ref=\alph*]
    \item\label{it:hyperdyadic-fiber-comparability} the nonempty fibers of every $Y_r$ have mutually $\Delta^{-\tau n}$-comparable cardinalities;
    \item\label{it:hyperdyadic-fiber-branching} for each $r$ there is a non-decreasing $g_r:\mathcal H_\tau(n)\to\R$, with $g_r(0)=0$, such that
    \[
        \left|\log_{1/\Delta}
        \bigl|\,X_{h'}^r(\mathcal J'\mid (Y_r,X_h^r)=(y,x))\,\bigr|
        -[g_r(h')-g_r(h)]\right|\le\tau n
    \]
    whenever $h\le h'$ belong to $\mathcal H_\tau(n)$ and $(y,x)\in(Y_r,X_h^r)(\mathcal J')$;
    \item\label{it:hyperdyadic-fiber-protection} every map in
    $\mathcal Z\cup\{Y_r:1\le r\le R\}$ is $\Delta^{\tau n}$-protected from $\mathcal J$ to $\mathcal J'$.
\end{enumerate}
\end{lemma}

\begin{proof}
The joint pigeonholing of fiber cardinalities is analogous to the common branching regularization in \cite[Lemma~1.10]{WangWu24Restriction}. Pigeonhole jointly the $\mathcal J$-fiber cardinalities of $(Y_r,X_h^r)$, $1\le r\le R$, $h\in\mathcal H_\tau(n)$. Since $R|\mathcal H_\tau(n)|=O_{A,\tau}(1)$ and $|\mathcal J|\le\Delta^{-An}$, there are $(n+2)^{O_{A,\tau,\Delta}(1)}$ joint dyadic classes. Choose one determining $\widetilde{\mathcal J}\prec_{\Delta^{o_{A,\tau,\Delta}(n)}}\mathcal J$; for each $(Y_r,X_h^r)$, the $\mathcal J$ fibers meeting $\widetilde{\mathcal J}$ have cardinalities differing by a factor less than $2$.

Apply Lemma \ref{lem:protected-fiber-refinement} to $\widetilde{\mathcal J}\subset\mathcal J$ with the maps $(Y_r,X_h^r)$ and those in $\mathcal Z$. We obtain $\mathcal J'\prec_{\Delta^{o_{A,\tau,\Delta}(n)}}\mathcal J$, with every map in $\mathcal Z\cup\{(Y_r,X_h^r):1\le r\le R,\ h\in\mathcal H_\tau(n)\}$ $\Delta^{o_{A,\tau,\Delta}(n)}$-protected from $\mathcal J$ to $\mathcal J'$. Protection and the factor-$2$ comparability of the $\mathcal J$ fibers imply that the nonempty $(Y_r,X_h^r)$ fibers on $\mathcal J'$ have mutually $\Delta^{-o_{A,\tau,\Delta}(n)}$-comparable cardinalities. Since $(Y_r,X_0^r)$ and $Y_r$ have the same fibers, this establishes \textup{(\ref{it:hyperdyadic-fiber-comparability})} and \textup{(\ref{it:hyperdyadic-fiber-protection})} for $n\gg_{A,\tau,\Delta}1$.

For fixed $r$ and $h$, the numbers of level-$n$ descendants of the nonempty $(Y_r,X_h^r)$ fibers are $\Delta^{-o(n)}$-comparable. Set $g_r(n)-g_r(h)$ equal to the maximum of their $\log_{1/\Delta}$-cardinalities, with $g_r(0)=0$. By nesting, $g_r$ is non-decreasing, and partitioning level-$n$ descendants according to their $X_{h'}^r$ values proves \textup{(\ref{it:hyperdyadic-fiber-branching})}. 
\end{proof}

\subsection{Anisotropic rescaling and slablets}

Fix $\xi\in(0,1/2)$ and let $n$ satisfy the non-sticky condition \eqref{eq:non-sticky}. For $m\gg_{\xi,\Delta}1$, the Katz-Tao condition \eqref{eq:Frostman-and-KT} and the $d$-Lipschitz property of $x\mapsto f(x)-(d-k)x$ yield
\begin{equation}\label{eq:non-sticky-scale-lengths}
    \min\{n,m-n\}\ge\frac{\xi m}{2d}.
\end{equation}
For each
$\mathbf{F}\in\mathcal D_{\Delta^n}(\mathcal F_{\mathcal I_0})$, let $(\theta_{\mathbf{F}},b_{\mathbf{F}})=z(\mathbf{F})$, and define
\[
    T_{\mathbf{F}}(\theta,b)=
    \Delta^{-n}(\theta-\theta_{\mathbf{F}},b-b_{\mathbf{F}}).
\]
If $(p,\theta)\in\mathcal I_0$ and $\mathcal D_{\Delta^n}(V_\theta+p)=\mathbf{F}$, then $\theta\mapsto\Delta^{-n}(\theta-\theta_{\mathbf{F}})$ maps $\mathcal D_{\Delta^i}(\theta)$ onto a dyadic $\Delta^{i-n}$-cell for $n\le i\le m$, while $T_{\mathbf{F}}$ maps $\mathcal D_{\Delta^{m}}(V_\theta+p)$ onto a dyadic $\Delta^{m-n}$-cell containing $T_{\mathbf{F}}(\theta,\rho_\theta(p))$. We use these image cells as the local plane and affine cells inside $\mathbf{F}$.

The map $T_{\mathbf{F}}$ on affine plane parameters is induced by the following change of physical coordinates on $\R^d=\R^k\times\R^{d-k}$:
\[ 
    A_{\mathbf{F}}(x,y)=\bigl(x,\Delta^{-n}(y-\theta_{\mathbf{F}}^*x-b_{\mathbf{F}})\bigr),
    \qquad (x,y)\in\R^k\times\R^{d-k}.
\]
It sends $V_{\theta,b}$ to
$V_{\Delta^{-n}(\theta-\theta_{\mathbf{F}}),\Delta^{-n}(b-b_{\mathbf{F}})}$.

For $\mathcal J\subset\mathcal I_0$ and $(q,\mathbf F)\in(p,\mathcal D_{\Delta^n}(V_\theta+p))(\mathcal J)$, define the normalized plane family
\begin{equation}\label{eq:normalized-plane-family}
    \Xi_{\mathcal J}(q,\mathbf F)=
    \{\mathcal D_{\Delta^{m-n}}(\Delta^{-n}(\theta-\theta_{\mathbf F})):
    (q,\theta)\in\mathcal J,\ \mathcal D_{\Delta^n}(V_\theta+q)=\mathbf F\}.
\end{equation}

For $\mathbf{F}\in\mathcal D_{\Delta^n}(\mathcal F_{\mathcal I_0})$, define
\[
    S_0^{\mathbf{F}}(p)=\ast,
    \qquad
    S_i^{\mathbf{F}}(p)=\mathcal D_{\Delta^i}(A_{\mathbf{F}}(p)),
    \quad 1\le i\le m-n.
\]
On $\mathcal I_0$, write
$S_i(p,\theta)=S_i^{\mathcal D_{\Delta^n}(V_\theta+p)}(p)$.

For $(p,\theta)\in\mathcal I_0$ with $\mathcal D_{\Delta^n}(V_\theta+p)=\mathbf{F}$ and $1\le i\le m-n$, the cell $S_i^{\mathbf{F}}(p)$ contains $A_{\mathbf{F}}(p)$. Its inverse image under $A_{\mathbf{F}}$ is a slablet: it has side length $\Delta^i$ in the $k$ directions parallel to $V_{\theta_{\mathbf{F}},b_{\mathbf{F}}}$ and side length $\Delta^{n+i}$ in the $d-k$ transverse directions. We regard $S_0^{\mathbf{F}}=\ast$ as the ambient $\Delta^n$-slab $\mathbf{F}$. Thus $S_0^{\mathbf{F}},\ldots,S_{m-n}^{\mathbf{F}}$ describe nested slablets inside $\mathbf{F}$. At the finest level, $A_{\mathbf{F}}^{-1}(S_{m-n}^{\mathbf{F}}(p))$ has tangential side length $\Delta^{m-n}$ and transverse side length $\delta$. See Figure \ref{fig:slablet-geometry}.

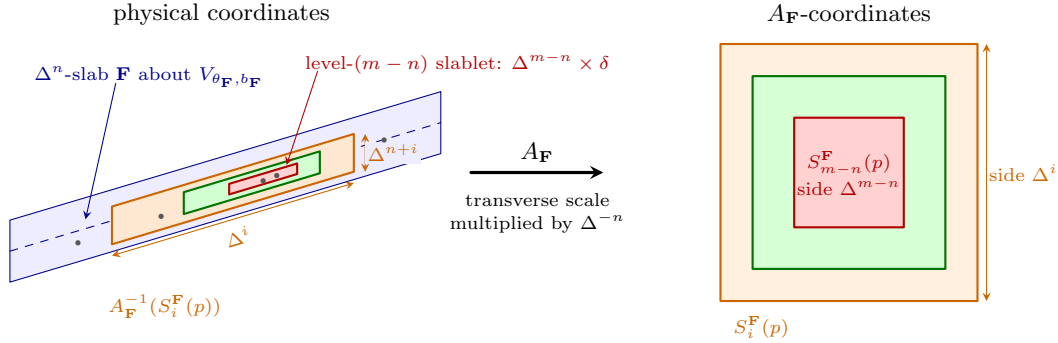
\begin{figure}[htbp]
\centering
\begin{tikzpicture}[x=1cm,y=1cm,>=stealth,line cap=round,line join=round]
    \node[font=\small] at (3.05,4.45) {physical coordinates};
    \node[font=\small] at (11.35,4.45) {$A_{\mathbf{F}}$-coordinates};

    \fill[blue!7] (0.25,0.90) -- (5.95,2.60) -- (5.95,3.42) -- (0.25,1.72) -- cycle;
    \draw[blue!55!black] (0.25,0.90) -- (5.95,2.60) -- (5.95,3.42) -- (0.25,1.72) -- cycle;
    \draw[blue!55!black,dashed] (0.25,1.31) -- (5.95,3.01);
    \node[font=\scriptsize,blue!55!black,anchor=west] at (0.45,3.62)
        {$\Delta^n$-slab $\mathbf{F}$ about $V_{\theta_{\mathbf{F}},b_{\mathbf{F}}}$};
    \draw[->,blue!55!black] (1.65,3.53) -- (1.25,1.73);

    \fill[orange!18] (1.60,1.40) -- (4.80,2.36) -- (4.80,2.86) -- (1.60,1.90) -- cycle;
    \draw[orange!80!black,thick] (1.60,1.40) -- (4.80,2.36) -- (4.80,2.86) -- (1.60,1.90) -- cycle;
    \fill[green!18] (2.55,1.80) -- (4.35,2.34) -- (4.35,2.63) -- (2.55,2.09) -- cycle;
    \draw[green!45!black,thick] (2.55,1.80) -- (4.35,2.34) -- (4.35,2.63) -- (2.55,2.09) -- cycle;
    \fill[red!18] (3.15,2.06) -- (4.05,2.33) -- (4.05,2.47) -- (3.15,2.20) -- cycle;
    \draw[red!70!black,thick] (3.15,2.06) -- (4.05,2.33) -- (4.05,2.47) -- (3.15,2.20) -- cycle;

    \node[font=\scriptsize,orange!80!black] at (2.30,0.58)
        {$A_{\mathbf{F}}^{-1}(S_i^{\mathbf{F}}(p))$};
    \draw[<->,orange!80!black] (1.60,1.25) -- (4.80,2.21)
        node[midway,below,sloped,font=\scriptsize] {$\Delta^i$};
    \draw[<->,orange!80!black] (4.95,2.36) -- (4.95,2.86)
        node[midway,right,font=\scriptsize,fill=blue!7,inner sep=0.8pt] {$\Delta^{n+i}$};
    \node[font=\scriptsize,red!70!black,anchor=west,fill=white,inner sep=0.8pt] at (4.12,3.80)
        {level-$(m-n)$ slablet: $\Delta^{m-n}\times\delta$};
    \draw[->,red!70!black] (4.32,3.69) -- (3.91,2.45);

    \foreach \x/\y in {1.15/1.42,2.25/1.77,3.60/2.24,3.78/2.31,5.20/2.78}
        \fill[black!65] (\x,\y) circle (0.035);

    \draw[->,very thick] (6.35,2.35) -- (8.10,2.35)
        node[midway,above,font=\small] {$A_{\mathbf{F}}$};
    \node[font=\scriptsize,align=center] at (7.23,1.82)
        {transverse scale\\multiplied by $\Delta^{-n}$};

    \fill[orange!12] (9.65,0.65) rectangle (13.05,4.05);
    \draw[orange!80!black,thick] (9.65,0.65) rectangle (13.05,4.05);
    \fill[green!16] (10.075,1.075) rectangle (12.625,3.625);
    \draw[green!45!black,thick] (10.075,1.075) rectangle (12.625,3.625);
    \fill[red!18] (10.625,1.625) rectangle (12.075,3.075);
    \draw[red!70!black,thick] (10.625,1.625) rectangle (12.075,3.075);
    \node[font=\scriptsize,orange!80!black,anchor=north west] at (9.70,0.58)
        {$S_i^{\mathbf{F}}(p)$};
    \draw[<->,orange!80!black] (13.17,0.65) -- (13.17,4.05)
        node[midway,right,font=\scriptsize,fill=white,inner sep=0.8pt] {side $\Delta^i$};
    \node[font=\scriptsize,red!70!black,align=center] at (11.35,2.35)
        {$S_{m-n}^{\mathbf{F}}(p)$\\side $\Delta^{m-n}$};
\end{tikzpicture}
\caption{$A_{\mathbf{F}}$ maps each level-$i$ slablet, of tangential length $\Delta^i$ and transverse width $\Delta^{n+i}$, to a square of side $\Delta^i$.}
\label{fig:slablet-geometry}
\end{figure}

\subsection{Fiber regularization at the non-sticky scale}

We now regularize the plane and slablet trees simultaneously. The refinement must preserve the shading at each point and the spread-out property of the normalized planes. Write $\mathcal H=\mathcal H_\e(m-n)$.
\begin{proposition}\label{prop:middle-scale-regularization}
  If $\Delta\ll_{\xi,\e}1$ and $m\gg_{\xi,\e,\Delta}1$, there is $\mathcal I\prec_{\delta^{O(\e)}}\mathcal I_0$ such that
    \begin{enumerate}[label=\textup{(\roman*)},ref=\roman*]
        \item\label{it:middle-scale-point-protection} The map $p$ is $\delta^{O(\e)}$-protected from $\mathcal I_0$ to $\mathcal I$.
        \item\label{it:middle-scale-fiber-comparability} The nonempty fibers of each of the maps $p$ and $(p,\mathcal D_{\Delta^n}(\theta))$ have mutually $\Delta^{-O(\e m)}$-comparable cardinalities.
        \item\label{it:middle-scale-slablet-branching} There is a non-decreasing $g:\mathcal H\to\R$ with $g(0)=0$ such that
        \[
            \left|\log_{1/\Delta}
            \bigl|\,S_{h'}^{\mathbf{F}}(\mathcal I\mid (\mathcal D_{\Delta^n}(V_\theta+p),S_h^{\mathbf{F}})=(\mathbf{F},\Sigma_h))\,\bigr|
            -[g(h')-g(h)]\right|\le O(\e m)
        \]
        for $h\le h'$ in $\mathcal H$, $\mathbf{F}\in\mathcal D_{\Delta^n}(\mathcal F_{\mathcal I})$, and
        $\Sigma_h\in S_h^{\mathbf{F}}(\mathcal I\mid \mathcal D_{\Delta^n}(V_\theta+p)=\mathbf{F})$.
        \item\label{it:middle-scale-plane-tree} There is a non-decreasing $\widetilde g:\mathcal H\to\R$ with $\widetilde g(0)=0$ such that
        \[
            \left|\log_{1/\Delta}
            \bigl|\,\mathcal D_{\Delta^{m}}(\theta)(\mathcal I\mid (p,\mathcal D_{\Delta^n}(V_\theta+p),\mathcal D_{\Delta^{n+h}}(\theta))
            =(q,\mathbf{F},J))\,\bigr|
            -[\widetilde g(m-n)-\widetilde g(h)]\right|\le O(\e m)
        \]
       for $h\in\mathcal H$ and
        $(q,\mathbf{F},J)\in(p,\mathcal D_{\Delta^n}(V_\theta+p),\mathcal D_{\Delta^{n+h}}(\theta))(\mathcal I)$. Moreover, each family $\Xi_{\mathcal I}(q,\mathbf F)$ defined in \eqref{eq:normalized-plane-family} is $(\Delta^{m-n},s,\Delta^{-O(\e m)})_{d,k}$-spread-out.
    \end{enumerate}
\end{proposition}

\begin{proof}
We first select $\Delta^n$-cells whose normalized descendants remain spread-out, with a slightly worse constant. For every $q\in P$, apply Lemma \ref{lem:spread-out-sub-cube} to $\Theta_{\mathcal I_0}(q)$ at scales $\Delta^{m}$ and $\Delta^n$. Select the resulting cells $J\in\mathcal D_{\Delta^n}(\Theta_{\mathcal I_0}(q))$, which account for at least half the $\Delta^{m}$-cells and satisfy
\[
    \Delta^{-n}\bigl((\Theta_{\mathcal I_0}(q)\cap J)-z(J)\bigr)
    \quad\text{is }(\Delta^{m-n},s,\Delta^{-O(\e m)})_{d,k}\text{-spread-out}.
\]
Let $\mathcal J_1$ consist of the incidences $(q,\theta)\in\mathcal I_0$ for which $\mathcal D_{\Delta^n}(\theta)$ is selected for $q$. Since $\theta\in\delta\Z^{k\times(d-k)}$, each nonempty $(p,\mathcal D_{\Delta^{m}}(\theta))$ fiber is a singleton. Thus $\mathcal J_1$ contains at least half of every $p$ fiber of $\mathcal I_0$.

We next arrange that the plane and affine cells determine each other at each point. For each $(q,J)\in(p,\mathcal D_{\Delta^n}(\theta))(\mathcal J_1)$, select a value $\mathbf F$ of $\mathcal D_{\Delta^n}(V_\theta+p)$ occurring in the largest number of incidences in the $(q,J)$ fiber, and let $\mathcal J_2$ consist of the selected incidences. There are $O(1)$ choices of $\mathbf F$ for each $(q,J)$. The plane-coordinate projection of $\mathbf F$ is exactly $J$, so $(p,\mathcal D_{\Delta^n}(\theta))$ and $(p,\mathcal D_{\Delta^n}(V_\theta+p))$ have the same fibers on $\mathcal J_2$. In particular, $\mathcal J_2$ has density $\gtrsim1$ in every $p$ fiber of $\mathcal I_0$.

The refinement $\mathcal J_1\to\mathcal J_2$ preserves a fixed proportion of each $(p,\mathcal D_{\Delta^n}(\theta))$ fiber. Since $\theta_{\mathbf F}=z(J)$, Lemma \ref{lem:spread-out-subset} shows that each $\Xi_{\mathcal J_2}(q,\mathbf F)$ is $(\Delta^{m-n},s,\Delta^{-O(\e m)})_{d,k}$-spread-out.

Apply Lemma \ref{lem:hyperdyadic-fiber-regularization} with depth $m-n$ to $\mathcal J_2$, using $\tau=\e$, the trees
\[
    \bigl(\mathcal D_{\Delta^{n+h}}(\theta)\bigr)_{h\in\mathcal H}
    \quad\text{over }(p,\mathcal D_{\Delta^n}(V_\theta+p)),
    \qquad
    (S_h)_{h\in\mathcal H}\quad\text{over }\mathcal D_{\Delta^n}(V_\theta+p),
\]
and the additional protected map $p$. Denote the resulting refinement by $\mathcal I$, and the branching functions by $\widetilde g$ and $g$. The $p$ fibers of $\mathcal J_2$ have cardinalities $\sim M$, so Lemma \ref{lem:hyperdyadic-fiber-regularization}\textup{(\ref{it:hyperdyadic-fiber-protection})} implies \textup{(\ref{it:middle-scale-point-protection})} and their comparability in \textup{(\ref{it:middle-scale-fiber-comparability})}. Conclusion \textup{(\ref{it:hyperdyadic-fiber-comparability})} implies the $(p,\mathcal D_{\Delta^n}(\theta))$-fiber comparison in \textup{(\ref{it:middle-scale-fiber-comparability})}, and \textup{(\ref{it:hyperdyadic-fiber-branching})} implies \textup{(\ref{it:middle-scale-slablet-branching})} and the branching assertion in \textup{(\ref{it:middle-scale-plane-tree})}. Protection of $(p,\mathcal D_{\Delta^n}(V_\theta+p))$ preserves a $\Delta^{O(\e m)}$ fraction of every $\Xi_{\mathcal J_2}(q,\mathbf F)$, so Lemma \ref{lem:spread-out-subset} proves the spread-out assertion in \textup{(\ref{it:middle-scale-plane-tree})}.
\end{proof}

\subsection{Gain at the coarse scales}

We begin the proof of Proposition \ref{prop:non-sticky-furst-gain} in earnest, by using Lemma~\ref{lem:away-from-critical}\textup{(\ref{it:frostman-away-from-critical})} to get a gain on the scales $[\Delta^n, 1]$. We note that this by itself is insufficient because there could be a loss on the scales $[\delta, \Delta^n]$.
\begin{proposition}\label{prop:middle-scale-frostman-gain}
There is $C_1\ge10$, depending only on $d,k,s,\eta,\xi$, such that, whenever $\e\ll_\xi1$, $\Delta\ll_{\xi,\e}1$, and $m\gg_{\xi,\e,\Delta}1$, one may choose the refinement $\mathcal I$ and branching function $g$ in Proposition \ref{prop:middle-scale-regularization} with the following properties:
\begin{enumerate}[label=\textup{(\roman*)},ref=\roman*]
    \item
    \begin{equation}\label{eq:middle-scale-frostman-gain}
        \log_{1/\Delta}\frac{
        |\mathcal D_{\Delta^n}(\mathcal F_{\mathcal I})|}
        {\displaystyle\min_{q\in P_{\mathcal I}}
        |\mathcal D_{\Delta^n}(\Theta_{\mathcal I}(q))|}
        \ge\left(d-k-\eta+\frac{\eta\xi}{16k}\right)n-2C_1\e m.
    \end{equation}
    \item\label{it:middle-scale-frostman-point-fibers} For every $q\in P_{\mathcal I}$,
    \[
        |\{(p,\theta)\in\mathcal I:p=q\}|
        \ge \Delta^{C_1\e m}|\{(p,\theta)\in\mathcal I_0:p=q\}|.
    \]
\end{enumerate}
\end{proposition}
\begin{proof}
Choose $C_1\gg_{\xi} 1$; in particular, we will require it to be much larger than a number of implicit constants. Take $\e\ll_{\xi}1$, $\Delta\ll_{\xi,\e}1$, and $m\gg_{\xi,\e,\Delta}1$ so that Lemma \ref{lem:away-from-critical}\textup{(\ref{it:frostman-away-from-critical})} with $t=d-k+\xi/4$ and Proposition \ref{prop:middle-scale-regularization} apply.

Apply Proposition \ref{prop:middle-scale-regularization} to obtain $\mathcal I$ and the branching function $g$ in \textup{(\ref{it:middle-scale-slablet-branching})}. Thus $g$ describes the branching of the nested slablets inside every $\mathbf F\in\mathcal D_{\Delta^n}(\mathcal F_{\mathcal I})$.

Since the $p$ fibers on $\mathcal I_0$ have size  $\sim M$, Proposition \ref{prop:middle-scale-regularization} yields $|P_{\mathcal I}|\ge\delta^{O(\e)}|P|$. By uniformity of $P$, we also have $|\mathcal D_{\Delta^n}(P_{\mathcal I})|\ge\delta^{O(\e)}|\mathcal D_{\Delta^n}(P)|$,  and
\[
    |\mathcal D_{\Delta^n}(P_{\mathcal I}\cap Q)|
    \le\Delta^{-[f(n)-f(i)]}
    \quad(0\le i\le n,\ Q\in\mathcal D_{\Delta^i}(P_{\mathcal I})).
\]
We deduce from the Frostman condition on $P$ and \eqref{eq:non-sticky} that $P_{\mathcal I}$ is $(\Delta^n,d-k,\Delta^{-C_1\e m/10})$-Frostman and  $|P_{\mathcal I}|_{\Delta^n}\ge\Delta^{-(d-k+\xi/4)n}$.

By Lemma \ref{lem:spread-out-subset} and Proposition \ref{prop:middle-scale-regularization}\textup{(\ref{it:middle-scale-point-protection})}, the shading $\Theta_{\mathcal I}(q)$ is $(\delta,s,\delta^{-C_1\e/1000})_{d,k}$-spread-out for every $q\in P_{\mathcal I}$. Since $\Theta_{\mathcal I}(q)\subset\delta\Z^{k\times(d-k)}$, Proposition \ref{prop:middle-scale-regularization}\textup{(\ref{it:middle-scale-fiber-comparability})} shows that the nonzero cardinalities
\[
    |\{I\in\mathcal D_{\Delta^{m}}(\Theta_{\mathcal I}(q)):I\subset J\}|,
    \qquad J\in\mathcal D_{\Delta^n}(\Theta_{\mathcal I}(q))
\]
are mutually $\Delta^{-C_1\e m/1000}$-comparable. By Lemma \ref{lem:spread-out-coarsening}, the families of occupied $\Delta^n$-plane cells are $(\Delta^n,s,\Delta^{-C_1\e m/50})_{d,k}$-spread-out.

Double counting each $p$ fiber over its $\mathcal D_{\Delta^n}(\theta)$ values, using both comparisons in Proposition \ref{prop:middle-scale-regularization}\textup{(\ref{it:middle-scale-fiber-comparability})}, we deduce
\begin{equation}\label{eq:frostman-slablet-coarse-plane-cardinality}
    M_{\mathcal I}(n)
    \le |\mathcal D_{\Delta^n}(\Theta_{\mathcal I}(q))|
    \le\Delta^{-C_1\e m/10}M_{\mathcal I}(n),
    \qquad q\in P_{\mathcal I}.
\end{equation}

For each $Q\in\mathcal D_{\Delta^n}(P_{\mathcal I})$, choose $q\in P_{\mathcal I}\cap Q$, and restrict $\mathcal D_{\Delta^n}(\Theta_{\mathcal I}(q))$ to exactly $M_{\mathcal I}(n)$ cells. Equation \eqref{eq:frostman-slablet-coarse-plane-cardinality} and Lemma \ref{lem:spread-out-subset} show that these restricted families are $(\Delta^n,s,\Delta^{-C_1\e m/5})_{d,k}$-spread-out. They form a $(\Delta^n,s,\Delta^{-C_1\e m/5},M_{\mathcal I}(n))$-spread-out configuration on $\mathcal D_{\Delta^n}(P_{\mathcal I})$.

For every $Q\in\mathcal D_{\Delta^n}(P_{\mathcal I})$ and each plane cell $J$ selected for $Q$, choose $(q,\theta)\in\mathcal I$ with $q\in Q$ and $\mathcal D_{\Delta^n}(\theta)=J$. The affine coordinates of $V_{z(J)}+z(Q)$ and $V_\theta+q$ differ by $O(\Delta^n)$. Hence the affine cells of the configuration on $\mathcal D_{\Delta^n}(P_{\mathcal I})$ map with $O(1)$ multiplicity into $\mathcal D_{\Delta^n}(\mathcal F_{\mathcal I})$. Applying Lemma~\ref{lem:away-from-critical}\textup{(\ref{it:frostman-away-from-critical})} with $t=d-k+\xi/4$, we obtain
\[
    \log_{1/\Delta}\frac{|\mathcal D_{\Delta^n}(\mathcal F_{\mathcal I})|}
    {M_{\mathcal I}(n)}
    \ge\left(d-k-\eta+\frac{3\eta\xi}{32k}\right)n-O(1).
\]
We have established \eqref{eq:middle-scale-frostman-gain} for $m\gg_{\xi,\e,\Delta} 1$. Finally, Proposition \ref{prop:middle-scale-regularization}\textup{(\ref{it:middle-scale-point-protection})} yields \textup{(\ref{it:middle-scale-frostman-point-fibers})} if $C_1\gg 1$.
\end{proof}

\subsection{Parameter selection}
\label{s:slablet-parameters}

Let $2000\sigma$ be the value of $\e$ provided by $\mathsf{Furst}(s;\eta)$ for $\rho=\eta\xi^2/(1024dk)$. Let $C_1$ be large enough that Proposition \ref{prop:middle-scale-frostman-gain} holds (in particular, $C_1=C_1(\xi)$), and also large enough in terms of $\sigma$ and many absolute implicit constants as required in the proof. Define $\lambda_0$ and then $\eta'$ by
\begin{equation}\label{eq:non-sticky-gain-parameter}
    \lambda_0=\frac{\eta\sigma}{8(d-k)},\qquad
    \eta-\eta'=\min\left\{\frac{\eta\xi^2}{256dk},
    \frac{\lambda_0\sigma\xi}{16d(d-k)}\right\}.
\end{equation}
Fix $\rho>0$ and let $\e \ll_{\xi,\eta-\eta',\sigma,\lambda_0,C_1,\rho} 1$. Choose $\Delta\in2^{-\N}$ with $\Delta\ll_{\xi,\e}1$ so that $\mathsf{Furst}(s;\eta)$ with $\eta\xi^2/(1024dk)$ in place of $\rho$, Lemma \ref{lem:katz-tao-induction} with $\rho=(\eta-\eta')/8$, Lemma~\ref{lem:away-from-critical}\textup{(\ref{it:katz-tao-away-from-critical})} with $t=d-k-\sigma/2$, and Proposition \ref{prop:middle-scale-frostman-gain} apply; then choose $m\gg_{\xi,\e,\Delta,\rho}1$.

With $d,k,s,\eta$ fixed, the parameters are chosen in the following order:
\[
\begin{aligned}
    &\xi\longrightarrow\sigma\longrightarrow\lambda_0,\qquad
     (\xi,\sigma)\longrightarrow C_1, \qquad
    (\xi,\lambda_0,\sigma)\longrightarrow\eta',
   \\
    &(\xi,\eta-\eta',\sigma,\lambda_0,C_1,\rho)\longrightarrow\e,\qquad
    (\xi,\e)\longrightarrow\Delta,\qquad
    (\xi,\e,\Delta,\rho)\longrightarrow m.
\end{aligned}
\]

For the rest of the section, fix the incidence set $\mathcal I$ and branching function $g$ from Proposition \ref{prop:middle-scale-frostman-gain} with these parameters. By Proposition \ref{prop:middle-scale-regularization}\textup{(\ref{it:middle-scale-slablet-branching})},
\begin{equation}\label{eq:middle-scale-slablet-descendants}
    \left|\log_{1/\Delta}
    \bigl|\,S_{h'}^{\mathbf{F}}(\mathcal I\mid (\mathcal D_{\Delta^n}(V_\theta+p),S_h^{\mathbf{F}})=(\mathbf{F},\Sigma_h))\,\bigr|
    -[g(h')-g(h)]\right|
    \le C_1\e m/1000
\end{equation}
for $h\le h'$ in $\mathcal H$, $\mathbf F\in\mathcal D_{\Delta^n}(\mathcal F_{\mathcal I})$, and $\Sigma_h\in S_h^{\mathbf F}(\mathcal I\mid \mathcal D_{\Delta^n}(V_\theta+p)=\mathbf F)$.

\subsection{Critical scale for slablet branching}

We compare the slablet branching function with the line of slope $d-k$. Choose $a\in\mathcal H$ maximizing $(d-k)h-g(h)$ and set
\begin{equation}\label{eq:slablet-branching-deficit}
    R=(d-k)a-g(a)
    =\max_{h\in\mathcal H}\bigl((d-k)h-g(h)\bigr)\ge0.
\end{equation}
By maximality of $a$,
\begin{equation}
 \label{eq:slablet-branching-F} 
 g(a)-g(h)\le(d-k)(a-h)\quad(h\in\mathcal H,\ h\le a).
\end{equation}
Fix $0\le i\le m-n$. Recall the construction of $\mathcal H_\e(m-n)$ in \S\ref{s:hyperdyadic-scales}, and let $h\in\mathcal H$ be the largest element with $h\le i$. Then $h$ is well-defined and $i-h\le\e(m-n)+1$.

If $R\le\sigma(m-n)$, maximality in \eqref{eq:slablet-branching-deficit} yields $g(h)\ge(d-k)h-\sigma(m-n)$ for every $h\in\mathcal H$. Since every level-$i$ cell has a unique level-$h$ ancestor, \eqref{eq:middle-scale-slablet-descendants} implies that, for every $\mathbf{F}\in\mathcal D_{\Delta^n}(\mathcal F_{\mathcal I})$,
\begin{equation}\label{eq:slablet-frostman-branching}   
    S_{m-n}^{\mathbf{F}}(\mathcal I\mid \mathcal D_{\Delta^n}(V_\theta+p)=\mathbf{F})
    \text{ is }(\Delta^{m-n},d-k,\Delta^{-\sigma(m-n)-C_1\e m/400})\text{-Frostman}.
\end{equation}
We will use this to apply $\mathsf{Furst}(s;\eta)$ directly at level $m-n$.

If $R\ge\sigma(m-n)$, then \eqref{eq:slablet-branching-deficit} yields
\begin{equation}\label{eq:slablet-deficit-scale}
    a\ge\frac{\sigma(m-n)}{d-k},
    \qquad
    g(a)\le(d-k-\sigma)a.
\end{equation}
Arguing as in the case $R\le \sigma(m-n)$, \eqref{eq:middle-scale-slablet-descendants}, and \eqref{eq:slablet-branching-F} show that, for every $\mathbf{F}\in\mathcal D_{\Delta^n}(\mathcal F_{\mathcal I})$,
\begin{equation}\label{eq:slablet-katz-tao-branching}
    S_a^{\mathbf{F}}(\mathcal I\mid \mathcal D_{\Delta^n}(V_\theta+p)=\mathbf{F})    \text{ is }(\Delta^a,d-k,\Delta^{-2C_1\e m})\text{-Katz--Tao}.
\end{equation}
Proposition \ref{prop:middle-scale-regularization}\textup{(\ref{it:middle-scale-slablet-branching})} and \eqref{eq:slablet-deficit-scale} yield
\begin{equation}\label{eq:slablet-prefix-cardinality}
    \bigl|\,S_a^{\mathbf{F}}(\mathcal I\mid \mathcal D_{\Delta^n}(V_\theta+p)=\mathbf{F})\,\bigr|
    \le \Delta^{-(d-k-\sigma/2)a}.
\end{equation}
The Katz--Tao bound \eqref{eq:slablet-katz-tao-branching} and the subcritical cardinality \eqref{eq:slablet-prefix-cardinality} will allow us to apply Lemma~\ref{lem:away-from-critical}\textup{(\ref{it:katz-tao-away-from-critical})} at level $a$ to obtain a gain in this case. See Figure \ref{fig:slablet-branching} for a depiction of the two cases.

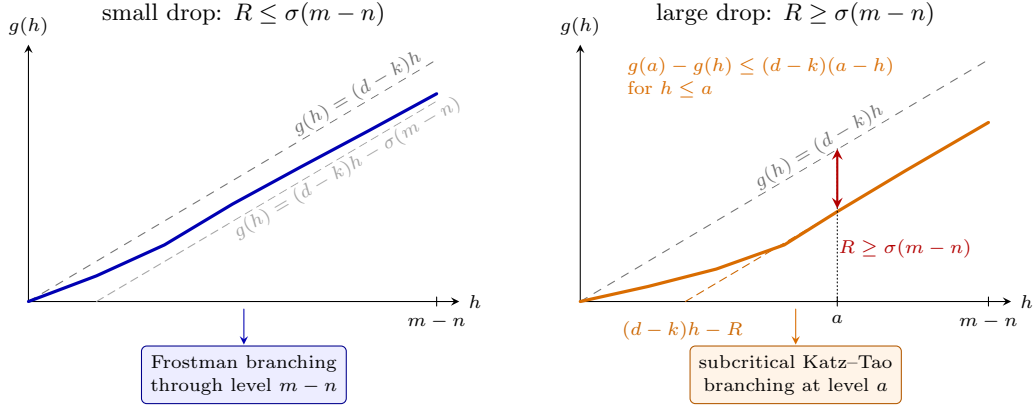
\begin{figure}[htbp]
\centering
\begin{tikzpicture}[x=1cm,y=1cm,>=stealth,line cap=round,line join=round]
    \node[font=\small] at (3.50,5.15) {small drop: $R\le\sigma(m-n)$};
    \node[font=\small] at (10.80,5.15) {large drop: $R\ge\sigma(m-n)$};

    \draw[->] (0.65,1.35) -- (6.35,1.35) node[right,font=\scriptsize] {$h$};
    \draw[->] (0.65,1.35) -- (0.65,4.75) node[above,font=\scriptsize] {$g(h)$};
    \draw[black!55,dashed] (0.65,1.35) -- (6.05,4.55)
        node[pos=.82,above,sloped,font=\scriptsize,fill=white,inner sep=0.8pt] {$g(h)=(d-k)h$};
    \draw[black!35,densely dashed] (1.55,1.35) -- (6.05,4.00)
        node[pos=.72,below,sloped,font=\scriptsize,fill=white,inner sep=0.8pt] {$g(h)=(d-k)h-\sigma(m-n)$};
    \draw[blue!70!black,very thick] plot coordinates {
        (0.65,1.35) (1.55,1.69) (2.45,2.10) (3.35,2.64)
        (4.25,3.13) (5.15,3.61) (6.05,4.10)};
    \draw (6.05,1.29) -- (6.05,1.41);
    \node[font=\scriptsize,below] at (6.05,1.30) {$m-n$};
    \draw[->,blue!70!black] (3.50,1.25) -- (3.50,0.78);
    \node[draw=blue!55!black,fill=blue!7,rounded corners=2pt,
        font=\scriptsize,align=center] at (3.50,0.38)
        {Frostman branching\\through level $m-n$};

    \draw[->] (7.95,1.35) -- (13.65,1.35) node[right,font=\scriptsize] {$h$};
    \draw[->] (7.95,1.35) -- (7.95,4.75) node[above,font=\scriptsize] {$g(h)$};
    \draw[black!55,dashed] (7.95,1.35) -- (13.35,4.55)
        node[pos=.60,above,sloped,font=\scriptsize,fill=white,inner sep=0.8pt] {$g(h)=(d-k)h$};
    \draw[orange!80!black,densely dashed] (9.34,1.35) -- (11.35,2.54);
    \node[font=\scriptsize,orange!80!black] at (9.30,0.97)
        {$(d-k)h-R$};
    \draw[orange!85!black,very thick] plot coordinates {
        (7.95,1.35) (8.85,1.55) (9.75,1.78) (10.65,2.10)
        (11.35,2.54) (12.25,3.08) (13.35,3.72)};
    \draw[densely dotted] (11.35,1.35) -- (11.35,3.37);
    \draw[<->,red!70!black,thick] (11.35,2.58) -- (11.35,3.37);
    \node[font=\scriptsize,red!70!black] at (12.25,2.05) {$R\ge\sigma(m-n)$};
    \draw (11.35,1.29) -- (11.35,1.41);
    \node[font=\scriptsize,below] at (11.35,1.30) {$a$};
    \draw (13.35,1.29) -- (13.35,1.41);
    \node[font=\scriptsize,below] at (13.35,1.30) {$m-n$};
    \node[font=\scriptsize,orange!85!black,anchor=west,align=left] at (8.45,4.32)
        {$g(a)-g(h)\le(d-k)(a-h)$\\for $h\le a$};
    \draw[->,orange!85!black] (10.80,1.25) -- (10.80,0.78);
    \node[draw=orange!70!black,fill=orange!9,rounded corners=2pt,
        font=\scriptsize,align=center] at (10.80,0.38)
        {subcritical Katz--Tao\\branching at level $a$};
\end{tikzpicture}
\caption{The maximal drop $R=(d-k)a-g(a)$ separates two cases. A small drop forces Frostman slablet branching through level $m-n$; a large drop forces subcritical Katz--Tao branching at level $a$. Both cases lead to a gain.}
\label{fig:slablet-branching}
\end{figure}

\subsection{Frostman slablets: the case \texorpdfstring{$R\le\sigma(m-n)$}{R <= sigma(m-n)}}

For every $\mathbf{F}\in\mathcal D_{\Delta^n}(\mathcal F_{\mathcal I})$, the maps $S_0^{\mathbf{F}},\ldots,S_{m-n}^{\mathbf{F}}$ describe nested slablets inside the $\Delta^n$-slab $\mathbf{F}$. Since $R\le\sigma(m-n)$, the lower-left corners of the level-$(m-n)$ slablets form a $(d-k)$-Frostman set in the $A_{\mathbf{F}}$ coordinates.

The gain comes from the scales $[\Delta^n,1]$: under the non-sticky assumption, Proposition \ref{prop:middle-scale-frostman-gain} gives us \eqref{eq:middle-scale-frostman-gain}. On the scales $[\delta,\Delta^n]$, we use the Frostman property of the slablets and incur only a small loss. This is the only case that uses the non-sticky assumption \eqref{eq:non-sticky}.
\begin{lemma}\label{lem:slablet-entropy-below-threshold}
Suppose that $R\le\sigma(m-n)$. Then
\[
    \log_{1/\Delta}\frac{
    |\mathcal D_{\Delta^{m}}(\mathcal F_{\mathcal I})|}
    {\displaystyle\min_{q\in P_{\mathcal I}}
    |\mathcal D_{\Delta^{m}}(\Theta_{\mathcal I}(q))|}
    \ge
    \left(d-k-\eta+\frac{\eta\xi^2}{64dk}-8C_1\e\right)m.
\]
\end{lemma}
\begin{proof}

By \eqref{eq:slablet-frostman-branching}, the level-$(m-n)$ slablets are $(\Delta^{m-n},d-k,\Delta^{-\sigma(m-n)-C_1\e m/400})$-Frostman for all $\mathbf{F}\in\mathcal D_{\Delta^n}(\mathcal F_{\mathcal I})$.

Proposition \ref{prop:middle-scale-regularization}\textup{(\ref{it:middle-scale-plane-tree})}, with $h=0$, yields
\begin{equation}\label{eq:frostman-slablet-plane-descendants}
    \left|\log_{1/\Delta}|\Xi_{\mathcal I}(q,\mathbf{F})|-\widetilde g(m-n)\right|
    \le C_1\e m\quad\text{for all }(q,\mathbf F)\in(p,\mathcal D_{\Delta^n}(V_\theta+p))(\mathcal I).
\end{equation}
Moreover, by Proposition \ref{prop:middle-scale-regularization}\textup{(\ref{it:middle-scale-plane-tree})}, each $\Xi_{\mathcal I}(q,\mathbf{F})$ is a $(\Delta^{m-n},s,\Delta^{-C_1\e m/1000})_{d,k}$-spread-out family.

Let
\[
    M_*=\min_{(q,\mathbf F)\in(p,\mathcal D_{\Delta^n}(V_\theta+p))(\mathcal I)}
    |\Xi_{\mathcal I}(q,\mathbf F)|.
\]
For every $(\mathbf{F},\Sigma_{m-n})\in(\mathcal D_{\Delta^n}(V_\theta+p),S_{m-n})(\mathcal I)$, choose an $M_*$-element family $\Xi'_{\mathbf F}(\Sigma_{m-n})\subset\Xi_{\mathcal I}(q,\mathbf F)$ for some $(q,\mathbf F,\Sigma_{m-n})\in(p,\mathcal D_{\Delta^n}(V_\theta+p),S_{m-n})(\mathcal I)$. By \eqref{eq:frostman-slablet-plane-descendants}, $\Xi'_{\mathbf F}(\Sigma_{m-n})$ has relative cardinality at least $\Delta^{2C_1\e m}$. Lemma \ref{lem:spread-out-subset} shows that $\Xi'_{\mathbf F}(\Sigma_{m-n})$ is $(\Delta^{m-n},s,\Delta^{-3C_1\e m})_{d,k}$-spread-out.

By \eqref{eq:non-sticky-scale-lengths} and \eqref{eq:slablet-frostman-branching}, $\mathsf{Furst}(s;\eta)$ applies to the level-$(m-n)$ slablets in $\mathbf F$ with shadings $\Xi'_{\mathbf F}(\Sigma_{m-n})$, with $\rho=\eta\xi^2/(1024dk)$ and $1000\sigma$ in place of $\e$.

For each $I\in\Xi'_{\mathbf F}(\Sigma_{m-n})$, choose an incidence $(q,\theta)\in\mathcal I$ with $\mathcal D_{\Delta^n}(V_\theta+q)=\mathbf F$, $A_{\mathbf F}(q)\in\Sigma_{m-n}$, and $I=\mathcal D_{\Delta^{m-n}}(\Delta^{-n}(\theta-\theta_{\mathbf F}))$. 
Applying $T_{\mathbf{F}}^{-1}$, the affine cells in the conclusion of $\mathsf{Furst}(s;\eta)$ map with $O(1)$ multiplicity into $\mathcal D_{\Delta^{m}}(V_\theta+p)(\mathcal I\mid \mathcal D_{\Delta^n}(V_\theta+p)=\mathbf{F})$. Consequently,
\begin{equation}\label{eq:frostman-slablet-fine-gain}
    \log_{1/\Delta}\frac{
    \bigl|\,\mathcal D_{\Delta^{m}}(V_\theta+p)(\mathcal I\mid \mathcal D_{\Delta^n}(V_\theta+p)=\mathbf{F})\,\bigr|}{M_*}
    \ge\left(d-k-\eta-\frac{\eta\xi^2}{512dk}\right)(m-n)-2C_1\e m,
\end{equation}
for every $\mathbf{F}\in\mathcal D_{\Delta^n}(\mathcal F_{\mathcal I})$.

We next compare the plane counts at levels $n$ and $m$. For fixed $q$, equation \eqref{eq:affine-incidence-coordinates} shows that each $\Delta^n$-plane cell occurs with $O(1)$ affine $\Delta^n$-cells. Equations \eqref{eq:frostman-slablet-coarse-plane-cardinality} and \eqref{eq:frostman-slablet-plane-descendants}  therefore imply
\begin{equation}\label{eq:frostman-slablet-plane-cardinality}
    |\mathcal D_{\Delta^{m}}(\Theta_{\mathcal I}(q))|
    \le\Delta^{-4C_1\e m}M_{\mathcal I}(n)M_*,
    \qquad q\in P_{\mathcal I}.
\end{equation}
Summing \eqref{eq:frostman-slablet-fine-gain} over $\mathbf F\in\mathcal D_{\Delta^n}(\mathcal F_{\mathcal I})$ and using \eqref{eq:frostman-slablet-plane-cardinality}, we obtain
\[
    \mathcal E_{\mathcal I}(m)\ge
    \mathcal E_{\mathcal I}(n)+\left(d-k-\eta-\frac{\eta\xi^2}{512dk}\right)(m-n)-6C_1\e m.
\]
Together with \eqref{eq:non-sticky-scale-lengths} and \eqref{eq:middle-scale-frostman-gain}, this proves the lemma.
\end{proof}

\subsection{Katz--Tao slablets at \texorpdfstring{scale $\Delta^a$}{the Katz-Tao scale}: the case \texorpdfstring{$R\ge\sigma(m-n)$}{R > sigma(m-n)}}

When $R\ge\sigma(m-n)$, the gain comes from the subcritical number of slablets at level $a$. We first describe how this gain enters the argument.

Maximizing over $a$ in \eqref{eq:slablet-branching-deficit}, we obtain $g(a)-g(h)\le(d-k)(a-h)$ for $h\in\mathcal H$, $h\le a$. This is Katz--Tao non-concentration for the slablets. More precisely, by \eqref{eq:slablet-katz-tao-branching}, for $0\le i\le a$ and $\mathbf{F}\in\mathcal D_{\Delta^n}(\mathcal F_{\mathcal I})$, every occupied level-$i$ slablet inside $\mathbf{F}$ contains at most $\Delta^{-O(\e m)}\left(\frac{\Delta^i}{\Delta^a}\right)^{d-k}$ occupied level-$a$ slablets. Also, $R\ge\sigma(m-n)$, so $g(a)\le(d-k-\sigma)a$. By \eqref{eq:slablet-prefix-cardinality}, the total number of occupied level-$a$ slablets inside $\mathbf{F}$ is at most $\Delta^{-(d-k-\sigma/2)a}$, strictly below the critical size $\Delta^{-(d-k)a}$.

For each $\mathbf{F}\in\mathcal D_{\Delta^n}(\mathcal F_{\mathcal I})$, renormalize the slablets to cubes using $A_{\mathbf{F}}$. Lemma~\ref{lem:away-from-critical}\textup{(\ref{it:katz-tao-away-from-critical})} then provides a bound for the number of level-$a$ point cells that can occur with one value of $\mathcal D_{\Delta^{m}}(V_\theta+p)$, as stated in \eqref{eq:slablet-prefix-collision-target} below. The other main bound, stated in  \eqref{eq:point-fiber-entropy-target}, uses the Katz--Tao property of $P\cap Q$ for an occupied $\Delta^a$-cube $Q$, after rescaling. Morally, \eqref{eq:point-fiber-entropy-target} produces at worst a negligible loss. This small loss is offset by the gain in \eqref{eq:slablet-prefix-collision-target}.
\begin{proposition}
\label{prop:slablet-entropy-above-threshold}
Suppose that $R\ge\sigma(m-n)$. There is $\mathcal I'\subset\mathcal I_0$ such that, for every $q\in P_{\mathcal I'}$,
\begin{equation}\label{eq:slablet-point-fibers-above-threshold}
    |\{(p,\theta)\in\mathcal I':p=q\}|\ge
    \delta^{C_1\e}|\{(p,\theta)\in\mathcal I_0:p=q\}|,
\end{equation}
\begin{equation}\label{eq:slablet-prefix-collision-target}
    \bigl|\,\mathcal D_{\Delta^a}(p)(\mathcal I'\mid \mathcal D_{\Delta^{m}}(V_\theta+p)=\mathbf{F}_{m})\,\bigr|\le
    \Delta^{-(\eta-\lambda_0)a-2C_1\e m}
    \qquad(\mathbf{F}_{m}\in\mathcal D_{\Delta^{m}}(\mathcal F_{\mathcal I'})),
\end{equation}
and
\begin{equation}\label{eq:point-fiber-entropy-target}
    \sum_{Q\in\mathcal D_{\Delta^a}(P_{\mathcal I'})}
    \bigl|\,\mathcal D_{\Delta^{m}}(V_\theta+p)(\mathcal I'\mid \mathcal D_{\Delta^a}(p)=Q)\,\bigr|
    \ge
    \Delta^{-(d-k-\eta-(\eta-\eta')/2)m-\eta a}
    \min_{q\in P_{\mathcal I'}}
    |\mathcal D_{\Delta^{m}}(\Theta_{\mathcal I'}(q))|.
\end{equation}
\end{proposition}

Before proving Proposition \ref{prop:slablet-entropy-above-threshold}, we produce a refinement that simultaneously uniformizes a number of fiber cardinalities.
\begin{lemma}\label{lem:slablet-above-threshold-refinement}
Suppose that $R\ge\sigma(m-n)$. There is $\mathcal I'\subset\mathcal I$ with the following properties.
\begin{enumerate}[label=\textup{(\roman*)},ref=\roman*]
    \item\label{it:above-threshold-refinement-point-protection}
    \eqref{eq:slablet-point-fibers-above-threshold} holds.

    \item\label{it:above-threshold-refinement-plane-shadings}
    The nonempty sets
    \[
        \mathcal D_{\Delta^{n+a}}(\theta)\bigl(\mathcal I'\mid (\mathcal D_{\Delta^n}(V_\theta+p),S_a^{\mathbf{F}})=(\mathbf{F},\Sigma_a)\bigr)
    \]
    have mutually $\Delta^{-C_1\e m/1000}$-comparable cardinalities. For each $(\mathbf F,\Sigma_a)$, their images under $\theta\mapsto\Delta^{-n}(\theta-\theta_{\mathbf F})$ form a $(\Delta^a,s,\Delta^{-C_1\e m/5})_{d,k}$-spread-out family of $\Delta^a$-cells.

    \item\label{it:above-threshold-refinement-affine-slablet-fibers}
    The nonzero cardinalities
    \[
        \bigl|\,S_a^{\mathbf{F}_n}(\mathcal I'\mid (\mathcal D_{\Delta^n}(V_\theta+p),\mathcal D_{\Delta^{n+a}}(V_\theta+p))
        =(\mathbf{F}_n,\mathbf{F}_{n+a}))\,\bigr|
    \]
    are mutually $\Delta^{-C_1\e m/1000}$-comparable.

    \item\label{it:above-threshold-refinement-point-planes}
    For every $q\in P_{\mathcal I'}$, the set
    $\Theta_{\mathcal I'}(q)$ is
    $(\Delta^{m-a},s,\Delta^{-6C_1\e m})_{d,k}$-spread-out, and the covering numbers $|\Theta_{\mathcal I'}(q)|_{\Delta^{m-a}}$ are mutually $\Delta^{-C_1\e m/1000}$-comparable as $q$ ranges over $P_{\mathcal I'}$.

    \item\label{it:above-threshold-refinement-point-fibers}
    The nonzero cardinalities
    \[
        \bigl|\,\mathcal D_{\Delta^{m}}(\theta)(\mathcal I'\mid (p,\mathcal D_{\Delta^{m-a}}(\theta))=(q,J))\,\bigr|
    \]
    are mutually $\Delta^{-C_1\e m/1000}$-comparable, and
    $|P_{\mathcal I'}|\ge\Delta^{-(d-k)m+C_1\e m}$.
\end{enumerate}
\end{lemma}

\begin{proof}

Since $\theta\in\delta\Z^{k\times(d-k)}$, Proposition \ref{prop:middle-scale-regularization}\textup{(\ref{it:middle-scale-plane-tree})}, with $h=a$, shows that the nonempty
\[
    (p,\mathcal D_{\Delta^n}(V_\theta+p),
    \mathcal D_{\Delta^{n+a}}(\theta))
\]
fibers of $\mathcal I$ have mutually $\Delta^{-C_1\e m/1000}$-comparable cardinalities.

 Lemma \ref{lem:spread-out-coarsening}  and Proposition \ref{prop:middle-scale-regularization}\textup{(\ref{it:middle-scale-plane-tree})} show that every $\mathcal D_{\Delta^a}(\Xi_{\mathcal I}(q,\mathbf F))$ is $(\Delta^a,s,2\Delta^{-C_1\e m/500})_{d,k}$-spread-out.

For each $(\mathbf{F},\Sigma_a)\in(\mathcal D_{\Delta^n}(V_\theta+p),S_a)(\mathcal I)$, assign to each $\mathcal D_{\Delta^{n+a}}(\theta)$ value the weight equal to its fiber cardinality over $(\mathbf F,\Sigma_a)$. Select a dyadic class of weights accounting for at least $m^{-O(1)}$ of the incidences, and let $\mathcal I_1$ consist of the selected incidences. Then $\mathcal D_{\Delta^{n+a}}(\theta)$ is $2$-uniform over $(\mathcal D_{\Delta^n}(V_\theta+p),S_a)$ on $\mathcal I_1$. Summing the spread-out inequalities for $\mathcal D_{\Delta^a}(\Xi_{\mathcal I}(q,\mathbf F))$ and using the $\Delta^{-C_1\e m/1000}$-comparability of the $(p,\mathcal D_{\Delta^n}(V_\theta+p),\mathcal D_{\Delta^{n+a}}(\theta))$ fibers shows that the normalized weights satisfy the spread-out inequality with constant $2\Delta^{-3C_1\e m/1000}$. Comparing weights in the selected class shows that the sets in \textup{(\ref{it:above-threshold-refinement-plane-shadings})}, with $\mathcal I_1$ in place of $\mathcal I'$, are $(\Delta^a,s,\Delta^{-C_1\e m/250})_{d,k}$-spread-out after rescaling by $\theta\mapsto\Delta^{-n}(\theta-\theta_{\mathbf F})$.

Pigeonhole jointly on $\mathcal I_1$ the fiber cardinalities of the six maps
\begin{equation}\label{eq:slablet-above-threshold-regularized-maps}
\begin{gathered}
    p,\qquad (p,\mathcal D_{\Delta^{m-a}}(\theta)),\\
    (\mathcal D_{\Delta^n}(V_\theta+p),S_a),\qquad
    (\mathcal D_{\Delta^n}(V_\theta+p),S_a,\mathcal D_{\Delta^{n+a}}(\theta)),\\
    (\mathcal D_{\Delta^n}(V_\theta+p),\mathcal D_{\Delta^{n+a}}(V_\theta+p)),\qquad
    (\mathcal D_{\Delta^n}(V_\theta+p),\mathcal D_{\Delta^{n+a}}(V_\theta+p),S_a).
\end{gathered}
\end{equation}
There are $m^{O(1)}$ joint classes. Let $\mathcal I_2$ consist of the incidences in a class of largest cardinality; then $\mathcal I_2\prec_{m^{-O(1)}}\mathcal I_1$. Apply Lemma \ref{lem:protected-fiber-refinement} to $\mathcal I_2\subset\mathcal I$ with the maps in \eqref{eq:slablet-above-threshold-regularized-maps}, and denote the resulting refinement by $\mathcal I'$. Then $|\mathcal I'|\ge m^{-O(1)}|\mathcal I|$, and all six maps are $m^{-O(1)}$-protected from $\mathcal I$ to $\mathcal I'$. The dyadic classes of their $\mathcal I_1$-fiber cardinalities and the inclusion $\mathcal I_1\subset\mathcal I$ imply that the nonempty fibers of each map on $\mathcal I'$ have mutually $m^{O(1)}$-comparable cardinalities.

Comparing the two fiber cardinalities in each row of \eqref{eq:slablet-above-threshold-regularized-maps} establishes the cardinality assertions in \textup{(\ref{it:above-threshold-refinement-plane-shadings})}, \textup{(\ref{it:above-threshold-refinement-affine-slablet-fibers})}, and \textup{(\ref{it:above-threshold-refinement-point-planes})}. Since $\theta\in\delta\Z^{k\times(d-k)}$, the $(p,\mathcal D_{\Delta^{m-a}}(\theta))$-fiber cardinalities equal their numbers of $\mathcal D_{\Delta^{m}}(\theta)$ descendants, so the comparability in \textup{(\ref{it:above-threshold-refinement-point-fibers})} follows.

For the uniform fibers in Definition \ref{def:uniform-pair}, protection also controls the loss in the number of distinct values. The map $(\mathcal D_{\Delta^n}(V_\theta+p),S_a)$ is $m^{-O(1)}$-protected from $\mathcal I_1$ to $\mathcal I'$. For each $(\mathbf F,\Sigma_a)\in(\mathcal D_{\Delta^n}(V_\theta+p),S_a)(\mathcal I')$, it follows from $2$-uniformity that the cardinality of each nonempty $\mathcal D_{\Delta^{n+a}}(\theta)$ fiber of $\mathcal I_1$ over $(\mathbf F,\Sigma_a)$ is bounded by twice the average cardinality of those fibers. Hence at least an $m^{-O(1)}$ fraction of the $\mathcal D_{\Delta^{n+a}}(\theta)$ values in $\mathcal I_1$ over $(\mathbf F,\Sigma_a)$ also occur in $\mathcal I'$ over $(\mathbf F,\Sigma_a)$. Lemma \ref{lem:spread-out-subset} establishes the spread-out assertion in \textup{(\ref{it:above-threshold-refinement-plane-shadings})}.

Protection of $p$ from $\mathcal I$ to $\mathcal I'$, Lemma \ref{lem:spread-out-subset}, and Proposition \ref{prop:middle-scale-regularization}\textup{(\ref{it:middle-scale-point-protection})} imply \eqref{eq:slablet-point-fibers-above-threshold} and show that $\Theta_{\mathcal I'}(q)$ is $(\delta,s,\delta^{-2C_1\e})_{d,k}$-spread-out for every $q\in P_{\mathcal I'}$. Lemma \ref{lem:spread-out-coarsening} and the descendant comparisons in \textup{(\ref{it:above-threshold-refinement-point-fibers})} establish the spread-out assertion in \textup{(\ref{it:above-threshold-refinement-point-planes})}.

Finally, $|\mathcal I'|\ge m^{-O(1)}|\mathcal I|$, $|\mathcal I|\ge\delta^{O(\e)}|\mathcal I_0|$, and $M\le|\Theta_{\mathcal I_0}(q)|<2M$ imply $|P_{\mathcal I'}|\ge\delta^{O(\e)}|P|$. The cardinality bound in \textup{(\ref{it:above-threshold-refinement-point-fibers})} follows from the Frostman condition on $P$.
\end{proof}

\begin{proof}[Proof of Proposition \ref{prop:slablet-entropy-above-threshold}]
Let $\mathcal I'$ be the refinement provided by Lemma \ref{lem:slablet-above-threshold-refinement}; \eqref{eq:slablet-point-fibers-above-threshold} holds by \textup{(\ref{it:above-threshold-refinement-point-protection})}.

\textbf{(a) Proof of \eqref{eq:slablet-prefix-collision-target}.}

The gain will come from Lemma~\ref{lem:away-from-critical}\textup{(\ref{it:katz-tao-away-from-critical})}, applied at scale $\Delta^a$ with $t=d-k-\sigma/2$ to the slablets inside each $\mathbf F$. Set
\[
    M_*=
    \min_{(\mathbf F,\Sigma_a)\in(\mathcal D_{\Delta^n}(V_\theta+p),S_a)(\mathcal I')}
    \bigl|\,\mathcal D_{\Delta^{n+a}}(\theta)(\mathcal I'\mid (\mathcal D_{\Delta^n}(V_\theta+p),S_a^{\mathbf{F}})=(\mathbf{F},\Sigma_a))\,\bigr|.
\]
For every $\mathbf{F}\in\mathcal D_{\Delta^n}(\mathcal F_{\mathcal I'})$ and $\Sigma_a\in S_a^{\mathbf{F}}(\mathcal I'\mid \mathcal D_{\Delta^n}(V_\theta+p)=\mathbf{F})$, restrict $\mathcal D_{\Delta^{n+a}}(\theta)(\mathcal I'\mid (\mathcal D_{\Delta^n}(V_\theta+p),S_a^{\mathbf{F}})=(\mathbf{F},\Sigma_a))$ to exactly $M_*$ cells.

The Katz--Tao and sub-critical cardinality hypotheses of Lemma~\ref{lem:away-from-critical}\textup{(\ref{it:katz-tao-away-from-critical})} follow from \eqref{eq:slablet-katz-tao-branching} and \eqref{eq:slablet-prefix-cardinality}. Lemma \ref{lem:slablet-above-threshold-refinement}\eqref{it:above-threshold-refinement-plane-shadings} and Lemma \ref{lem:spread-out-subset} show that the spread-outness hypothesis for the selected plane families also holds.

The inverse of $T_{\mathbf{F}}$ sends an affine cell of side length $\Delta^a$ to a set of affine diameter $O(\Delta^{n+a})$. Hence the affine cells in the conclusion of Lemma~\ref{lem:away-from-critical}\textup{(\ref{it:katz-tao-away-from-critical})} map with $O(1)$ multiplicity into $\mathcal D_{\Delta^{n+a}}(V_\theta+p)(\mathcal I'\mid \mathcal D_{\Delta^n}(V_\theta+p)=\mathbf{F})$, and Lemma~\ref{lem:away-from-critical}\textup{(\ref{it:katz-tao-away-from-critical})} yields
\begin{equation}\label{eq:slablet-prefix-one-scale}
    \bigl|\,\mathcal D_{\Delta^{n+a}}(V_\theta+p)(\mathcal I'\mid \mathcal D_{\Delta^n}(V_\theta+p)=\mathbf{F})\,\bigr|
    \ge
    \Delta^{(\eta-\lambda_0)a+C_1\e m}\cdot
    \bigl|\,S_a^{\mathbf{F}}(\mathcal I'\mid \mathcal D_{\Delta^n}(V_\theta+p)=\mathbf{F})\,\bigr|M_*.
\end{equation}
The exponent in \eqref{eq:slablet-prefix-one-scale} follows from the choice of $\lambda_0$ in \eqref{eq:non-sticky-gain-parameter}, and \eqref{eq:slablet-deficit-scale}.

The tuple $(S_a,\mathcal D_{\Delta^{n+a}}(\theta),\mathcal D_{\Delta^n}(V_\theta+p))$ determines $\mathcal D_{\Delta^{n+a}}(V_\theta+p)$ up to $O(1)$ choices. Double counting and the comparability assertions in Lemma \ref{lem:slablet-above-threshold-refinement}\eqref{it:above-threshold-refinement-plane-shadings} and \eqref{it:above-threshold-refinement-affine-slablet-fibers} convert \eqref{eq:slablet-prefix-one-scale} into a bound on the number of slablets meeting each affine cell: for $\mathbf{F}_n\in\mathcal D_{\Delta^n}(\mathcal F_{\mathcal I'})$,
\[
\begin{aligned}
    &\max_{\substack{\mathbf{F}_{n+a}\in\mathcal D_{\Delta^{n+a}}(V_\theta+p)(\mathcal I'\mid \mathcal D_{\Delta^n}(V_\theta+p)=\mathbf{F}_n)}}
    \bigl|\,S_a^{\mathbf{F}_n}(\mathcal I'\mid (\mathcal D_{\Delta^n}(V_\theta+p),\mathcal D_{\Delta^{n+a}}(V_\theta+p))=(\mathbf{F}_n,\mathbf{F}_{n+a}))\,\bigr|\\
    &\hspace{50mm}
    \le \Delta^{-(\eta-\lambda_0)a-3C_1\e m/2}.
\end{aligned}
\]
For $\mathbf{F}_{m}\in\mathcal D_{\Delta^{m}}(\mathcal F_{\mathcal I'})$, let $\mathbf{F}_n$ and $\mathbf{F}_{n+a}$ be its $\mathcal D_{\Delta^n}(V_\theta+p)$ and $\mathcal D_{\Delta^{n+a}}(V_\theta+p)$ ancestors. Since $a\le m-n$, the definition of $A_{\mathbf{F}_n}$ shows that, on the fiber $\{\mathcal D_{\Delta^{m}}(V_\theta+p)=\mathbf{F}_{m}\}$, $\mathcal D_{\Delta^a}(p)$ and $S_a^{\mathbf{F}_n}$ determine one another up to $O(1)$ choices. Hence
\[
    \bigl|\,\mathcal D_{\Delta^a}(p)(\mathcal I'\mid \mathcal D_{\Delta^{m}}(V_\theta+p)=\mathbf{F}_{m})\,\bigr|
   \lesssim
    \bigl|\,S_a^{\mathbf{F}_n}(\mathcal I'\mid (\mathcal D_{\Delta^n}(V_\theta+p),\mathcal D_{\Delta^{n+a}}(V_\theta+p))
    =(\mathbf{F}_n,\mathbf{F}_{n+a}))\,\bigr|,
\]
so \eqref{eq:slablet-prefix-collision-target} follows.

\textbf{(b) Proof of \eqref{eq:point-fiber-entropy-target}.}

For each $Q\in\mathcal D_{\Delta^a}(P_{\mathcal I'})$, define $\Lambda$ on the fiber $\{\mathcal D_{\Delta^a}(p)=Q\}$ of $\mathcal I'$ by
\[
    \Lambda(p,\theta)=
    \mathcal D_{\Delta^{m-a}}
    \bigl(V_\theta+\Delta^{-a}(p-z(Q))\bigr).
\]
The pair $(\Delta^{-a}(p-z(Q)),\mathcal D_{\Delta^{m-a}}(\theta))$ determines $\Lambda(p,\theta)$ up to $O(1)$ choices.

For $Q\in\mathcal D_{\Delta^a}(P_{\mathcal I'})$, the map $p\mapsto\Delta^{-a}(p-z(Q))$ sends the scale $\delta$ inside $Q$ to $\Delta^{m-a}$. For $0\le i\le m-a$, we have, by the Frostman condition on $P$ and \eqref{eq:Frostman-and-KT},
\[
    f(m)-f(a+i)\le(d-k)(m-a-i)+C_1\e m/1000.
\]
Lemma \ref{lem:branching-translation}, the $(\Delta,m)$-uniformity of $P$, $\mathcal I'\subset\mathcal I_0$, and the fact that each cell in $\mathcal D_{\Delta^{m}}(P)$ contains $O(1)$ points of $P$ therefore show that, for every $Q\in\mathcal D_{\Delta^a}(P_{\mathcal I'})$,
\[
    \Delta^{-a}\bigl((P_{\mathcal I'}\cap Q)-z(Q)\bigr)
    \text{ is }(\Delta^{m-a},d-k,\Delta^{-2C_1\e m})\text{-Katz--Tao}.
\]

For every $Q\in\mathcal D_{\Delta^a}(P_{\mathcal I'})$ and $q\in P_{\mathcal I'}\cap Q$, restrict $\mathcal D_{\Delta^{m-a}}(\Theta_{\mathcal I'}(q))$ to exactly
\begin{equation}\label{eq:point-fiber-plane-truncation}
    \min_{q'\in P_{\mathcal I'}\cap Q}
    |\mathcal D_{\Delta^{m-a}}(\Theta_{\mathcal I'}(q'))|
\end{equation}
cells. It follows from Lemma \ref{lem:spread-out-subset}, Lemma \ref{lem:slablet-above-threshold-refinement}\eqref{it:above-threshold-refinement-point-planes}, and $m-a\ge n\ge\xi m/(2d)$ that the hypotheses of Lemma \ref{lem:katz-tao-induction} are satisfied with $\rho=(\eta-\eta')/8$ for the set $\Delta^{-a}((P_{\mathcal I'}\cap Q)-z(Q))$ and the plane-cell subsets selected with cardinality \eqref{eq:point-fiber-plane-truncation}. The affine cells in the conclusion of Lemma \ref{lem:katz-tao-induction} map with $O(1)$ multiplicity into $\Lambda(\mathcal I'\mid \mathcal D_{\Delta^a}(p)=Q)$. We can therefore apply Lemma \ref{lem:katz-tao-induction} to obtain
\begin{equation}\label{eq:point-fiber-one-scale}
\begin{split}
    \bigl|\,\Lambda(\mathcal I'\mid \mathcal D_{\Delta^a}(p)=Q)\,\bigr|\ge
    \Delta^{(\eta+(\eta-\eta')/4)(m-a)+C_1\e m}
    |P_{\mathcal I'}\cap Q|\cdot
    \min_{q\in P_{\mathcal I'}\cap Q}
    |\mathcal D_{\Delta^{m-a}}(\Theta_{\mathcal I'}(q))|.
\end{split}
\end{equation}

For fixed $Q\in\mathcal D_{\Delta^a}(P_{\mathcal I'})$, the map $\Lambda$ determines $\mathcal D_{\Delta^{m-a}}(\theta)$. Conversely, $(p,\mathcal D_{\Delta^{m-a}}(\theta))$ determines $\Lambda$ up to $O(1)$ choices. If $\mathbf{F}_{m}\in\mathcal D_{\Delta^{m}}(\mathcal F_{\mathcal I'})$ is represented by $(\theta_{\mathbf{F}_{m}},b_{\mathbf{F}_{m}})$, then on the fibers $\{\mathcal D_{\Delta^a}(p)=Q\}$ and $\{\mathcal D_{\Delta^{m}}(V_\theta+p)=\mathbf{F}_{m}\}$ the translation parameter of $\Lambda$ differs by $O(\Delta^{m-a})$ from $\Delta^{-a}(b_{\mathbf{F}_{m}}-\rho_{\theta_{\mathbf{F}_{m}}}(z(Q)))$.
Hence $\Lambda$ and the pair
\[
    \left(\mathcal D_{\Delta^{m-a}}(\theta),
    \mathcal D_{\Delta^{m-a}}
    \bigl(\Delta^{-a}(b_{\mathbf{F}_{m}}-\rho_{\theta_{\mathbf{F}_{m}}}(z(Q)))\bigr)\right)
\]
determine one another up to $O(1)$ choices. The map $\mathcal D_{\Delta^{m}}(\theta)$ refines $\mathcal D_{\Delta^{m-a}}(\theta)$. Consequently,
\begin{equation}\label{eq:rescaled-affine-correspondence}
    (\mathcal D_{\Delta^a}(p),\Lambda,\mathcal D_{\Delta^{m}}(\theta))
    \longleftrightarrow
    (\mathcal D_{\Delta^a}(p),\mathcal D_{\Delta^{m}}(V_\theta+p),\mathcal D_{\Delta^{m}}(\theta))
\end{equation}
has $O(1)$ fibers in both directions: $\mathcal D_{\Delta^{m}}(\theta)$ specifies $\theta$ to accuracy $O(\delta)$, while the $\Delta^{m-a}$-precision of the rescaled translation coordinate becomes $\Delta^a\Delta^{m-a}=\delta$ in the original coordinates.

For fixed $Q\in\mathcal D_{\Delta^a}(P_{\mathcal I'})$, the map $\Lambda$ determines $\mathcal D_{\Delta^{m-a}}(\theta)$. Its values on a fixed $(p,\mathcal D_{\Delta^{m-a}}(\theta))=(q,J)$ fiber lie in $O(1)$ affine $\Delta^{m-a}$-cells. Lemma \ref{lem:slablet-above-threshold-refinement}\eqref{it:above-threshold-refinement-point-fibers}, bounded overlap, \eqref{eq:point-fiber-one-scale}, and \eqref{eq:rescaled-affine-correspondence} yield
\begin{equation}\label{eq:point-fiber-affine-count}
    \bigl|\,\mathcal D_{\Delta^{m}}(V_\theta+p)(\mathcal I'\mid \mathcal D_{\Delta^a}(p)=Q)\,\bigr|
    \ge
    \Delta^{(\eta+(\eta-\eta')/4)(m-a)+3C_1\e m}
    |P_{\mathcal I'}\cap Q|\cdot
    \min_{q\in P_{\mathcal I'}}
    |\mathcal D_{\Delta^{m}}(\Theta_{\mathcal I'}(q))|.
\end{equation}
Sum \eqref{eq:point-fiber-affine-count} over $Q\in\mathcal D_{\Delta^a}(P_{\mathcal I'})$ and use the lower bound for $|P_{\mathcal I'}|$ in Lemma \ref{lem:slablet-above-threshold-refinement}\eqref{it:above-threshold-refinement-point-fibers}. For $\e\ll_{\eta-\eta',C_1}1$, we obtain \eqref{eq:point-fiber-entropy-target}.
\end{proof}

\subsection{Conclusion of the proof of Proposition \ref{prop:non-sticky-furst-gain}}

We now combine Lemma \ref{lem:slablet-entropy-below-threshold} and Proposition \ref{prop:slablet-entropy-above-threshold} to conclude the proof. 

\begin{proof}[Proof of Proposition \ref{prop:non-sticky-furst-gain}]
If $R\le\sigma(m-n)$, let $\mathcal I'=\mathcal I$ and apply Lemma \ref{lem:slablet-entropy-below-threshold}.

If $R>\sigma(m-n)$, let $\mathcal I'$ be the refinement provided by Proposition \ref{prop:slablet-entropy-above-threshold}. By double counting,
\[
    \begin{aligned}
    &\sum_{Q\in\mathcal D_{\Delta^a}(P_{\mathcal I'})}
    \bigl|\,\mathcal D_{\Delta^{m}}(V_\theta+p)(\mathcal I'\mid \mathcal D_{\Delta^a}(p)=Q)\,\bigr|\\
    &\qquad\le|\mathcal D_{\Delta^{m}}(\mathcal F_{\mathcal I'})|
    \cdot
    \max_{\mathbf{F}\in\mathcal D_{\Delta^{m}}(\mathcal F_{\mathcal I'})}
    \bigl|\,\mathcal D_{\Delta^a}(p)(\mathcal I'\mid \mathcal D_{\Delta^{m}}(V_\theta+p)=\mathbf{F})\,\bigr|.
    \end{aligned}
\]
Combining \eqref{eq:slablet-prefix-collision-target} and \eqref{eq:point-fiber-entropy-target}, we deduce
\begin{equation}\label{eq:slablet-entropy-gain-above-threshold}
 \mathcal E_{\mathcal I'}(m)
 \ge(d-k-\eta)m+\lambda_0a
 -\frac{\eta-\eta'}{2}m-2C_1\e m.
\end{equation}
Equations \eqref{eq:non-sticky-scale-lengths}, \eqref{eq:non-sticky-gain-parameter}, and \eqref{eq:slablet-deficit-scale} turn \eqref{eq:slablet-entropy-gain-above-threshold} into 
\[
\mathcal E_{\mathcal I'}(m)\ge\left(d-k-\eta'+\frac{\eta-\eta'}2\right)m-8C_1\e m.
\]
By \eqref{eq:non-sticky-gain-parameter}, both cases yield
\begin{equation}\label{eq:non-sticky-entropy-gain}
    \mathcal E_{\mathcal I'}(m)
    \ge(d-k-\eta')m-20C_1\e m.
\end{equation}

When $R\le\sigma(m-n)$, we have, by Proposition \ref{prop:middle-scale-frostman-gain}\textup{(\ref{it:middle-scale-frostman-point-fibers})},
\[
    |\{(p,\theta)\in\mathcal I':p=q\}|
    \ge \Delta^{C_1\e m}|\{(p,\theta)\in\mathcal I_0:p=q\}|
    \qquad(q\in P_{\mathcal I'}),
\]
and Proposition \ref{prop:slablet-entropy-above-threshold} establishes the same inequality when $R>\sigma(m-n)$. Since $\theta\in\delta\Z^{k\times(d-k)}$, each $p$ fiber has exactly $|\mathcal D_{\Delta^m}(\Theta_{\mathcal I'}(p))|$ incidences. Thus $M_{\mathcal I'}(m)\ge\Delta^{C_1\e m}M$. Moreover, $\mathcal F_{\mathcal I'}\subset\mathcal F$. Equation \eqref{eq:non-sticky-entropy-gain} and the definition of $\mathcal E_{\mathcal I'}(m)$ yield
\[
    \log_{1/\Delta}\frac{|\mathcal F|_\delta}{M}
    \ge(d-k-\eta')m
    -21C_1\e m.
\]
For $m\gg_\rho1$, the right side is at least $(d-k-\eta'-\rho)m$. This concludes the proof.
\end{proof}

\subsection{Proof of Theorem \ref{thm:proj-general}}

We deduce Theorem \ref{thm:proj-general} from Theorem \ref{thm:furst-general}. The argument is standard. In short, if too many projections are small, their fibers form a spread-out configuration with too few planes.

Choose $\e>0$ so that Theorem \ref{thm:furst-general}, applied with  $t$ and  $\eta/2$ in place of $\eta$, is valid for Frostman and spread-out constants $\delta^{-4\e}$, and so that $4\e<\eta$. Without loss of generality, $P$ is $\delta$-discretized.

Call $\mathbf I\in\mathbf\Theta$ bad if there are $\theta\in\mathbf I$ and $P_\theta\subset P$ such that $|P_\theta|_\delta\ge\delta^\e|P|_\delta$ and $|\pi_\theta(P_\theta)|_\delta<\delta^{\eta-t}$. Choose one such pair for each bad cell, and denote the selected parameters by $\Theta_{\mathrm{bad}}$. Suppose that $|\Theta_{\mathrm{bad}}|>\delta^\e|\mathbf\Theta|$. Let $M=\lceil\delta^{2\e}|\mathbf\Theta|/8\rceil$. Double counting yields a $(\delta,t,2\delta^{-3\e})$-Frostman subset $P_0\subset P$ such that, for every $p\in P_0$, at least $M$ bad cells have selected parameters $\theta$ satisfying $p\in P_\theta$. Select exactly $M$ such cells at each $p$. By Lemma \ref{lem:spread-out-subset}, each selected family is $(\delta,s,8\delta^{-3\e})_{d,d-k}$-spread-out.

Apply $(x,y)\mapsto(y,x)$ to the points and $\theta\mapsto-\theta^*$ to the plane cells. These maps preserve the Frostman and spread-out conditions. After relabeling, denote the transformed configuration by $(P_0,\mathbf\Phi)$. From Theorem \ref{thm:furst-general} and \eqref{eq:affine-incidence-coordinates}, we deduce
\[
 \frac18\delta^{\eta/2+2\e-t}|\mathbf\Theta|
 \le\delta^{\eta/2-t}M
 \le|\mathcal F(P_0,\mathbf\Phi)|_\delta
 \lesssim\sum_{\theta\in\Theta_{\mathrm{bad}}}
 |\pi_\theta(P_\theta)|_\delta
 <\delta^{\eta-t}|\mathbf\Theta|.
\]
This is impossible for sufficiently small $\delta$, since $4\e<\eta$. Hence the non-bad cells satisfy the Frostman conclusion.

For the Katz--Tao conclusion, call $\mathbf I\in\mathbf\Theta$ bad if it contains some $\theta$ such that $|\pi_\theta(P)|_\delta<\delta^\eta|P|_\delta$, and choose one such $\theta$ from each bad cell. Denote the selected parameters by $\Theta_{\mathrm{bad}}$. If $|\Theta_{\mathrm{bad}}|>\delta^\e|\mathbf\Theta|$, use the bad cells as the common shading at every point of $P$, apply $(x,y)\mapsto(y,x)$ and $\theta\mapsto-\theta^*$, and denote the transformed common-shading configuration by $(P,\mathbf\Phi)$. Then Lemma \ref{lem:spread-out-subset} and the Katz--Tao conclusion of Theorem \ref{thm:furst-general} yield
\[
 \delta^{\eta/2}|P|_\delta|\Theta_{\mathrm{bad}}|
 \le|\mathcal F(P,\mathbf\Phi)|_\delta
 \lesssim\sum_{\theta\in\Theta_{\mathrm{bad}}}
 |\pi_\theta(P)|_\delta
 <\delta^\eta|P|_\delta|\Theta_{\mathrm{bad}}|,
\]
which is impossible for sufficiently small $\delta$. The non-bad cells therefore satisfy the Katz--Tao conclusion.


\section{Examples and applications}
\label{sec:examples-applications}


\subsection{Non-planar measures on \texorpdfstring{$\mathbf{G}_{d,k}$}{the big cell} are spread-out}

\label{s:non-planar-measures-are-spread-out}

A measure on $\mathbf{G}_{d,k}$ is spread out if it avoids concentration near model hyperplanes at most scales. Proposition \ref{prop:almost-nonplanar-centered-spread-out} quantifies this principle, allowing concentration at a small number of scales and locations. It is a higher dimensional extension of Lemma \ref{lem:line-decay-spread-out}, which applies to uniformly perfect measures on $\R$.
\begin{proposition}\label{prop:almost-nonplanar-centered-spread-out}
Let $1\le k<d$. There exists a constant $\kappa=\kappa(d,k)>0$ such that the following holds.

Fix an integer $L\ge0$, $\eta\in(0,1/2]$, and $\chi\ge0$, and set $\delta=2^{-L}$. Let $\mu$ be a locally finite Borel measure on $\mathbf G_{d,k}$. For $\theta\in\mathbf G_{d,k}$, define
\begin{displaymath}
\begin{aligned}
    &\mathcal N_{\delta,\eta,\kappa}^{\mu}(\theta)
    =\#\bigl\{j\in\{1,\ldots,L+1\}:\;\text{there is }H\in\mathfrak{H}_{d,k}
    \text{ such that}\\
    &\qquad\dist(\theta,H)\le\eta2^{-j}
    \text{ and }\mu\bigl(B(\theta,2^{-j})\cap H(\eta2^{-j})\bigr)
    >\kappa\mu\bigl(B(\theta,2^{-j-1})\bigr)\bigr\}.
\end{aligned}
\end{displaymath}
Set $A_\delta=B(0,1)\cap\{\theta:\mathcal N_{\delta,\eta,\kappa}^{\mu}(\theta)\le\chi L\}$. If $\mu(A_\delta)>0$, then $\mu|_{A_\delta}$ is $(\delta,\eta/8,\mu(B(0,5))\delta^{-\eta\chi/2})_{d,k}$-spread-out.
\end{proposition}

\begin{proof}
In a first read through, the reader may assume that the set in the definition of $\mathcal N_{\delta,\eta,\kappa}^{\mu}(\theta)$ is empty for every $\theta\in\spt\mu$. 

Let $C=C(d,k)\ge1$ be the largest overlap of families of Euclidean balls in $\mathbf{G}_{d,k}$ whose centers lie outside each other's balls, and set $\kappa=1/(2C)$.
Write $\nu=\mu|_{A_\delta}$, and let $\Theta\subset\spt\nu$ be a Borel set such that
\begin{displaymath}
    \mathfrak{s}_{d,k}(\theta\mid\Theta,\eta/8,\delta\to1)\ge\overline{\mathfrak{s}}
    \qquad\text{for all }\theta\in\Theta.
\end{displaymath}
Replacing $\Theta$ by $\Theta\cap A_\delta$ preserves $\nu(\Theta)$ and cannot decrease the number of singular scales. We may therefore assume $\Theta\subset A_\delta$ and $\Theta\ne\varnothing$.

For each $\theta\in\Theta$, discard every singular scale $r$ for which $r/2$ is counted by $\mathcal N_{\delta,\eta,\kappa}^{\mu}(\theta)$. At least $\max\{\overline{\mathfrak{s}}-\chi L,0\}$ singular scales remain. Select the largest remaining scale $r$ and discard the remaining scales in $[\tfrac{\eta}{8}r,r]$. Repeat until no scales remain. Since $\eta\le1/2$, each interval $[\tfrac{\eta}{8}r,r]$ contains at most $\tfrac8\eta$ dyadic scales. Pruning to keep the same number of selected scales at each point, we obtain an integer 
\begin{equation}\label{eq:almost-nonplanar-scale-count}
    \mathfrak{s}\ge\frac\eta8\max\{\overline{\mathfrak{s}}-\chi L,0\}
\end{equation}
and, for each $\theta\in\Theta$, scales $r_j(\theta)$, $1\le j\le\mathfrak{s}$, such that
\begin{equation}\label{eq:almost-nonplanar-centered-scale-subsequence}
    r_{j+1}(\theta)\le\frac\eta8r_j(\theta) \quad (1\le j<\mathfrak{s}).
\end{equation}
 By construction, each $r_j(\theta)$ is $(\eta/8)$-singular for $(\Theta,\theta)$, and $r_j(\theta)/2$ is not counted by $\mathcal N_{\delta,\eta,\kappa}^{\mu}(\theta)$. Choose $H_j^{\theta}\in \mathfrak{H}_{d,k}$ such that
\begin{equation}\label{eq:almost-nonplanar-centered-hyperplane-containment}
    \Theta\cap B(\theta,r_j(\theta))\subset H_j^{\theta}\left(\frac\eta8r_j(\theta)\right).
\end{equation}

We construct nested neighborhoods of $\Theta$ whose $\mu$-masses decrease by a factor of two at each step. For each $1\le j\le\mathfrak{s}$, start with $T_j=\varnothing$. Choose a point of $\Theta$ outside $\bigcup_{\omega\in T_j}B(\omega,r_j(\omega)/4)$ at which $r_j$ is maximal, and add it to $T_j$. Repeat until $\Theta$ is covered; this must happen in finitely many steps since $\Theta\subset B(0,1)$ and all the balls have radius $\ge \delta/4$. By definition, the balls in the cover have overlap at most $C$. For every $\theta\in\Theta$, there is $\omega\in T_j$ with 
\begin{equation} \label{eq:Tj-covering-property}
|\theta-\omega|\le r_j(\omega)/4, \qquad r_j(\theta)\le r_j(\omega).
\end{equation}
Set $U_0=B(0,5)$ and 
\[
    U_j= \bigcup_{\theta\in\Theta} B\bigl(\theta,\eta r_j(\theta)/8\bigr),
    \qquad 1\le j\le\mathfrak{s}.
\]
We have $\Theta\subset U_j\subset U_{j-1}$ and $B(\omega,r_j(\omega)/4)\subset U_{j-1}$ for $\omega\in T_j$, by \eqref{eq:almost-nonplanar-centered-scale-subsequence}. Combining \eqref{eq:almost-nonplanar-centered-hyperplane-containment} with \eqref{eq:Tj-covering-property}, we obtain
\[
    U_j\subset\bigcup_{\omega\in T_j}
    \left[B\left(\omega,\frac12r_j(\omega)\right)
    \cap H_j^{\omega}\left(\frac\eta4r_j(\omega)\right)\right].
\]
Since $\dist(\omega,H_j^{\omega})<\eta r_j(\omega)/8$ by \eqref{eq:almost-nonplanar-centered-hyperplane-containment}, and $r_j(\omega)/2$ is not counted by $\mathcal N_{\delta,\eta,\kappa}^{\mu}(\omega)$ for every $\omega\in T_j$, we estimate
\[
    \mu(U_j)
    \le\kappa\sum_{\omega\in T_j}\mu\bigl(B(\omega,r_j(\omega)/4)\bigr)
    \le C\kappa\mu(U_{j-1})
    =\frac12\mu(U_{j-1}).
\]
As $\Theta\subset U_{\mathfrak{s}}$, we conclude from \eqref{eq:almost-nonplanar-centered-scale-subsequence} that $\nu(\Theta)\le2^{-\mathfrak{s}}\mu(B(0,5))$. In light of \eqref{eq:almost-nonplanar-scale-count}, this shows that $\nu$ is $(\delta,\eta/8,\mu(B(0,5))\delta^{-\eta\chi/2})_{d,k}$-spread-out, as desired.
\end{proof}

We establish a related result, with a similar proof, showing that a simpler but uniform non-concentration near model hyperplanes implies spread-outness. Similar conditions have appeared elsewhere in the literature, especially in the context of Diophantine approximation and homogeneous dynamics; see e.g. \cite[\S 2]{KLW04}.
\begin{lemma}\label{lem:hyperplane-non-concentration-implies-spreadout}
Let $1\le k<d$, let $U\subset\mathbf G_{d,k}$ be open, and let $\mu$ be a nonzero finite Borel measure on $U$. Suppose that, for some $\e\in(0,1/2)$ and $r_0\in(0,1]$,
\begin{equation}\label{eq:model-hyperplane-contraction}
    \mu\bigl(B(\theta,r)\cap H(\e r)\bigr)
    \le(1-\e)\mu(B(\theta,r))
\end{equation}
for $\theta\in\spt\mu\cap U$, $0<r\le r_0$, and $H\in\mathfrak{H}_{d,k}$, whenever $B(\theta,r)\subset U$. For $a\in\spt\mu\cap U$ and $\rho\in2^{-\N}\cap(0,r_0]$ with $B(a,4\rho)\subset U$, the restriction $\mu|_{B(a,\rho)}$ is $(\delta,\beta,C\mu(B(a,2\rho)))_{d,k}$-spread-out for every $\delta\in2^{-\N}$, where $\beta=\beta(d,k,\e)>0$ and $C=C(d,k,\e,\rho)\ge1$.
\end{lemma}
\begin{proof}
Fix $\delta\in2^{-\N}\cap(0,\rho]$, and let $\Theta\subset\spt\mu\cap B(a,\rho)$ be a Borel set such that $\mathfrak{s}_{d,k}(\theta\mid\Theta,\e/16,\delta\to \rho)\ge\overline{\mathfrak{s}}$ for all $\theta\in\Theta$. Select singular scales greedily as in Proposition \ref{prop:almost-nonplanar-centered-spread-out}. For each $\theta\in\Theta$, select a sequence of singular scales $r_j(\theta)$, $j=1,\ldots,\mathfrak{s}$, with $\mathfrak{s}\ge\tfrac\e{128}\overline{\mathfrak{s}}$, such that
\begin{equation} \label{eq:schubert-scale-subsequence}
   r_{j+1}(\theta)\le\frac\e{128}r_j(\theta), \quad 1\le j<\mathfrak{s}.
\end{equation}
For each $\theta$ and $j$, choose  $H_j^{\theta}\in\mathfrak{H}_{d,k}$ such that 
\begin{equation} \label{eq:choice-Hj-theta}
    \Theta\cap B(\theta,r_j(\theta))\subset H_j^{\theta}(\e r_j(\theta)/16).
\end{equation}

Set $U_0=B(a,2\rho)$ and
\[
    U_j=\bigcup_{\theta\in\Theta} \operatorname{int} B\bigl(\theta,\e r_j(\theta)/16\bigr),
    \qquad 1\le j\le\mathfrak{s}.
\]
For $\theta\in U_j$, define
\begin{equation} \label{eq:def-rho-j}
    \rho_j(\theta)=\max\left\{r_j(\omega):\omega\in\Theta,\
    |\theta-\omega|<\frac12r_j(\omega)\right\}.
\end{equation}
This is well-defined since the maximum is over a nonempty finite set of dyadic values.

Fix $1\le j\le\mathfrak{s}$ and $\theta\in U_j$, and choose $\omega\in\Theta$ with 
\[
r_j(\omega)=\rho_j(\theta), \quad  |\theta-\omega|<\rho_j(\theta)/2.
\]
It follows from $U_0=B(a,2\rho)$ for $j=1$ and \eqref{eq:schubert-scale-subsequence} for $j\ge2$ that  $B(\theta,\rho_j(\theta)/8)\subset U_{j-1}$.

For $\theta'\in U_j\cap B(\theta,\rho_j(\theta)/8)$, choose $\omega'\in\Theta$ with $|\theta'-\omega'|<\e r_j(\omega')/16$. From \eqref{eq:def-rho-j} we see that $r_j(\omega')\le\rho_j(\theta)$ and $\omega'\in B(\omega,r_j(\omega))$. By the choice \eqref{eq:choice-Hj-theta} of $H_j^{\omega}$, we have $\dist(\theta',H_j^{\omega})<\e\rho_j(\theta)/8$.

To estimate $\mu(U_j)$, set $B_\theta=B(\theta,\rho_j(\theta)/8)$ for $\theta\in\spt\mu\cap U_j$. The non-concentration assumption \eqref{eq:model-hyperplane-contraction} yields $\mu(U_j\cap B_\theta)\le(1-\e)\mu(B_\theta)$.

Applying the Besicovitch covering theorem \cite[Theorem~2.7]{zbMATH01249699}, we may select an at most countable family $\mathcal B_j$ of the balls $B_\theta$, $\theta\in\spt\mu\cap U_j$, covering $\spt\mu\cap U_j$, with overlap at most $C_0$ depending only on $d,k$. Since every $B\in\mathcal B_j$ lies in $U_{j-1}$,
\begin{equation}\label{eq:non-concentration-neighborhood-mass-loss}
\begin{aligned}
    C_0\mu(U_{j-1}\setminus U_j)
    \ge\sum_{B\in\mathcal B_j}\mu(B\setminus U_j)
    \ge\frac{\e}{1-\e}\sum_{B\in\mathcal B_j}\mu(B\cap U_j)
    \ge\frac{\e}{1-\e}\mu(U_j).
\end{aligned}
\end{equation}
Hence $\mu(U_j)\le\gamma\mu(U_{j-1})$ for some $\gamma=\gamma(\e,C_0)<1$. Since $\Theta\subset U_{\mathfrak{s}}$, we deduce that $\mu(\Theta)\le\gamma^{\mathfrak{s}}\mu(B(a,2\rho))$. Recalling that $\mathfrak{s}\ge \tfrac{\e}{128}\overline{\mathfrak{s}}$, we conclude that $\mu|_{B(a,\rho)}$ is $(\delta,\beta,C\mu(B(a,2\rho)))_{d,k}$-spread-out, with $\beta=\beta(d,k,\e)>0$ and $C=C(d,k,\e,\rho)\ge1$.
\end{proof}

With a very similar proof, we next show that the non-concentration condition near Schubert varieties can be taken on the Grassmannian with its natural Riemannian metric, rather than in a big-cell chart. 
\begin{lemma}\label{lem:schubert-non-concentration-implies-spreadout}
Let $1\le k<d$, and let $\nu$ be a nonzero finite Borel measure on $\G(d,k)$ satisfying \eqref{eq:schubert-contraction} for all $V\in\spt\nu$, $W\in\G(d,d-k)$, and $0<r\le1$, with some $\e\in(0,1/2)$. Fix a big-cell chart. Every $V\in\spt\nu$ in this chart has a neighborhood $U$ such that the coordinate image of $\nu|_U$ is $(\delta,\beta,C)_{d,k}$-spread-out for all $\delta\in2^{-\N}$, with some $\beta>0$ and $C\ge1$.
\end{lemma}
\begin{proof}
Let $d$ be the distance induced by the natural $\mathbf{O}(d)$-invariant Riemannian metric on $\G(d,k)$, and let $|\cdot|$ be the Euclidean norm on $\mathbf G_{d,k}$. Denote their balls by $B^d(V,r)$ and $B^{|\cdot|}(\theta,r)$, respectively. The neighborhoods $H(t)$ use $|\cdot|$, and neighborhoods of Schubert cycles use $d$.

Fix $V_a\in\spt\nu$ in the chart, and let $\mu$ be the coordinate image of the restriction of $\nu$ to the chart. Choose $\rho\in2^{-\N}$ sufficiently small and $M\ge1$ such that, for $\theta,\omega\in B^{|\cdot|}(a,4\rho)$,
\begin{equation}\label{eq:schubert-chart-metric-comparison}
    M^{-1}|\theta-\omega|
    \le d(V_\theta,V_\omega)
    \le M|\theta-\omega|.
\end{equation}
Fix $\delta\in2^{-\N}\cap(0,\rho]$ and a Borel set $\Theta\subset\spt\mu\cap B^{|\cdot|}(a,\rho)$ with at least $\overline{\mathfrak{s}}$ $(\e/(16M^2))$-singular dyadic scales in $[\delta,\rho]$ at every point of $\Theta$. Repeat the selection of scales $r_j(\theta)$ and the definitions of $H_j^{\theta},U_j,\rho_j$ from Lemma \ref{lem:hyperplane-non-concentration-implies-spreadout}, with $\e/M^2$ in place of $\e$.

For $1\le j\le\mathfrak{s}$ and $\theta\in\spt\mu\cap U_j$, let $B_\theta$ be the coordinate image of $B^d(V_\theta,\rho_j(\theta)/(8M))$, and choose $\omega$ attaining the maximum in \eqref{eq:def-rho-j}. By \eqref{eq:schubert-chart-metric-comparison}, $B_\theta\subset B^{|\cdot|}(\theta,\rho_j(\theta)/8)\subset U_{j-1}$. By \eqref{eq:schubert-variety} and \eqref{eq:schubert-chart-metric-comparison}, the planes with coordinates in $U_j\cap B_\theta$ lie within $d$-distance $\e\rho_j(\theta)/(8M)$ of the Schubert cycle corresponding to $H_j^{\omega}$. Consequently, \eqref{eq:schubert-contraction} yields $\mu(U_j\cap B_\theta)\le(1-\e)\mu(B_\theta)$.

The Besicovitch covering theorem \cite[Theorem~2.7]{zbMATH01249699} on $\G(d,k)$ establishes \eqref{eq:non-concentration-neighborhood-mass-loss} in the current setting, with overlap bound depending only on $d,k$. Since $\mathfrak{s}\ge\e\overline{\mathfrak{s}}/(128M^2)$, the same argument as in the proof of Lemma \ref{lem:hyperplane-non-concentration-implies-spreadout} shows that $\mu|_{B^{|\cdot|}(a,\rho)}$ is spread out at every dyadic scale, as required.
\end{proof}

\begin{corollary}
\label{cor:small-dim-spread-out}

If $\mu$ is a self-similar measure on $\mathbf G_{d,k}$ which is not supported on a proper affine subspace, then there are $C,s>0$ such that $\mu$ is $(\delta,s,C)_{d,k}$-spread-out for all $\delta\in2^{-\N}$.The same conclusion holds for every non-atomic self-conformal measure on $\mathbf G_{2,1}$.

In particular, there are measures on $\mathbf G_{d,k}$ whose topological support has arbitrarily small Hausdorff dimension and which are $(\delta,s,C)_{d,k}$-spread-out for all $\delta\in2^{-\N}$ and some $s,C>0$.
\end{corollary}

\begin{proof}
By Lemma \ref{lem:spread-out-scaling}, we may assume that $\spt\mu\subset B(0,1)$. Kleinbock, Lindenstrauss and Weiss \cite[\S 8]{KLW04} proved absolute decay for the natural Hausdorff measures on irreducible self-similar sets satisfying the open set condition. Lemma \ref{lem:hyperplane-non-concentration-implies-spreadout} therefore implies that these measures are spread out. To allow overlaps and arbitrary probability weights, we use Baker's theorem \cite[Theorem~1.1]{Baker26}, which expresses $\mu$ as an average of measures with uniform doubling and hyperplane decay. For self-similar measures, use properties~(2) and~(4); for self-conformal measures on $\mathbf G_{2,1}$, use properties~(2) and~(3), since hyperplanes are points. More precisely, there are a probability space $(\Omega,\mathbb P)$, probability measures $\mu_\omega$, and constants $C,\alpha>0$ such that
\begin{displaymath}
    \mu=\int_\Omega\mu_\omega\,d\mathbb P(\omega),
\end{displaymath}
and, uniformly in $\omega$,
\begin{displaymath}
    \mu_\omega(B(x,2r))\le C\mu_\omega(B(x,r)),
    \qquad
    \mu_\omega(B(x,r)\cap H(\rho r))
    \le C\rho^\alpha\mu_\omega(B(x,r))
\end{displaymath}
for $x\in\spt\mu_\omega$, $r\in(0,1]$, $\rho>0$, and every affine hyperplane $H$.

Applying Lemma \ref{lem:hyperplane-non-concentration-implies-spreadout} with non-concentration parameter $\min\{1/4,(2C)^{-1/\alpha}\}$ and summing over a finite cover of $\spt\mu_\omega\subset B(0,1)$ by balls of radius $1/2$, we obtain, after increasing $C$, that every $\mu_\omega$ is $(\delta,\beta,C)_{d,k}$-spread-out for some $\beta>0$, uniformly in $\omega$ and $\delta\in2^{-\N}$.

If $\Theta\subset\spt\mu$ has at least $\mathfrak{s}$ $\beta$-singular scales at every point, then
\begin{displaymath}
    \mathfrak{s}_{d,k}(\theta\mid\Theta\cap\spt\mu_\omega,\beta,\delta\to1)
    \ge\mathfrak{s}_{d,k}(\theta\mid\Theta,\beta,\delta\to1)
    \ge\mathfrak{s}
\end{displaymath}
for $\theta\in\Theta\cap\spt\mu_\omega$. Hence
\begin{displaymath}
    \mu(\Theta)
    =\int_\Omega\mu_\omega(\Theta)\,d\mathbb P(\omega)
    \le C2^{-\beta\mathfrak{s}}.
\end{displaymath}
This proves the first assertion.

To obtain supports of small dimension, choose $k(d-k)+1$ affinely independent points in $\mathbf G_{d,k}$ and similarities fixing these points, all with contraction ratio $r$. For all sufficiently small $r$, the IFS satisfies the strong separation condition, its attractor is not contained in a proper affine subspace, and its Hausdorff dimension is arbitrarily small.
\end{proof}

\begin{remark}
    For self-similar measures, a variant of Baker's argument requires only that $\mu$ not be supported on a translate of a plane in $G\cdot \mathcal{H}$, where $G$ is the closed subgroup generated by the orthogonal parts of the similarities in the IFS, and $\mathcal{H}$ is the set of model hyperplanes.
\end{remark}

\subsection{Continuous results and the proofs of Theorems \ref{thm:projection-ADR}, \ref{thm:projection-non-conc-Schubert-cycle} and \ref{thm:Furstenberg-continuous}}

\label{s:continuous-versions}

\subsubsection{Continuous Furstenberg sets and the proof of Theorem \ref{thm:Furstenberg-continuous}}

We being by deducing from Theorem \ref{thm:furst-general} an optimal continuous (dual) Furstenberg estimate under a spread-out assumption on the shadings.
\begin{theorem} \label{thm:mainGeneralContinuous}
    Fix parameters $t\in (0,d]$ and $s\in(0,k(d-k)]$.

    Let $E\subset\R^d$ be a Borel set with $\hdim(E)\ge t$. For each $x\in E$, let $\nu_x$ be a Borel probability measure on $\mathbf{G}_{d,k}$ which is $(\delta,\beta_x,C_x)_{d,k}$-spread-out for all $\delta\in2^{-\N}$, and some $\beta_x>0, C_x\ge 1$. Assume also that $\nu_x$ satisfies the Frostman condition
    \[
        \nu_x(B(\theta,r))\le C_x r^s \qquad(\theta\in\spt\nu_x, r>0).
    \]
    Then,
    \[
        \hdim\left(\bigcup_{x\in E}\{x+V:V\in\spt\nu_x\}\right)\ge s+\min\{t,d-k\}.
    \]
\end{theorem}
\begin{proof}
Set
\[
    \mathcal F=\bigcup_{x\in E}\{x+V:V\in\spt\nu_x\}.
\]
Fix $0<u<s+\min\{t,d-k\}$ and choose $\underline{t}, u,\eta>0$ such that
\begin{equation} \label{eq:parameters-for-continuous-furstenberg}
    \underline t<\min\{t,d-k\},
    \qquad u<s+\underline t-3\eta.
\end{equation}
By pigeonholing and a change of coordinates, we may assume that $\mathcal{H}^{\underline t}_\infty(E)>0$, and there are constants $\beta>0,C\ge 1$ such that every $\nu_x$ is $(s,C)$-Frostman and $(\delta,\beta,C)_{d,k}$-spread-out for all $\delta\in2^{-\N}$, with $\spt\nu_x\subset [0,1]^{k\times(d-k)}$.

Let $\Delta_0=\Delta_0(\beta,\underline t,\eta)$ be sufficiently small.
Let $\mathcal Q$ be a dyadic cover of $\mathcal F$ by cubes of side length at most $\Delta_0$. It suffices to prove
\begin{equation}\label{form114a}
    \sum_{Q\in\mathcal Q}\ell(Q)^u\gtrsim_{C,\e}1.
\end{equation}
Write $\mathcal Q_\delta=\{Q\in\mathcal Q:\ell(Q)=\delta\}$. Dyadic pigeonholing yields $\delta\in2^{-\N}\cap(0,\Delta_0]$ and $E_\delta\subset E$ such that
\begin{equation}\label{form114b}
    \mathcal H^{\underline t}_\infty(E_\delta)\gtrapprox_{\delta} 1,
    \qquad
    \nu_x\bigl(\{\theta:x+V_\theta\in\textstyle\bigcup\mathcal Q_\delta\}\bigr)
    \gtrapprox_{\delta} 1\quad(x\in E_\delta).
\end{equation}
By \cite[Proposition 3.13]{FasslerOrponen14}, $E_{\delta}$ contains a $(\delta,\underline t,\lessapprox_{\delta} 1)$-Frostman set $F_{\delta}$. Let $P$ consist of the lower-left endpoints of the dyadic $\delta$-cubes meeting $F_{\delta}$. For $p\in P$, choose $\overline p\in F_{\delta}\cap \mathcal{D}_{\delta}(p)$.
 
For $p\in P$, apply Lemma \ref{lem:uniform-spread-out-discretization} to $\nu_{\overline p}$ and the set in the second part of \eqref{form114b} with $x=\overline p$. We obtain a $(\delta,\beta,\lessapprox_{\delta} 1)_{d,k}$-spread-out family $\mathbf\Theta(p)$ of dyadic $\delta$-cubes. By the Frostman condition on $\nu_{\overline p}$, $|\mathbf\Theta(p)|\gtrapprox_\delta\delta^{-s}$ for all $p\in P$. A further dyadic pigeonholing reduces to an integer $M\gtrapprox_\delta\delta^{-s}$ such that $M\le|\mathbf\Theta(p)|<2M$ for every $p\in P$; the restricted set $P$ remains $(\delta,\underline t,\lessapprox_{\delta} 1)$-Frostman.

For sufficiently small $\Delta_0$, the $\lessapprox_{\delta} 1$ factors are bounded by $\delta^{-\e}$. We may therefore apply Theorem \ref{thm:furst-general} with parameters $(\beta,\underline t,\eta)$ to obtain
\[
    |\mathcal Q_\delta|
    \gtrsim |\mathcal F(P,\mathbf\Theta)|_\delta
    \ge\delta^{\eta-\underline t}M
    \gtrapprox_\delta\delta^{\eta-s-\underline t}.
\]
In light of the parameter choice \eqref{eq:parameters-for-continuous-furstenberg}, \eqref{form114a} holds for $\Delta_0$ sufficiently small. Therefore $\hdim\mathcal F\ge u$. Since $u< s+\min\{t,d-k\}$ was arbitrary, this completes the proof.
\end{proof}

We can now easily deduce Theorem \ref{thm:Furstenberg-continuous}.
\begin{proof} [Proof of Theorem \ref{thm:Furstenberg-continuous}]
    Considering a finite cover by big cells, we may reduce to the case in which all $\nu_x$ are supported on a single big cell.
   
    Fix $0<s'<s$. By Lemma \ref{lem:schubert-non-concentration-implies-spreadout}, each $\nu_x$ can be restricted to a subset of positive mass so that this restriction is spread out at every dyadic scale. Since $\hdim\nu_x\ge s$, a further restriction to a set of relative mass $1/2$ is $s'$-Frostman, and still spread-out. Apply Theorem \ref{thm:mainGeneralContinuous} and let $s'\uparrow s$.
\end{proof}

\subsubsection{Continuous projections and proof of Theorems \ref{thm:projection-ADR} and \ref{thm:projection-non-conc-Schubert-cycle}}

Recall that the Hausdorff dimension of a measure $\mu$ was defined in \eqref{eq:hdim-measure}. We deduce from Theorem \ref{thm:proj-general} a continuous projection theorem for measures. Theorems \ref{thm:projection-ADR} and \ref{thm:projection-non-conc-Schubert-cycle} will follow as corollaries.
\begin{theorem}\label{thm:continuous-projection}
    Let $\nu$ be a Borel probability measure on $\mathbf{G}_{d,k}$ which is $(\delta,s,C)_{d,k}$-spread-out for every $\delta\in2^{-\N}$, where $s>0$ and $C\ge 1$. Then, for every Borel probability measure $\mu$ on $\R^d$,
    \[
        \hdim(\pi_\theta\mu)=\min\{\hdim\mu,k\} \quad\text{for $\nu$-almost every $\theta\in\mathbf{G}_{d,k}$}.
    \]
\end{theorem}
\begin{proof}
The upper bound follows from Lipschitz continuity. For the lower bound, fix $0<t<\hdim\mu$. By restricting $\mu$ to Borel sets whose union has full $\mu$-measure, we may assume that $\mu$ is $t$-Frostman.

Fix $0<u<t$ and suppose that the Borel set $\Theta=\{\theta\in\mathbf G_{d,k}:\hdim(\pi_\theta\mu)<u\}$ has positive $\nu$-measure. Restrict $\Theta$ to a unit dyadic cube $a+[0,1)^{k\times(d-k)}$ of positive $\nu$-measure. Replace $\mu$ by its push-forward under $(x,y)\mapsto(x+ay,y)$ and $\nu$ by the normalized push-forward of $\nu|_\Theta$ under $\theta\mapsto\theta-a$. The identity $\pi_{\theta-a}(x+ay,y)=\pi_\theta(x,y)$ preserves the definition of $\Theta$, while Lemmas \ref{lem:spread-out-subset} and \ref{lem:spread-out-scaling} show that spread-outness is preserved up to a rescaling of $C$. Thus, after relabeling, we may assume that $\Theta\subset[0,1)^{k\times(d-k)}$. Likewise, we may assume that $\mu([0,1)^d)=1$.

For each $\theta\in\Theta$, choose a Borel set $A_\theta\subset\R^k$ with $(\pi_\theta\mu)(A_\theta)>0$ and $\mathcal H^u(A_\theta)=0$. By restricting $\Theta$ again, we may assume that $\inf_{\theta\in\Theta}(\pi_\theta\mu)(A_\theta)>0$. Let $\e>0$ be provided by Theorem \ref{thm:proj-general}, applied in dimensions $(d,d-k)$ with parameters $s,t,(t-u)/2$.

Use dyadic covers $\mathcal Q_\theta$ of $A_\theta$ with $\sum_{Q\in\mathcal Q_\theta}\ell(Q)^u\le1$ and arbitrarily small maximal side length. Group each $\mathcal Q_\theta$ by side length. Selecting a side length by dyadic pigeonholing the $\pi_\theta\mu$-mass of these groups, then the planes by the selected side length, yields arbitrarily small $\delta\in2^{-\N}$ and $\Theta_\delta\subset\Theta$ with $\nu(\Theta_\delta)\gtrapprox_\delta1$ such that
\[
    (\pi_\theta\mu)\Bigl(\bigcup\{Q'\in\mathcal Q_\theta:\ell(Q')=\delta\}\Bigr)
    \gtrapprox_\delta1\quad(\theta\in\Theta_\delta).
\]
Partition the cubes $Q\in\mathcal D_\delta([0,1)^d)$ with $\mu(Q)>\delta^{d+1}$ into $O(\log(1/\delta))$ classes $\mathcal{D}^{(m)}$ according to $2^{-m-1}<\mu(Q)\le2^{-m}$. For each $\theta\in\Theta_\delta$, select $m$ such that the cubes in $\mathcal{D}^{(m)}$ meeting the inverse images of the side-$\delta$ cubes in $\mathcal Q_\theta$ have total $\mu$-mass $\gtrapprox_\delta1$. Pigeonholing the planes again, we may pass to a further subset, still denoted by $\Theta_\delta$, with $\nu(\Theta_\delta)\gtrapprox_\delta1$, so that the same $m$ works for all $\theta\in\Theta_\delta$. Define
\[
    P=\{z(Q):Q\in\mathcal D_\delta([0,1)^d),\
    2^{-m-1}<\mu(Q)\le2^{-m}\}.
\]
For $\theta\in\Theta_\delta$, let $P_\theta$ consist of those $z(Q)\in P$ for which $Q$ meets $\pi_\theta^{-1}(Q')$ for some $Q'\in\mathcal Q_\theta$ with $\ell(Q')=\delta$. Thus $\sum_{z(Q)\in P_\theta}\mu(Q)\gtrapprox_\delta1$ for every $\theta\in\Theta_\delta$, and these sets satisfy
\begin{equation}\label{eq:continuous-projection-discretization}
\begin{aligned}
    &\nu(\Theta_\delta)\gtrapprox_\delta1,
    \quad P \text{ is }(\delta,t,\lessapprox_\delta1)\text{-Frostman},\\
    &|P_\theta|  \gtrapprox_\delta|P|,
    \quad |\pi_\theta(P_\theta)|_\delta\lesssim\delta^{-u}
    \quad(\theta\in\Theta_\delta).
\end{aligned}
\end{equation}
The last bound follows from the definition of $u$, because $\mathcal Q_\theta$ contains at most $\delta^{-u}$ cubes of side length $\delta$.

Lemma \ref{lem:uniform-spread-out-discretization}, applied to $\nu$ and $\Theta_\delta$, yields a $(\delta,s,\lessapprox_\delta1)_{d,k}$-spread-out family $\mathbf\Theta\subset\mathcal D_\delta(\Theta_\delta)$. For sufficiently small $\delta$, Theorem \ref{thm:proj-general} provides a subfamily $\mathbf\Theta'\subset\mathbf\Theta$ with $|\mathbf\Theta'|\ge(1-\delta^\e)|\mathbf\Theta|$ such that
\[
    |\pi_{z(\mathbf I)}(P')|_\delta\ge\delta^{-\frac{t+u}{2}}
    \quad\text{for }\mathbf I\in\mathbf\Theta',\ P'\subset P,
    \ |P'|\ge\delta^\e|P|.
\]
Choose $\mathbf I\in\mathbf\Theta'$ and $\theta\in\Theta_\delta\cap\mathbf I$. For sufficiently small $\delta$, we have 
\[
    |\pi_{z(\mathbf I)}(P_\theta)|_\delta
    \lesssim_{d,k}|\pi_\theta(P_\theta)|_\delta
    \lesssim_{d,k}\delta^{-u}.
\]
This is a contradiction since $u<t$. Thus $\nu(\Theta)=0$. Letting $u\uparrow t$ completes the proof.
\end{proof}

We record a set version of Theorem \ref{thm:continuous-projection} for later reference.
\begin{corollary}\label{cor:continuous-projection}
Under the assumptions on $\nu$ in Theorem \ref{thm:continuous-projection}, for every Borel set $E\subset\R^d$ we have
\[
    \hdim\pi_\theta(E)=\min\{\hdim E,k\} \quad\text{for $\nu$-almost every $\theta\in\mathbf G_{d,k}$}.
\]
\end{corollary}
\begin{proof}
    This follows from Theorem \ref{thm:continuous-projection} and Frostman's Lemma.
\end{proof}

\begin{proof}[Proof of Theorems \ref{thm:projection-ADR} and \ref{thm:projection-non-conc-Schubert-cycle}]
In Theorem \ref{thm:projection-ADR}, $\nu$ is uniformly perfect, so it has no atoms and any fixed slope chart has full $\nu$-measure. Let $\mu$ be its push-forward to the chart. For each $a\in\spt\mu$, local bi-Lipschitz comparison of the metrics and iteration of uniform perfectness allow us to choose $R\in2^{-\N}$ and $0<\e<1/8$ such that
\begin{displaymath}
    \mu(B(x,8\e r))\le\frac14\mu(B(x,r))
    \qquad\text{whenever }B(x,r)\subset B(a,5R).
\end{displaymath}
Choose $R$ so that the corresponding Grassmannian balls do not contain $\spt\nu$, and then choose $\e$ to obtain the factor $1/4$ by iteration. Apply Lemma \ref{lem:line-decay-spread-out} after the change of variables $\theta\mapsto(\theta-a)/R$. Lemma \ref{lem:spread-out-scaling} then shows that $\mu(B(a,R))^{-1}\mu|_{B(a,R)}$ is spread-out at every dyadic scale, with fixed parameters. These intervals $B(a,R)$ admit a countable subcover of $\spt\mu$.

Under the assumptions of Theorem \ref{thm:projection-non-conc-Schubert-cycle}, Lemma \ref{lem:schubert-non-concentration-implies-spreadout} provides a finite cover of $\spt\nu$ by big-cell neighborhoods whose normalized coordinate restrictions are spread-out at every dyadic scale.

For either theorem, fix $0<t<\hdim E$ and choose a $t$-Frostman probability measure carried by a compact subset of $E$. Apply Theorem \ref{thm:continuous-projection} to this measure and each normalized coordinate restriction of $\nu$. The coordinate projections and the corresponding orthogonal projections have images of the same Hausdorff dimension, so $\hdim\Pi_V(E)\ge\min\{t,k\}$ for $\nu$-almost every $V$, with $k=1$ in Theorem \ref{thm:projection-ADR}. Let $t$ increase to $\hdim E$ along a sequence. Lipschitz continuity implies the reverse inequality and settles $\hdim E=0$.

For the final assertion of Theorem \ref{thm:projection-non-conc-Schubert-cycle}, apply Theorem \ref{thm:continuous-projection} and Frostman's lemma to normalized positive-mass compact restrictions of the measures in Corollary \ref{cor:small-dim-spread-out}. Lemma \ref{lem:spread-out-subset} preserves spread-outness, and the big-cell parametrization preserves the Hausdorff dimension of each compact support.
\end{proof}

\begin{remark}
As stated in \cite[Remark~1.4]{SpreadOutFurstSticky}, the assumption $\hdim E=\pdim E$ in the planar Furstenberg estimate \cite[Theorem~1.3]{SpreadOutFurstSticky} can be removed. For each $x\in E$, restrict $\nu_x$ to a sufficiently small positive-mass neighborhood in a fixed slope chart. The normalized coordinate images remain $s$-Frostman and, by Lemmas \ref{lem:spread-out-scaling} and \ref{lem:line-decay-spread-out}, are spread-out at every dyadic scale, with constants allowed to depend on $x$. Theorem \ref{thm:mainGeneralContinuous} then applies when $\hdim E>0$; when $\hdim E=0$, the conclusion follows from the Frostman condition on any one $\nu_x$.
\end{remark}

\subsection{Restricted projections for surfaces in \texorpdfstring{$\G(4,2)$}{G(4,2)}}

\label{s:restricted-projections-surfaces}

The results in this section arose from conversations with ChatGPT 5.6 sol pro aimed to understand which surfaces in $\G(d,k)$ can be foliated by Gan--Guo--Wang curves. 

Which curves in $\G(d,k)$ satisfy a Marstrand-type theorem? Gan, Guo and Wang \cite[Theorem 1.2]{GGW22}, building on a series of earlier developments, answered this question for the osculating planes of non-degenerate curves.
\begin{theorem}[Gan--Guo--Wang]\label{thm:GGW}
    Fix $1\le k<d$. Let $\gamma:[0,1]\to\R^d$ be a smooth non-degenerate curve, that is, the vectors $\gamma'(t),\ldots,\gamma^{(d)}(t)$ are linearly independent for all $t\in[0,1]$. Let $V_t\in\G(d,k)$ be the $k$-dimensional subspace spanned by $\gamma'(t),\ldots,\gamma^{(k)}(t)$. Then, for any Borel set $E\subset\R^d$ with $\hdim(E)\le k$,
    \[
        \hdim(\Pi_{V_t}(E))=\min\{\hdim(E),k\}\quad\text{for a.e. }t\in[0,1].
    \]
\end{theorem}
For $k\in\{1,d-1\}$, this family is generic in the space of smooth curves in $\G(d,k)$. For intermediate values of $k$, it is very special. Understanding which curves satisfy a Marstrand-type theorem remains a difficult open problem, even for polynomial curves; see \cite{Du25RestrictedMarstrand, Wang26RestrictedProjections}. Some progress for polynomial curves in $\G(4,2)$ appears in \cite{Du25RestrictedMarstrand}.

For smooth surfaces in $\G(4,2)$, we prove a Marstrand-type theorem by combining two ideas. Elliptic surfaces are spread out. Hyperbolic surfaces satisfying an additional derivative condition are locally foliated by curves to which Theorem \ref{thm:GGW} applies. These two cases cover a generic surface except for a set of zero area.  By comparison, generic $3$-manifolds in the four-dimensional space $\G(4,2)$ can be treated by the classical transversality method \cite{MR1749437}, although this fact does not seem to be recorded in the literature (the special case of isotropic planes in $\G(4,2)$ was established by Hovila \cite[Theorem~1.1]{Hovila14} using transversality). That method breaks down in dimensions less than $3$, or less than $d-1$ in $\G(d,k)$.

A big-cell chart represents a plane in $\G(4,2)$ by a $2\times2$ matrix. In this chart, a smooth surface $\Sigma$ has a local parametrization $F(u,v)$ with linearly independent partial derivatives. At $V=F(u,v)$, its tangent plane is $T_V\Sigma=\operatorname{span}\{\partial_uF,\partial_vF\}\subset \R^{2\times 2}$. Keep the two planes distinct: $V$ is a plane in $\R^4$, while $T_V\Sigma$ consists of the possible first derivatives of curves on $\Sigma$ through $V$. 

The function $(a,b)\mapsto\det(a\partial_uF+b\partial_vF)$ is a quadratic form in two variables. We call $V$ \emph{elliptic} if this form is definite (strictly positive or strictly negative away from the origin), \emph{hyperbolic} if it is indefinite (it takes both signs), and \emph{parabolic} if it is degenerate (a scalar multiple of the square of a linear form, including the zero form). At an elliptic point every nonzero tangent matrix is invertible. At a hyperbolic point the determinant vanishes along exactly two lines in $T_V\Sigma$; these are the \emph{characteristic directions}, and their nonzero matrices have rank one. Changes of big-cell chart act on tangent matrices by invertible left and right multiplication. They multiply the determinant by a nonzero factor, so the classification does not depend on the chart.

\begin{proposition}\label{prop:G42-elliptic-surfaces}
Let $K$ be a compact $C^2$ coordinate patch contained in the elliptic locus of a surface $\Sigma\subset \G(4,2)$. Then $\mathcal H^2|_K$ is spread out. Consequently, for every Borel set $E\subset\R^4$,
\[
    \hdim\Pi_V(E)=\min\{\hdim E,2\}
    \quad\text{for $\mathcal H^2$-a.e. }V\in K.
\]
\end{proposition}

\begin{proof}
All implicit constants may depend on $K$. We use the Euclidean inner product on the matrix entries. A two-plane $T\subset \R^{2\times 2}$ contains a rank-one matrix if and only if $T^\perp$ does. To see this, suppose $T$ contains a rank-one matrix. After orthogonal changes of bases, write $T=\operatorname{span}\{E_{11},B\}$, where $E_{11}$ has entry $1$ in position $(1,1)$ and $0$ elsewhere. Matrices supported on the second row form a two-dimensional space orthogonal to $E_{11}$. Orthogonality to $B$ imposes at most one more linear condition, so this space contains a nonzero matrix in $T^\perp$, necessarily of rank one. Apply the same implication to $T^\perp$ to obtain the converse, since $(T^\perp)^\perp=T$.

Normals to the model hyperplanes in a big-cell chart are rank-one matrices. Write a model hyperplane as $H=f^{-1}(0)$, where $f$ is affine and $|\nabla f|=1$. Since $T_V\Sigma$ contains no rank-one matrix, neither does $(T_V\Sigma)^\perp$. By compactness of $K$ and the unit rank-one matrices, $|\nabla_\Sigma f|\gtrsim1$ uniformly on $K$ and over all model hyperplanes; here $\nabla_\Sigma f$ is the tangential gradient. The uniform $C^2$ bounds imply that, for some $r_0>0$, one partial derivative of $f$ is bounded away from zero on a rotated coordinate rectangle of diameter $O(r)$ containing $K\cap B(V,r)$, for $V\in K$ and $0<r\le r_0$. Each level curve in the rectangle is therefore a graph with uniformly bounded slope. Hence
\[
    \mathcal H^1\bigl(K\cap B(V,r)\cap f^{-1}(t)\bigr)\lesssim r
    \qquad(t\in\R).
\]
Applying the weighted coarea formula \cite[Chapter 2, \S6, (6.6)]{Simon14GMT} to the $C^1$ map $f|_\Sigma$, we obtain
\[
\begin{aligned}
    \mathcal H^2\bigl(K\cap B(V,r)\cap H(\lambda r)\bigr)
    &=\int_{-\lambda r}^{\lambda r}
      \int_{K\cap B(V,r)\cap f^{-1}(t)}
      \frac{1}{|\nabla_\Sigma f|}\,d\mathcal H^1\,dt\\
    &\lesssim\lambda r^2
    \lesssim\lambda\,\mathcal H^2\bigl(K\cap B(V,r/2)\bigr),
\end{aligned}
\]
for $0<\lambda\le1$. For sufficiently small $\lambda\in(0,1/2)$, apply Lemma \ref{lem:hyperplane-non-concentration-implies-spreadout} in a finite cover of $K$ by balls of radius at most $r_0$. Thus $\mathcal H^2|_K$ is spread out. The projection conclusion follows from Corollary \ref{cor:continuous-projection} and the Lipschitz continuity of orthogonal projections.
\end{proof}

Say that a smooth curve $V(t)\in \G(4,2)$ is a \emph{Gan--Guo--Wang curve} if $V(t)=\operatorname{span}\{\gamma'(t),\gamma''(t)\}$ for a smooth curve $\gamma$ in $\R^4$ such that $\gamma'(t),\ldots,\gamma^{(4)}(t)$ are linearly independent. Theorem \ref{thm:GGW} shows that Gan--Guo--Wang curves satisfy a Marstrand-type theorem. 

For a smooth family of planes $V(t)$, its velocity is the linear map $A(t)=\dot V(t):V(t)\to\R^4/V(t)$ defined by $A(t)x(t)=x'(t)\bmod V(t)$ for every smooth choice of vector $x(t)\in V(t)$. Taking the quotient by $V(t)$ records only motion out of the plane: vectors differing by an element of $V(t)$ are identified. This is the identification $T_V\G(4,2)\cong\operatorname{Hom}(V,\R^4/V)$.

Suppose $A(t)$ has rank one. Define the line $K(t)=\ker A(t)$ and the three-plane $W(t)\supset V(t)$ by $W(t)/V(t)=\operatorname{im}A(t)$. The line $K(t)$ consists of vectors whose derivatives stay in $V(t)$, while $W(t)$ includes the one additional direction in which $V(t)$ moves. Differentiation induces maps
\begin{equation}\label{eq:G42-flag-derivatives}
    K(t)\longrightarrow V(t)/K(t),
    \qquad
    W(t)/V(t)\longrightarrow\R^4/W(t).
\end{equation}
The first map sends $x(t)$ to $x'(t)\bmod K(t)$ for a smooth choice $x(t)\in K(t)$; the second sends $y(t)+V(t)$ to $y'(t)\bmod W(t)$ for a smooth choice $y(t)\in W(t)$. These are well-defined because $x'(t)\in V(t)$ and the derivative of any smooth choice in $V(t)$ belongs to $W(t)$. Nonvanishing of the two maps means, respectively, that the line $K(t)$ and the three-plane $W(t)$ change to first order.

On a hyperbolic surface, the two characteristic directions vary smoothly. Choose a smooth nonzero vector in either direction locally and solve the resulting ordinary differential equation. Its integral curves form a local foliation.

\begin{proposition}\label{prop:G42-hyperbolic-surfaces}
Let $U$ be a hyperbolic open subset of a smooth surface in $\G(4,2)$. Suppose that, for one characteristic foliation, both maps in \eqref{eq:G42-flag-derivatives} are nonzero on $U$. Then its leaves are locally Gan--Guo--Wang curves. Consequently, for every Borel set $E\subset\R^4$,
\[
    \hdim\Pi_V(E)=\min\{\hdim E,2\}
    \quad\text{for $\mathcal H^2$-a.e. }V\in U.
\]
\end{proposition}

\begin{proof}
Let $V(t)$ be a leaf of the specified foliation, and choose a smooth nonzero vector $x(t)\in K(t)$. Since the first map in \eqref{eq:G42-flag-derivatives} is nonzero, $x(t),x'(t)$ span $V(t)$. The identity $x''(t)\bmod V(t)=A(t)x'(t)$ and $K(t)=\ker A(t)$ show that $x''(t)$ modulo $V(t)$ spans $\operatorname{im}A(t)$. Hence $x(t),x'(t),x''(t)$ span $W(t)$. Nonvanishing of the second map says that $x'''(t)\notin W(t)$. Locally define $\gamma$ by $\gamma'=x$. Then $V(t)=\operatorname{span}\{\gamma'(t),\gamma''(t)\}$ and $\gamma'(t),\ldots,\gamma^{(4)}(t)$ are linearly independent.

Cover $U$ by countably many surface coordinate patches in which the leaves are obtained by fixing one coordinate. Theorem \ref{thm:GGW} applies on every leaf segment. Fubini's theorem in these coordinates proves the asserted conclusion for $\mathcal H^2$-almost every $V\in U$.
\end{proof}

Hyperbolicity alone is insufficient in Proposition \ref{prop:G42-hyperbolic-surfaces}: the surface $(u,v)\mapsto\operatorname{diag}(u,v)$ is hyperbolic, but the kernel and image lines of the velocity are constant along both characteristic foliations, so both maps in \eqref{eq:G42-flag-derivatives} vanish. Since these maps are nonzero for every Gan--Guo--Wang curve, neither foliation consists of Gan--Guo--Wang curves.

Propositions \ref{prop:G42-elliptic-surfaces} and \ref{prop:G42-hyperbolic-surfaces} cover a generic surface in $\G(4,2)$ up to a set of zero area. Here is the reason. Exclude parabolic points and hyperbolic points where one map in \eqref{eq:G42-flag-derivatives} vanishes along either characteristic direction. At every other point, Proposition \ref{prop:G42-elliptic-surfaces} applies locally, or Proposition \ref{prop:G42-hyperbolic-surfaces} applies locally to both characteristic foliations. Each excluded condition imposes at least one equation on the first two derivatives. After a small perturbation, these equations hold only on curves or isolated points.

\begin{theorem}\label{thm:G42-generic-surfaces}
Let $M$ be a compact $C^\infty$ surface, and equip $\operatorname{Emb}^\infty(M,\G(4,2))$ with the compact-open $C^\infty$ topology. There is an open dense subset $\mathcal U\subset\operatorname{Emb}^\infty(M,\G(4,2))$ such that, for every $F\in\mathcal U$, there is a closed set $Z\subset M$ with $\mathcal H^2(F(Z))=0$ such that every $p\in M\setminus Z$ has a neighborhood $U$ satisfying one of the following:
\begin{enumerate}[label=\textup{(\roman*)}]
    \item $F(U)$ is elliptic;
    \item $F(U)$ is hyperbolic, and each characteristic foliation of $F(U)$ consists of Gan--Guo--Wang curves.
\end{enumerate}
Consequently, for every Borel set $E\subset\R^4$,
\[
    \hdim\Pi_{F(x)}(E)=\min\{\hdim E,2\}
    \quad\text{for $\mathcal H^2$-a.e. }x\in M.
\]
\end{theorem}

\begin{proof}
The exceptional conditions depend only on the first two derivatives. We record these derivatives, together with $p$ and $F(p)$, in the \emph{two-jet} $j^2F(p)$. In local coordinates, these data form a finite-dimensional space, denoted $J^2_{\mathrm{imm}}(M,\G(4,2))$ when the first derivative is injective. The map $j^2F$ assigns this data to each $p\in M$.

The tangent planes $T\subset \R^{2\times 2}$ form a four-dimensional space. A plane is parabolic precisely when the discriminant of $\det|_T$ vanishes. The parabolic locus is therefore a proper algebraic subset, of positive codimension.

Fix a hyperbolic two-jet at $p$, represented by a local parametrization $F$, and one of its characteristic directions. Choose coordinates $(x,y)$ on $M$ and bases in the domain and range of the tangent matrices such that $\partial_xF=E_{11}$ and $\partial_yF=E_{22}$ at $p$, with the selected characteristic direction equal to $\partial_x$. In big-cell coordinates centered at $F(p)$, let $V(t)$ be the characteristic leaf through $F(p)$ and write $B(t)=\dot V(t)$, with $V(0)=F(p)$ and $B(0)=E_{11}$. By the chain rule, $\dot B(0)-\partial_{xx}F\in\operatorname{span}\{E_{11},E_{22}\}$. Differentiating the rank-one condition for $B(t)$ at $t=0$, we find
\begin{equation}\label{eq:G42-flag-jet}
    \dot B(0)=
    \begin{pmatrix}
        *&(\partial_{xx}F)_{12}\\
        (\partial_{xx}F)_{21}&0
    \end{pmatrix}.
\end{equation}
Differentiating the kernel and image of $B(t)$ shows that the first map in \eqref{eq:G42-flag-derivatives} is nonzero exactly when $(\partial_{xx}F)_{12}\ne0$, and the second is nonzero exactly when $(\partial_{xx}F)_{21}\ne0$. The other characteristic direction is treated by replacing $\partial_{xx}F$ with $\partial_{yy}F$.

Define $\mathcal B\subset J^2_{\mathrm{imm}}(M,\G(4,2))$ to consist of the parabolic jets and the hyperbolic jets for which at least one map in \eqref{eq:G42-flag-derivatives} vanishes along at least one characteristic direction. Locally on the hyperbolic locus, the two characteristic lines can be chosen as smooth functions of the jet. For either choice, we see from \eqref{eq:G42-flag-jet} that the vanishing of either map is one equation with nonzero differential: the relevant off-diagonal second derivative can be varied while fixing $F(p)$ and its first derivatives.

The set $\mathcal B$ is closed. Indeed, a limit of hyperbolic jets in $\mathcal B$ is either parabolic or hyperbolic with a vanishing map in \eqref{eq:G42-flag-derivatives}, by passing to a convergent subsequence of characteristic directions.

Several exceptional conditions can meet, so $\mathcal B$ need not be smooth. We partition it into smooth pieces using a standard fact about \emph{semialgebraic sets}, which are finite unions of sets defined by polynomial equations and inequalities. Such sets have finite \emph{Whitney stratifications}; see \cite[\S9.7, pp.~236--243]{BochnakCosteRoy98}. The smooth pieces, called \emph{strata}, satisfy compatibility conditions at their boundaries. One condition says that if points of one stratum converge to a point of another and their tangent spaces converge, the limiting tangent space contains the tangent space of the stratum at the limit point. This compatibility ensures that transversality to all strata persists under perturbation.

In local two-jet coordinates, $\mathcal B$ is semialgebraic. The parabolic condition is polynomial. The characteristic directions and the vanishing conditions in \eqref{eq:G42-flag-derivatives} are also described by polynomial equations and inequalities if we introduce unit vectors for the characteristic, kernel, and image lines. Eliminating these vectors preserves semialgebraicity. Finite coordinate atlases therefore allow a finite Whitney stratification of $\mathcal B$. Every stratum has positive codimension, since each parabolic or hyperbolic exceptional condition does.

Transversality has a concrete consequence here. If $j^2F$ is transverse to a smooth stratum $S$, then the independent local equations for $S$, evaluated at $j^2F(p)$, have independent gradients in the surface coordinates at every solution $p$. By the implicit function theorem, $(j^2F)^{-1}(S)$ has dimension $2-\operatorname{codim}S$, and is empty if $\operatorname{codim}S>2$. The codimension in jet space thus bounds the dimension of the exceptional set on the surface.

Let $\mathcal U$ consist of the embeddings $F$ for which $j^2F$ is transverse to every stratum of $\mathcal B$. The jet transversality theorem \cite[Theorem~4.5.6, pp.~107--108]{Wall16} achieves transversality to the finitely many strata by arbitrarily small smooth perturbations. Embeddings are open in the compact-open $C^1$ topology, so $\mathcal U$ is dense. Since $\mathcal B$ is closed, its stratification is Whitney, and $M$ is compact, transversality to all strata persists under small $C^1$ perturbations of $j^2F$ \cite[Proposition~1.2(b) and proof]{Trivedi13}. Small $C^3$ perturbations of $F$ induce such perturbations of $j^2F$. Thus $\mathcal U$ is open in the compact-open $C^3$ topology, and hence in the compact-open $C^\infty$ topology.

Fix $F\in\mathcal U$ and set $Z=(j^2F)^{-1}(\mathcal B)$. Then $Z$ is closed and is a finite union of smooth pieces of dimension at most one. Since $F$ is Lipschitz on $M$, $\mathcal H^2(F(Z))=0$. If $p\in M\setminus Z$ is elliptic, ellipticity holds in a neighborhood $U$ of $p$. If $p$ is hyperbolic, both maps in \eqref{eq:G42-flag-derivatives} are nonzero along each characteristic direction throughout a neighborhood $U$ of $p$. Proposition \ref{prop:G42-hyperbolic-surfaces} therefore applies to both characteristic foliations of $F(U)$. Cover $M\setminus Z$ by countably many coordinate patches satisfying one of these two conditions. Propositions \ref{prop:G42-elliptic-surfaces} and \ref{prop:G42-hyperbolic-surfaces}, together with $\mathcal H^2(F(Z))=0$, establish the projection conclusion.
\end{proof}

\begin{remark}
Propositions \ref{prop:G42-elliptic-surfaces} and \ref{prop:G42-hyperbolic-surfaces} have an interpretation in terms of the pencils $\mathcal P_W$ from \cite[Definition~1.4]{Wang26RestrictedProjections}. For a fixed line $K$ inside a fixed three-plane $H$, the family $\{V':K\subset V'\subset H\}$ is called a \emph{Grassmannian line}: choosing $V'$ amounts to choosing a line in the two-dimensional quotient $H/K$. If $A:V\to\R^4/V$ has rank one, set $K=\ker A$ and define $H\supset V$ by $H/V=\operatorname{im}A$. Then $A$ is tangent to both $\mathcal P_{K^\perp}=\{V':K\subset V'\}$ and $\mathcal P_{H^\perp}=\{V':V'\subset H\}$, and to their intersection, the Grassmannian line just defined. Conversely, every nonzero tangent direction to a Grassmannian line has rank one.

Thus an elliptic tangent plane contains no direction tangent to a Grassmannian line. Equivalently, the surface is transverse to every pencil $\mathcal P_W$ through $V$ with $\dim W\in\{1,3\}$: the tangent spaces of the surface and the pencil together span $T_V\G(4,2)$. The surface is also transverse to every pencil $\mathcal P_W$ with $\dim W=2$ at its smooth points. These transversality statements express geometrically the nonvanishing tangential derivatives underlying Proposition \ref{prop:G42-elliptic-surfaces}.

A hyperbolic tangent plane instead contains exactly two directions tangent to Grassmannian lines. Along a characteristic leaf, $K(t)$ and $W(t)$ determine the pencils $\mathcal P_{K(t)^\perp}$ and $\mathcal P_{W(t)^\perp}$. The two maps in \eqref{eq:G42-flag-derivatives} test whether $K(t)$ and $W(t)$, respectively, change to first order. If $K(t)$ is locally constant, the leaf is contained in a fixed pencil $\mathcal P_{K(t)^\perp}$; if $W(t)$ is locally constant, it is contained in a fixed pencil $\mathcal P_{W(t)^\perp}$. Hence Theorem \ref{thm:G42-generic-surfaces} says that, outside an $\mathcal H^2$-null set, a generic surface either has no tangent pencil direction or has two such directions along whose characteristic leaves both associated pencils change to first order.
\end{remark}


\subsection{Distance sets with spread-out pins}
\label{S:distance-sets}

One form of the Falconer distance set problem asks whether $\hdim E\ge d/2$ forces the distance set $\{|x-y|:x,y\in E\}$ to have Hausdorff dimension $1$. The threshold $d/2$ would be sharp. The problem remains open in every dimension $d \ge 2$ \footnote{Well after a full draft of this paper was completed, OpenAI posted a file claiming that a closely related problem has been solved. The claim, even if accepted by the mathematical community, would not recover our estimates in dimension $d\ge 3$.}, but the conclusion and related statements are known for special classes of sets; see \cite{FiedlerStull24PinnedDistances, ShmerkinWang25, Liu26DistanceSets} and the references there.

A well-known lifting to the paraboloid converts pinned distances into linear projections (see e.g. \cite[\S 4.1]{Pham26DistanceSets}). Corollary \ref{cor:continuous-projection} therefore has the following consequence.
\begin{corollary}
\label{cor:distance-sets}
 Let $\mu$ be a spread-out measure on $\R^d$. For every Borel set $E\subset\R^d$,
 \[
      \hdim\{ |x-y| : x\in E \} = \min\{\hdim E,1\} \quad \text{for $\mu$-almost every }y\in\R^d.
 \]
\end{corollary}
\begin{proof}
Let $\widehat{E}=\{(|x|^2,x):x\in E\}$, so $\hdim\widehat{E}=\hdim E$. In the projection coordinates of Section \ref{s:bigcell},
\[
    \pi_{-2y}(|x|^2,x)=|x|^2-2y\cdot x=|x-y|^2-|y|^2.
\]
By Lemma \ref{lem:spread-out-scaling}, the push-forward of $\mu/\mu(\R^d)$ under $y\mapsto-2y$ is a spread-out probability measure on $\mathbf G_{d+1,1}$. Corollary \ref{cor:continuous-projection} now yields
\[
    \hdim\{|x-y|^2:x\in E\}=\min\{\hdim E,1\}
    \quad\text{for $\mu$-almost every }y\in\R^d.
\]
Since $t\mapsto t^2$ is locally bi-Lipschitz on $(0,\infty)$, squaring the distances preserves their Hausdorff dimension.
\end{proof}
In particular, if $\hdim E\ge1$, then $\hdim\{|x-y|:x\in E\}=1$ for $\mu$-almost every $y$, and hence for some $y\in\spt(\mu)$. This is the threshold $d/2$ in the plane and is smaller in higher dimensions. As for linear projections, sets of pins with arbitrarily small Hausdorff dimension can therefore satisfy a Marstrand-type theorem for pinned distances.
For comparison, Fiedler and Stull \cite[Theorem~6]{FiedlerStull24PinnedDistances} proved that every compact planar Ahlfors-regular set of dimension greater than $1$ contains, for each Borel set $E\subset\R^2$, a pin whose distance set has dimension $\min\{\hdim E,1\}$.


\subsection{Discretized sum-product}
\label{s:discretized-sum-product}

Theorem \ref{thm:furst-general} extends the discretized sum-product theorem of \cite[Theorem~6.5]{SpreadOutFurstSticky} beyond Ahlfors-regular sets. For background on the problem, see \cite[\S 6.2]{SpreadOutFurstSticky}.
\begin{theorem} \label{thm:discretized-SP}
    For $\beta\in(0,1]$ and $\eta>0$, there is $\e=\e(\beta,\eta)>0$ such that the following holds for every dyadic $\delta\ll_{\beta,\eta}1$.
    
    Let $A\subset[1,2]$ be a nonempty $(\delta,1/2,\delta^{-\e})$-Katz--Tao set which is $(\delta,\beta,\delta^{-\e})_{2,1}$-spread-out. Then,
    \begin{equation} \label{eq:discretized-SP-strong}
        |A+A|_{\delta}\cdot |A\cdot A|_{\delta} \ge \delta^{\eta}\cdot |A|_{\delta}^3.
    \end{equation}
    In particular, 
    \[
        \max\{|A+A|_{\delta}, |A\cdot A|_{\delta}\} \ge \delta^{\eta} \cdot |A|_{\delta}^{3/2}.
    \]
\end{theorem}
\begin{proof}
Use the Elekes construction \cite{Elekes97} in dual coordinates from \cite[Theorem~6.5]{SpreadOutFurstSticky}. The point set $P$ is the $\delta$-discretization of a bi-Lipschitz image of $A\times A$, so $|P|\sim|A|_\delta^2$ and $P$ is $(\delta,1,\delta^{-O(\e)})$-Katz--Tao. By \cite[Lemmas~2.10--2.11]{SpreadOutFurstSticky} and Lemmas  \ref{lem:spread-out-subset}--\ref{lem:spread-out-scaling}, discretizing the translates of $-A/4$ gives $(\delta,\beta/16,\delta^{-O(\e)})$-spread-out shadings with common cardinality $M\gtrapprox_\delta|A|_\delta$. As in that construction,
\[
    |A+A|_\delta\,|A\cdot A|_\delta
    \gtrsim |\mathcal F(P,\mathbf\Theta)|_\delta
    \ge\delta^{\eta/2}|P|M
    \gtrapprox_\delta\delta^{\eta/2}|A|_\delta^3,
\]
where Theorem \ref{thm:furst-general} applies with shading exponent $\beta/16$, $t=1$, and loss $\eta/2$ once $\e\ll_{\beta,\eta}1$. For sufficiently small $\delta$, this proves \eqref{eq:discretized-SP-strong}.
\end{proof}

\begin{remark}
Theorem \ref{thm:discretized-SP} recovers \cite[Theorem~6.5]{SpreadOutFurstSticky} for $s$-Ahlfors regular sets with $s\le1/2$. After decreasing $\e$, \eqref{eq:discretized-SP-strong} also holds for $\delta^\e$-dense subsets of $A$, since they preserve the spread-outness and Katz--Tao conditions.
\end{remark}

A set $A\subset\R$ with $\diam A>0$ is \emph{uniformly perfect} if there exists $D>1$ such that
\begin{displaymath}
    A\cap\bigl(B(a,Dr)\setminus B(a,r)\bigr)\ne\varnothing
    \quad\text{whenever }a\in A,\ r>0,\ A\not\subset B(a,Dr).
\end{displaymath}
For compact uniformly perfect sets, the projection theorem also has a continuous sum-product corollary; this extends the Ahlfors-regular case considered in \cite[Corollary~6.8]{SpreadOutFurstSticky}.
\begin{corollary}
    Let $A\subset[1,2]$ be a compact uniformly perfect set with Hausdorff dimension $s\in(0,1)$. Then $\hdim(A+A\cdot A)\ge \min\{2\hdim(A),1\}$.
\end{corollary}
\begin{proof}
By \cite[\S 5]{MR4163999}, the set $A$ supports a uniformly perfect probability measure. Its push-forward under $a\mapsto\operatorname{span}(1,a)$ is uniformly perfect, since this map is bi-Lipschitz on $[1,2]$. Applying Theorem \ref{thm:projection-ADR} to $A\times A$, we find $a\in A$ such that
\begin{displaymath}
    \hdim(A+aA)=\min\{\hdim(A\times A),1\}\ge\min\{2s,1\}.
\end{displaymath}
Since $A+aA\subset A+A\cdot A$, the conclusion follows. 
\end{proof}

\subsection{Dimension of planar self-affine sets and measures}
\label{s:self-affine}

Let $T_1,\ldots,T_m:\R^d\to\R^d$ be contracting affine maps. The associated self-affine set is the unique nonempty compact set $E\subset\R^d$ satisfying $E=\bigcup_{i=1}^m T_i(E)$. Probability weights $p_1,\ldots,p_m$ determine a unique Borel probability measure $\mu$ supported on $E$ satisfying $\mu=\sum_{i=1}^m p_i\,T_i\mu$.
A central question is whether the dimensions of self-affine sets and measures can be computed from the affine maps. Falconer's \emph{affinity dimension} \cite{Falconer88SelfAffine} is a natural upper bound for both Hausdorff and Minkowski dimension. To define it, let $A_i$ be the linear part of $T_i$, and let $\alpha_1(A)\ge \alpha_2(A)\ge\ldots\ge \alpha_d(A)$ be the singular values of $A:\R^d\to\R^d$. For $s\in [0,d]$, the \emph{singular value function} $\phi^s(A)$ is
\[
    \phi^s(A)=\alpha_1(A)\cdots\alpha_{\lfloor s\rfloor}(A)\alpha_{\lfloor s\rfloor+1}(A)^{s-\lfloor s\rfloor},\qquad 0\le s<d.
\]
At $s=d$, we set $\phi^d(A)=|\det A|$. When $s>d$, we let $\phi^s(A)=|\det(A)|^{s/d}$. The affinity dimension of the self-affine set $E$ is defined as
\[
    \dim_{\mathrm{aff}}(E)=\inf\Bigl\{s>0:\sum_{n=1}^{\infty}\sum_{i_1,\ldots,i_n=1}^m\phi^s(A_{i_1}\cdots A_{i_n})<\infty\Bigr\},   
\]
Strictly speaking, the affinity dimension depends on the iterated function system, since different systems can have the same attractor $E$. Falconer \cite{Falconer88SelfAffine} proved that
\[
\hdim(E)\le \min\{\dim_{\mathrm{aff}}(E),d\}
\]
for every self-affine set, with equality for almost every choice of translation vectors under certain assumptions. B\'ar\'any, Hochman and Rapaport \cite{BHR} proved equality for all planar self-affine sets satisfying the strong open set condition and mild irreducibility and non-conformality assumptions on the linear parts; see \cite[Theorem 1.1]{BHR}. The main step is a projection theorem. Under an additional positivity assumption, a self-affine measure $\mu$ and its Furstenberg measure $\nu$ on the projective line, induced by the linear parts, satisfy
\begin{equation} \label{eq:proj-thm-furst-measure}
\hdim(\Pi_V\mu)=\min\{\hdim\mu,1\} \quad \text{for $\nu$-almost every $V\in\mathbb{RP}^1$}.
\end{equation}
Combining \eqref{eq:proj-thm-furst-measure} with the Ledrappier--Young formula of Käenmäki and Bárány \cite{BaranyKaenmaki17} and a reduction of Morris and the second author \cite{MorrisShmerkin19} proves equality of Hausdorff and affinity dimension. See \cite{BHR} for the history and details.

Spread-outness provides another route to \eqref{eq:proj-thm-furst-measure}. Under the positivity assumption, the linear parts contract uniformly on the projective line, so the Furstenberg measure is self-conformal. Irreducibility ensures that it is not supported on a single point. By Corollary \ref{cor:small-dim-spread-out}, $\nu$ is spread out. Theorem \ref{thm:continuous-projection} now yields \eqref{eq:proj-thm-furst-measure} for the self-affine measure $\mu$. Applying the positive subsystem approximation of \cite[Theorem~1.1]{MorrisShmerkin19}, together with the Ledrappier--Young formula, we recover \cite[Theorem 1.1]{BHR}.

Hochman and Rapaport \cite{HochmanRapaport22} extended the results of B\'ar\'any, Hochman and Rapaport to overlapping systems, and Rapaport \cite{Rapaport24SelfAffine} extended them to higher dimensions. The general problem remains open for $d\ge 4$. Our method has an obstacle in higher dimensions: the Furstenberg measures are invariant under \emph{non-conformal} dynamics and can concentrate locally near hyperplanes, so spread-outness can fail.


\begin{thebibliography}{10}

\bibitem{Baker26}
Simon Baker.
\newblock Disintegration results for fractal measures and applications to {D}iophantine approximation.
\newblock {\em Ergodic Theory Dynam. Systems}, 46(4):885--902, 2026.

\bibitem{BHR}
Bal\'{a}zs B\'{a}r\'{a}ny, Michael Hochman, and Ariel Rapaport.
\newblock Hausdorff dimension of planar self-affine sets and measures.
\newblock {\em Invent. Math.}, 216(3):601--659, 2019.

\bibitem{BaranyKaenmaki17}
Bal\'{a}zs B\'{a}r\'{a}ny and Antti K{\"a}enm{\"a}ki.
\newblock Ledrappier--{Y}oung formula and exact dimensionality of self-affine measures.
\newblock {\em Adv. Math.}, 318:88--129, 2017.

\bibitem{BochnakCosteRoy98}
Jacek Bochnak, Michel Coste, and Marie-Fran\c{c}oise Roy.
\newblock {\em Real algebraic geometry}, volume~36 of {\em Ergebnisse der Mathematik und ihrer Grenzgebiete. 3. Folge}.
\newblock Springer-Verlag, Berlin, 1998.

\bibitem{Bo2}
Jean Bourgain.
\newblock The discretized sum-product and projection theorems.
\newblock {\em J. Anal. Math.}, 112:193--236, 2010.

\bibitem{Chen18RestrictedProjections}
Changhao Chen.
\newblock Restricted families of projections and random subspaces.
\newblock {\em Real Anal. Exchange}, 43(2):347--358, 2018.

\bibitem{CsornyeiStull26Hyperplanes}
Marianna Cs{\"o}rnyei and D.~M. Stull.
\newblock Hyperplane incidences and distance sets in higher dimensions.
\newblock {\em Preprint}, 2026.
\newblock arXiv:2609.38742v1.

\bibitem{DemeterWang25}
Ciprian Demeter and Hong Wang.
\newblock Szemer\'edi-{T}rotter bounds for tubes and applications.
\newblock {\em Ars Inven. Anal.}, pages Paper No. 1, 46, 2025.

\bibitem{Du25RestrictedMarstrand}
Jiahan Du.
\newblock Restricted {M}arstrand's projection theorem for general families of linear subspaces.
\newblock {\em Preprint}, 2025.
\newblock arXiv:2510.16671v2.

\bibitem{Elekes97}
Gy\"{o}rgy Elekes.
\newblock On the number of sums and products.
\newblock {\em Acta Arith.}, 81(4):365--367, 1997.

\bibitem{Falconer88SelfAffine}
Kenneth~J. Falconer.
\newblock The {H}ausdorff dimension of self-affine fractals.
\newblock {\em Math. Proc. Cambridge Philos. Soc.}, 103(2):339--350, 1988.

\bibitem{FasslerOrponen14}
Katrin F\"{a}ssler and Tuomas Orponen.
\newblock On restricted families of projections in $\mathbb{R}^3$.
\newblock {\em Proc. Lond. Math. Soc. (3)}, 109(2):353--381, 2014.

\bibitem{FiedlerStull24PinnedDistances}
Jacob~B. Fiedler and D.~M. Stull.
\newblock Pinned distances of planar sets with low dimension.
\newblock {\em Trans. Amer. Math. Soc. (to appear)}, 2024.
\newblock arXiv:2408.00889v2.

\bibitem{GGGHMW22}
Shengwen Gan, Shaoming Guo, Larry Guth, Terence L.~J. Harris, Dominique Maldague, and Hong Wang.
\newblock On restricted projections to planes in $\mathbb{R}^3$.
\newblock {\em Amer. J. Math.}, 148(4):915--960, 2026.

\bibitem{GGW22}
Shengwen Gan, Shaoming Guo, and Hong Wang.
\newblock A restricted projection problem for fractal sets in $\mathbb{R}^n$.
\newblock {\em Camb. J. Math.}, 12(3):535--561, 2024.

\bibitem{GanGuthMaldague24}
Shengwen Gan, Larry Guth, and Dominique Maldague.
\newblock An exceptional set estimate for restricted projections to lines in $\mathbb{R}^3$.
\newblock {\em J. Geom. Anal.}, 34(1):Paper No. 15, 13, 2024.

\bibitem{GuthWangZahl26}
Larry Guth, Hong Wang, and Joshua Zahl.
\newblock A streamlined proof of the {K}akeya set conjecture in $\mathbb{R}^3$.
\newblock {\em Preprint}, 2026.
\newblock arXiv:2601.14411v1.

\bibitem{He20}
Weikun He.
\newblock Orthogonal projections of discretized sets.
\newblock {\em J. Fractal Geom.}, 7(3):271--317, 2020.

\bibitem{HochmanRapaport22}
Michael Hochman and Ariel Rapaport.
\newblock {H}ausdorff dimension of planar self-affine sets and measures with overlaps.
\newblock {\em J. Eur. Math. Soc. (JEMS)}, 24(7):2361--2441, 2022.

\bibitem{Hovila14}
Risto Hovila.
\newblock Transversality of isotropic projections, unrectifiability, and {H}eisenberg groups.
\newblock {\em Rev. Mat. Iberoam.}, 30(2):463--476, 2014.

\bibitem{KaenmakiOrponenVenieri25}
Antti K\"{a}enm\"{a}ki, Tuomas Orponen, and Laura Venieri.
\newblock A {M}arstrand-type restricted projection theorem in $\mathbb{R}^3$.
\newblock {\em Amer. J. Math.}, 147(1):81--123, 2025.

\bibitem{MR407248}
R.~Kaufman and P.~Mattila.
\newblock Hausdorff dimension and exceptional sets of linear transformations.
\newblock {\em Ann. Acad. Sci. Fenn. Ser. A I Math.}, 1(2):387--392, 1975.

\bibitem{KLW04}
Dmitry Kleinbock, Elon Lindenstrauss, and Barak Weiss.
\newblock On fractal measures and {D}iophantine approximation.
\newblock {\em Selecta Math. (N.S.)}, 10(4):479--523, 2004.

\bibitem{Liu26DistanceSets}
Bochen Liu.
\newblock {L}ebesgue measure of distance sets with regular pins and multi-scale {M}izohata--{T}akeuchi-type estimates.
\newblock {\em Preprint}, 2026.
\newblock arXiv:2603.15328v3.

\bibitem{Mar}
John~M. Marstrand.
\newblock Some fundamental geometrical properties of plane sets of fractional dimensions.
\newblock {\em Proc. London Math. Soc. (3)}, 4:257--302, 1954.

\bibitem{MR409774}
Pertti Mattila.
\newblock Hausdorff dimension, orthogonal projections and intersections with planes.
\newblock {\em Ann. Acad. Sci. Fenn. Ser. A I Math.}, 1(2):227--244, 1975.

\bibitem{zbMATH01249699}
Pertti {Mattila}.
\newblock {\em {Geometry of sets and measures in Euclidean spaces. Fractals and rectifiability. 1st paperback ed.}}
\newblock Cambridge: Cambridge University Press, 1st paperback ed. edition, 1999.

\bibitem{MorrisShmerkin19}
Ian~D. Morris and Pablo Shmerkin.
\newblock On equality of {H}ausdorff and affinity dimensions, via self-affine measures on positive subsystems.
\newblock {\em Trans. Amer. Math. Soc.}, 371(3):1547--1582, 2019.

\bibitem{OS23}
Tuomas Orponen and Pablo Shmerkin.
\newblock On the {H}ausdorff dimension of {F}urstenberg sets and orthogonal projections in the plane.
\newblock {\em Duke Math. J.}, 172(18):3559--3632, 2023.

\bibitem{SpreadOutFurstSticky}
Tuomas Orponen and Pablo Shmerkin.
\newblock Spread-out {Furstenberg} sets {I}: the sticky case.
\newblock arXiv preprint, 2026.

\bibitem{MR1749437}
Yuval Peres and Wilhelm Schlag.
\newblock Smoothness of projections, {B}ernoulli convolutions, and the dimension of exceptions.
\newblock {\em Duke Math. J.}, 102(2):193--251, 2000.

\bibitem{Pham26DistanceSets}
Minh-Quy Pham.
\newblock On {F}alconer type functions and the distance set problem.
\newblock {\em Preprint}, 2025.
\newblock arXiv:2510.15118v2.

\bibitem{PYZ22}
Malabika Pramanik, Tongou Yang, and Joshua Zahl.
\newblock A {F}urstenberg-type problem for circles, and a {K}aufman-type restricted projection theorem in $\mathbb{R}^3$.
\newblock {\em Amer. J. Math. (to appear)}, 2022.
\newblock arXiv:2207.02259v3.

\bibitem{Rapaport24SelfAffine}
Ariel Rapaport.
\newblock On self-affine measures associated to strongly irreducible and proximal systems.
\newblock {\em Adv. Math.}, 449:Paper No. 109734, 116, 2024.

\bibitem{23arXiv230808819R}
Kevin Ren and Hong Wang.
\newblock {F}urstenberg sets estimate in the plane.
\newblock {\em Preprint}, 2023.
\newblock arXiv:2308.08819v3.

\bibitem{MR4163999}
Eino Rossi and Pablo Shmerkin.
\newblock On measures that improve {$L^q$} dimension under convolution.
\newblock {\em Rev. Mat. Iberoam.}, 36(7):2217--2236, 2020.

\bibitem{ShmerkinWang25}
Pablo Shmerkin and Hong Wang.
\newblock On the distance sets spanned by sets of dimension {$d/2$} in {$\mathbb{R}^d$}.
\newblock {\em Geom. Funct. Anal.}, 35(1):283--358, 2025.

\bibitem{Simon14GMT}
Leon Simon.
\newblock Introduction to geometric measure theory.
\newblock Lecture notes, Stanford University, 2014.

\bibitem{MR2289012}
Terence Tao and Van Vu.
\newblock {\em Additive combinatorics}, volume 105 of {\em Cambridge Studies in Advanced Mathematics}.
\newblock Cambridge University Press, Cambridge, 2006.

\bibitem{Trivedi13}
Saurabh Trivedi.
\newblock Transversality theorems for the weak topology.
\newblock {\em Proc. Amer. Math. Soc.}, 141(6):2181--2192, 2013.

\bibitem{Wall16}
C.~T.~C. Wall.
\newblock {\em Differential topology}, volume 156 of {\em Cambridge Studies in Advanced Mathematics}.
\newblock Cambridge University Press, Cambridge, 2016.

\bibitem{Wang26RestrictedProjections}
Hong Wang.
\newblock Restricted families of orthogonal projections in $\mathbb{R}^n$.
\newblock In {\em Proceedings of the {I}nternational {C}ongress of {M}athematicians 2026. Volume 4: Invited lectures (Sections 5--8)}, pages 785--800. Society for Industrial and Applied Mathematics, Philadelphia, PA, 2026.

\bibitem{WangWu24Restriction}
Hong Wang and Shukun Wu.
\newblock Restriction estimates using decoupling theorems and two-ends {F}urstenberg inequalities.
\newblock {\em Preprint}, 2024.
\newblock arXiv:2411.08871v3.

\bibitem{WangWu26TwoEnds}
Hong Wang and Shukun Wu.
\newblock Two-ends {F}urstenberg estimates in the plane.
\newblock {\em Math. Ann.}, 395(4):91, 2026.

\bibitem{WangZahl25Kakeya}
Hong Wang and Joshua Zahl.
\newblock Volume estimates for unions of convex sets, and the {K}akeya set conjecture in three dimensions.
\newblock {\em Preprint}, 2025.
\newblock arXiv:2502.17655v1.

\bibitem{Wolff99}
Thomas Wolff.
\newblock Recent work connected with the {K}akeya problem.
\newblock In {\em Prospects in mathematics ({P}rinceton, {NJ}, 1996)}, pages 129--162. Amer. Math. Soc., Providence, RI, 1999.

\end{thebibliography}

\def\cprime{$'$}

\end{document}